\documentclass{article}
\usepackage{amsmath,amsthm,amssymb,bbm}

\newcommand{\assign}{:=}
\newcommand{\tmop}[1]{\ensuremath{\operatorname{#1}}}
\newcommand{\tmstrong}[1]{\textbf{#1}}
\numberwithin{equation}{section}  
\newtheorem{corollary}{Corollary}
\newtheorem{definition}{Definition}
\newtheorem{proposition}{Proposition}
\newtheorem{theorem}{Theorem}
\newtheorem{lemma}{Lemma}

\newcommand{\XXint}[3]{{\setbox}0=\text{\ensuremath{#1 #2 #3 \int}}
{\vcenter{\text{\ensuremath{#2 #3}}}}{\kern}-.5{\tmwd}0}

\newcommand{\opn}[2]{\newcommand{\1}{\}} {\opn}{\Rm{Rm}} {\opn}{\Ric{Ric}}
{\opn}{\Rc{Rc}} {\opn}{\Scal{Sc}} {\opn}{\Tr{Tr}} {\opn}{\Trac{Tr}}
{\opn}detdet {\opn}{\diam{diam}} {\opn}{\dist{dist}} {\opn}{\Im}Im
{\opn}{\div}div {\opn}{\Ker{Ker}} {\opn}expexp {\opn}{\Vol{Vol}}
{\opn}{\exph{exph}} {\opn}{\Herm{Herm}} {\opn}{\End{End}} {\opn}{\Hess{Hess}}
{\opn}{\Vol{Vol}}}

\newcommand{\contract}{{\kern}-1.5pt{\vrule} width6.0pt height0.4pt depth0pt
{\vrule} width0.4pt height4.0pt depth0pt}
\newcommand{\retract}{{\kern}-1.5pt{\vrule} width0.4pt height4.0pt depth0pt
{\vrule} width6.0pt height0.4pt depth0pt}
\newcommand{\Openbox}{{\leavevmode} {\text{{\hfil}{\vrule}
width{\boxrulethickness} {\vbox} to{\Openboxwidth{{\advance}{\Openboxwidth}
-2{\boxrulethickness} {\hrule} height {\boxrulethickness}
width{\Openboxwidth}{\vfil} {\hrule} height{\boxrulethickness}}}{\vrule}
width{\boxrulethickness}{\hfil} }}}

\begin{document}

\title{The Soliton-Ricci Flow over Compact Manifolds}\author{\\
{\tmstrong{NEFTON PALI}}}\maketitle

\begin{abstract}
  We introduce a flow of Riemannian metrics over compact manifolds with
  formal limit at infinite time a shrinking Ricci soliton. We call this flow the
  Soliton-Ricci flow. It correspond to a Perelman's modified backward Ricci
  type flow with some special restriction conditions. The restriction
  conditions are motivated by convexity results for Perelman's
  $\mathcal{W}$-functional over convex subsets inside adequate subspaces of
  Riemannian metrics. We show indeed that the Soliton-Ricci flow is generated
  by the gradient flow of the restriction of Perelman's
  $\mathcal{W}$-functional over such subspaces. Assuming long time existence
  of the Soliton-Ricci flow we show exponentially fast convergence to a
  shrinking Ricci soliton provided that the Bakry-Emery-Ricci tensor is
  uniformly strictly positive with respect to the evolving metric. 
\end{abstract}

\section{Introduction}

The notion of Ricci soliton (in short RS) has been introduced by D.H. Friedon in \cite{Fri}.
It is a natural generalization of the notion of Einstein metric. The
terminology is justified by the fact that the pull back of the RS metric via
the flow of automorphisms generated by its vector field provides a Ricci flow.
\\
In this paper we introduce the Soliton-Ricci flow, (in short SRF) which is a
flow of Riemannian metrics with formal limit at infinite time a shrinking Ricci soliton.

Our interest in the SRF is motivated from the fact that it generates solutions
of the Soliton-K\"ahler-Ricci flow (in short SKRF) for K\"ahler initial data
(see \cite{Pal2}).

A remarkable formula due to Perelman \cite{Per} shows that the modified (and normalized)
Ricci flow is the gradient flow of Perelman's $\mathcal{W}$ functional with
respect to a fixed choice of the volume form $\Omega$. In this paper we will
denote by
$\mathcal{W}_{\Omega}$ the corresponding Perelman's functional. 
However
Perelman's work does not shows a priory any convexity concerning the
functional $\mathcal{W}_{\Omega}$.

The main attempt of this work is to fit the $\tmop{SRF}$ into a gradient
system picture. We mean by this the picture corresponding to the gradient flow
of a convex functional.

The SRF correspond to a Perelman's modified backward Ricci type flow with
$3$-symmetric covariant derivative of the $\Omega$-Bakry-Emery-Ricci (in short
$\Omega$-BER) tensor along the flow. The notion of SRF (or more precisely of $\Omega$-SRF) is inspired from the
recent work \cite{Pal1} in which we show convexity of Perelman's
$\mathcal{W}_{\Omega}$ functional along geodesics with $3$-symmetric covariant
derivative variation over points with non-negative $\Omega$-BER tensor.

This type of variations plays an important role also because they allow to
generate the $\Omega$-SKRF via the $\Omega$-SRF thanks to an ODE flow of complex
structures of Lax type (see \cite{Pal2}). 
\\
The surprising fact is that the $\Omega$-$\tmop{SRF}$ is a
{\tmstrong{forward}} and strictly parabolic heat type flow with respect to
such variations. However we can not expect to solve this flow equation for
arbitrary initial data. It is well known that backward
heat type equations, such as the backward Ricci flow (roughly speaking), can not be solved for arbitrary initial data.
\\
We will call {\tmstrong{scattering data}} some
special initial data which imply the formal existence of the
$\Omega$-$\tmop{SRF}$ as a formal gradient flow of the restriction of
Perelman's $\mathcal{W}$-functional over adequate subspaces of Riemannian
metrics.

To be more precise, we are looking for sub-varieties $\Sigma$ in the space of
Riemannian metrics such that at each point $g \in \Sigma$ the tangent space of
$\Sigma$ in $g$ is contained in the space of variations with $3$-symmetric
covariant derivative and such that the gradient of the functional
$\mathcal{W}_{\Omega}$ is tangent to $\Sigma$ at each point $g \in \Sigma$.

So at first place we want that the set of initial data allow a $3$-symmetric
covariant derivative of the variation of the metric along the
$\Omega$-$\tmop{SRF}$. The precise definition of the set of scattering data
and of the sub-varieties $\Sigma$ will be given in the next section.

\section{Statement of the main result}\label{main-res}

Let $\Omega > 0$ be a smooth volume form over an oriented Riemannian manifold
$(X, g)$ of dimension $n$. We remind that the $\Omega$-Bakry-Emery-Ricci tensor of $g$ is
defined by the formula
$$
\tmop{Ric}_g (\Omega) \;\; : = \;\; \tmop{Ric} (g)
   \;\, + \;\, \nabla_g \,d \log \frac{dV_g}{\Omega}\; . 
$$
A Riemannian metric $g$ is called a $\Omega$-Shrinking Ricci soliton (in short
$\Omega$-ShRS) if $g = \tmop{Ric}_g (\Omega)$. We observe that the set of
variations with $3$-symmetric covariant derivative coincides with the vector
space
\begin{eqnarray*}
  \mathbbm{F}_g \;\; \assign \;\; \left \{ v \in C^{\infty} \left( X,
  S_{_{\mathbbm{R}}}^2 T^{\ast}_X \right) \mid \hspace{0.25em} \nabla_{_{T_X,
  g}} v_g^{\ast} \;\,=\;\, 0 \right\}\;,
\end{eqnarray*}
where $\nabla_{_{T_X, g}}$ denotes the covariant exterior derivative acting on
$T_X$-valued differential forms and $v^{\ast}_g \assign g^{- 1} v$. We define
also the set of pre-scattering data $\mathcal{S}_{_{^{\Omega}}}$ as the
subset in the space of smooth Riemannian metrics $\mathcal{M}$ over $X$ given
by
\begin{eqnarray*}
  \mathcal{S}_{_{^{\Omega}}} \;\; := \;\; \Big\{ g \in \mathcal{M} \mid \hspace{0.25em}
  \nabla_{_{T_X, g}} \tmop{Ric}^{\ast}_g (\Omega) \;\,=\;\, 0 \Big\} \;.
\end{eqnarray*}
\begin{definition}
  {\tmstrong{$($The $\Omega$-Soliton-Ricci flow$)$}}. Let $\Omega > 0$ be a
  smooth volume form over an oriented Riemannian manifold $X$. A
  $\Omega$-Soliton-Ricci flow $($in short $\Omega$-$\tmop{SRF})$ is a flow of
  Riemannian metrics $(g_t)_{t \geqslant 0} \subset
  \mathcal{S}_{_{^{\Omega}}}$ solution of the evolution equation $ \dot{g}_t \,=\,
  \tmop{Ric}_{g_t} (\Omega) \,-\, g_t$.
\end{definition}

We equip the set $\mathcal{M}$ with the scalar product
\begin{equation}
  \label{Glb-Rm-m} G_g (u, v) \;\;=\;\; \int_X \left\langle \hspace{0.25em} u, v
  \right\rangle_g \,\Omega\;,
\end{equation}
for all $g\in\mathcal{M}$ and  all $u, v \in \mathcal{H} \, : =\, L^2 (X, S^2_{_{\mathbbm{R}}}
T_X^{\ast})$. We denote by $d_G$ the induced distance function. Let
$P_g^{\ast}$ be the formal adjoint of an operator $P$ with respect to a metric
$g$. We observe that the operator
\begin{eqnarray*}
  P^{\ast_{_{\Omega}}}_g \;\; : = \;\; e^f P^{\ast}_g  \left( e^{- f} \bullet
  \right)\;,
\end{eqnarray*}
with $f \assign \log \frac{d V_g}{\Omega}$, is the formal adjoint of $P$ with
respect to the scalar product (\ref{Glb-Rm-m}). We define also the
$\Omega$-Laplacian operator
\begin{eqnarray*}
  \Delta^{^{_{_{\Omega}}}}_g \;\;\assign \;\; \nabla_g^{\ast_{_{\Omega}}} \nabla_g
  \;\;=\;\; \Delta_g \;\,+\;\, \nabla_g f\; \neg\; \nabla_g\; .
\end{eqnarray*}
We remind (see \cite{Pal1}) that the first variation of the
$\Omega$-Bakry-Emery-Ricci tensor is given by the formula
\begin{equation}
  \label{var-Om-Ric} 
2\, \frac{d}{d t} \tmop{Ric}_{g_t} (\Omega) \;\;=\;\; -\;\,
  \nabla_{g_t}^{\ast_{_{\Omega}}} \mathcal{D}_{g_t} \, \dot{g}_t\;,
\end{equation}
where $\mathcal{D}_g \assign \hat{\nabla}_g - 2\, \nabla_g$, with
$\hat{\nabla}_g$ being the symmetrization of
$\nabla_g$ acting on symmetric 2-tensors.
Explicitly
\begin{eqnarray*}
  \hat{\nabla}_g \,\alpha\, (\xi_0, ..., \xi_p) \;\; : = \;\; \sum_{j = 0}^p \nabla_g\,
  \alpha \,(\xi_j, \xi_0, ..., \hat{\xi}_j, ..., \xi_p)\;,
\end{eqnarray*}
for all $p$-tensors $\alpha$. We observe that formula (\ref{var-Om-Ric})
implies directly the variation formula
\begin{equation}
  \label{col-vr-OmRc} 2 \,\frac{d}{d t} \tmop{Ric}_{g_t} (\Omega) \;\;=\;\; -\;\,
  \Delta^{^{_{_{\Omega}}}}_{g_t} \, \dot{g}_t\;,
\end{equation}
along any smooth family $(g_t)_{t \in (0, \varepsilon)} \subset \mathcal{M}$
such that $\dot{g}_t \in \mathbbm{F}_{g_t}$ for all $t \in (0, \varepsilon)$.
We deduce that the $\Omega$-$\tmop{SRF}$ is a {\tmstrong{forward}} and
strictly parabolic heat type flow of Riemannian metrics. In the appendix we
give a direct proof of the variation formula (\ref{col-vr-OmRc}) which shows
that the Laplacian therm on the right hand side is produced from the variation
of the Hessian of $f_t \assign \log \frac{d V_{g_t}}{\Omega}$. Moreover the
formula (\ref{col-vr-OmRc}) implies directly the variation formula
\begin{equation}
  \label{col-vrEnOmRc} 2\, \frac{d}{d t} \tmop{Ric}^{\ast}_{g_t} (\Omega) \;\;=\;\; -\;\,
  \Delta^{^{_{_{\Omega}}}}_{g_t} \, \dot{g}^{\ast}_t \;\,-\;\, 2\, \dot{g}_t^{\ast}
  \tmop{Ric}^{\ast}_{g_t} (\Omega) \;.
\end{equation}
It is quite crucial and natural from the technical point of view to introduce
a ''center of polarization'' $K$ of the space $\mathcal{M}$. We consider
indeed a section $K \in C^{\infty} (X, \tmop{End} (T_X))$ with $n$-distinct
real eigenvalues almost everywhere over $X$ and we define the vector space
\begin{eqnarray*}
  \mathbbm{F}^K_g \;\; := \;\; \Big\{ v \in \mathbbm{F}_g \mid \hspace{0.25em} 
  \left[ \nabla^p_g\, T, v^{\ast}_g \right]\; =\; 0\,,\;\, T \;=\;\mathcal{R}_g\,, K\,, \;\forall p
  \in \mathbbm{Z}_{\geqslant 0} \Big\}\; .
\end{eqnarray*}
(From the technical point of view is more natural to introduce this space in a
different way that we will explain in the next sections.) For any $g_0 \in
\mathcal{M}$ we define the sub-variety
$$
\Sigma_K (g_0) \;\;\assign\;\; \mathbbm{F}^K_{g_0} \cap \mathcal{M}\;. 
$$
It is a totally geodesic and flat sub-variety of the non-positively curved
Riemannian manifold $(\mathcal{M}, G)$ which satisfies the fundamental
property
$$
T_{\Sigma_{^{_K}} (g_0), g}\;\; =\;\;\mathbbm{F}^K_g\;,\quad \forall g \in \Sigma_K (g_0)\;,
$$
(see lemma \ref{flat-Scat-Space} in section \ref{Integ-F} below). We define the {\tmstrong{set of scattering
data}} with center $K$ as the set of metrics
\begin{eqnarray*}
  \mathcal{S}_{_{^{\Omega}}}^K & \assign & \left\{ g \in \mathcal{M} \mid
  \hspace{0.25em} \tmop{Ric}_g (\Omega) \in \mathbbm{F}^K_g \right\}\;,
\end{eqnarray*}
and the subset $ \mathcal{S}_{_{^{\Omega, +}}}^K : = \{ g \in
\mathcal{S}_{_{^{\Omega}}}^K \mid \hspace{0.25em} \tmop{Ric}_g (\Omega) > 0\}$. 
We observe that $\mathcal{S}_{_{^{\Omega, +}}}^K  \neq
\emptyset$ if the manifold $X$ admit a $\Omega$-ShRS. Moreover if $g\in\mathcal{S}_{_{^{\Omega, +}}}^K$ and if
$\dim_{_{\mathbbm{R}}}\mathbbm{F}^K_g=1$ then $g$ solves the $\Omega$-ShRS equation up to a constant factor $\lambda > 0$, i.e 
$\lambda \,g$ is a $\Omega$-ShRS.
With this
notations we can state our main result.

\begin{theorem}
  \label{Main-Teo}{\bf (Main result)}. Let $X$ be a $n$-dimensional
  compact and orientable manifold oriented by a smooth volume form $\Omega >
  0$ and let $K \in C^{\infty} (X, \tmop{End} (T_X))$ with almost everywhere
  $n$-distinct real eigenvalues over $X$. If $\mathcal{S}^K_{_{^{\Omega, +}}} 
  \neq \emptyset$ then hold the following statements.
  
  {\bf(A)} For any data $g_0 \in \mathcal{S}^K_{_{^{\Omega}}}$ and
  any metric $g \in \Sigma_K (g_0)$ hold the identities
  \begin{eqnarray*}
    \nabla_G \mathcal{W}_{\Omega} \,(g) & = & g \;\,-\;\, \tmop{Ric}_g (\Omega) \;\in\;
    T_{\Sigma_K (g_0), g}\;,\\
    &  & \\
    \nabla^{_{^{\Sigma_K (g_0)}}}_G D\, \mathcal{W}_{\Omega} \,(g)\, (v, v) & = &
    \int_X \left[ \left\langle v \tmop{Ric}^{\ast}_g (\Omega), v
    \right\rangle_g \;\,+\;\, \frac{1}{2}\, | \nabla_g \,v|^2_g  \right] \Omega\;,
  \end{eqnarray*}
for all $v\in T_{\Sigma_K (g_0), g}$. The functional $\mathcal{W}_{\Omega}$ is $G$-convex over the
  $G$-convex set
  \begin{eqnarray*}
    \Sigma^-_K (g_0) \;\; \assign \;\; \left\{ g \in \Sigma_K (g_0) \mid
    \tmop{Ric}_g (\Omega) \;\geqslant\; -\; \tmop{Ric}_{g_0} (\Omega) \right\}\;,
  \end{eqnarray*}
  inside the totally geodesic and flat sub-variety $\Sigma_K (g_0)$ of the
  non-positively curved Riemannian manifold $(\mathcal{M}, G)$.
  
  {\bf(B)} For all $g_0 \in \mathcal{S}^K_{_{^{\Omega}}}$ with
  $\tmop{Ric}_{g_0} (\Omega) \geqslant \varepsilon g_0$, $\varepsilon \in
  \mathbbm{R}_{> 0}$ the functional $\mathcal{W}_{\Omega}$ is $G$-convex over
  the $G$-convex sets
  \begin{eqnarray*}
    \Sigma^{\delta}_K (g_0) & \assign & \left\{ g \in \Sigma_K (g_0) \mid
    \tmop{Ric}_g (\Omega) \;\geqslant \;\delta \,g \right\}\;,\quad \forall \delta \;\in\; [0,
    \varepsilon)\;,
\\
\\
    \Sigma^+_K (g_0) & \assign & \left\{ g \in \Sigma_K (g_0) \mid 2
    \tmop{Ric}_g (\Omega) \;+\; g_0\, \Delta^{^{_{_{\Omega}}}}_{g_0} \log (g^{- 1}_0
    g) \;\geqslant\; 0 \right\} \;.
  \end{eqnarray*}
  In this case let $\overline{\Sigma}^+_K (g_0)$ be the closure of $\Sigma^+_K
  (g_0)$ with respect to the metric $d_G$. Then there exists a natural
  integral extension $\mathcal{W}_{\Omega} : \overline{\Sigma}^+_K (g_0)
  \longrightarrow \mathbbm{R}$ of the functional $\mathcal{W}_{\Omega}$ which
  is $d_G$-lower semi-continuous, uniformly bounded from below and $d_G$-convex
  over the $d_G$-closed and $d_G$-convex set $\overline{\Sigma}^+_K (g_0)$
  inside the non-positively curved length space $( \overline{\mathcal{M}}
  ^{d_G}, d_G)$.
  
  {\bf(C)} The formal gradient flow of the functional
  $\mathcal{W}_{\Omega} : \Sigma_K (g_0) \longrightarrow \mathbbm{R}$ with
  initial data $g_0 \in \mathcal{S}^K_{_{^{\Omega, +}}}$ represents a smooth
  solution of the $\Omega$-$\tmop{SRF}$ equation. Assume all time existence of
  the $\Omega$-$\tmop{SRF}$ $(g_t)_{t \geqslant 0} \subset \Sigma_K (g_0)$ and
  the existence of $\delta \in \mathbbm{R}_{> 0}$ such that $\tmop{Ric}_{g_t}
  (\Omega) \geqslant \delta g_t$ for all times $t \geqslant 0$. Then the
  $\Omega$-$\tmop{SRF}$ $(g_t)_{t \geqslant 0}$ converges exponentially fast
  with all its space derivatives to a $\Omega$-shrinking Ricci soliton
  $g_{\tmop{RS}} \in \Sigma^+_K (g_0)$ as $t \rightarrow + \infty$. 
\end{theorem}

We wish to point out that the $G$-convexity of the previous sets is part of the
statement. Moreover it is possible to define a $d_G$-lower semi-continuous and
$d_G$-convex extension of the functional $\mathcal{W}_{\Omega}$ over the
closure of $\Sigma^{\delta}_K (g_0)$ with respect to the metric $d_G$. 
However this statement is not needed for our purposes. 
\\
In order to show the
convexity statements we need to perform a key change of variables which shows
in particular that the SRF equation over $\Sigma_K (g_0)$ corresponds to an
endomorphism-valued porous medium type equation. 
\\
The assumption on the uniform positive lower bound of the $\Omega$-Bakry-Emery-Ricci tensor in the 
statement {\bf(C)} allows to obtain the exponential decay of $C^1(X)$-norms via the maximum principle. 
This assumption seem to be reasonable in view of the $G$-convexity of the sets $\Sigma^{\delta}_K (g_0)$.
\\
The presence of some curvature therms in the evolution equation of higher order space derivatives turns 
off the power of the maximum principle. 
In order to show the
exponentially fast convergence of higher order space derivatives we use an interpolation method introduced by Hamilton 
in his proof of the
exponential convergence of the Ricci flow in \cite{Ham}. 
\\
The difference with the technique in \cite{Ham} is
a more involved interpolation process due to the presence of some extra curvature therms
which seem to be alien to Hamilton's argument. We are able to perform our interpolation 
process by using some intrinsic properties of the $\Omega$-SRF.

\section{Conservative differential symmetries}

In this section we show that some relevant differential symmetries are
preserved along the geodesics induced by the scalar product (\ref{Glb-Rm-m}).

\subsection{First order conservative differential symmetries}
We introduce first the cone $\mathbbm{F}^{\infty}_g$ inside the vector space
$\mathbbm{F}_g$ given by;
\begin{eqnarray*}
  \mathbbm{F}^{\infty}_g & \assign & \left \{ v \in C^{\infty} \left( X,
  S_{_{\mathbbm{R}}}^2 T^{\ast}_X \right) \mid \hspace{0.25em} \nabla_{_{T_X,
  g}} (v_g^{\ast})^p \;=\; 0\,,\;\, \forall p \,\in\, \mathbbm{Z}_{> 0} \right\} \\
  &  & \\
  & = & \left \{ v \in C^{\infty} \left( X, S_{_{\mathbbm{R}}}^2 T^{\ast}_X
  \right) \mid \hspace{0.25em} \nabla_{_{T_X, g}} e^{t v^{\ast}_g} \;=\; 0\,,\;\,
  \forall t \,\in\, \mathbbm{R} \right\} \;.
\end{eqnarray*}
We need also a few algebraic definitions. Let $V$ be a real vector space. We
consider the contraction operator
$$
\neg \,:\, \tmop{End} (V) \times \Lambda^2 V^{\ast} \longrightarrow \Lambda^2
   V^{\ast} \;, 
$$
defined by the formula
$$
H \;\neg\; (\alpha \wedge \beta) \;: =\; (\alpha \cdot H) \wedge \beta \;\,+\;\, \alpha
   \wedge (\beta \cdot H) \;, 
$$
for any $H \in \tmop{End} (V), \hspace{0.25em} u, v \in V$ and $\alpha, \beta
\in V^{\ast}$. We can also define the contraction operator by the equivalent
formula
$$
(H \;\neg\; \varphi) \,(u, v) \;\;: =\;\; \varphi\, (H u, v) \;\,+\;\, \varphi\, (u, H v)\,, 
$$
for any $\varphi \in \Lambda^2 V^{\ast}$. Moreover for any element $A \in
(V^{\ast})^{\otimes 2} \otimes V$ we define the following elementary
operations over the vector space $(V^{\ast})^{\otimes 2} \otimes V$;
\begin{eqnarray*}
  (A\, H) (u, v) & \assign & A (u, H v)\;,\\
  &  & \\
  (H A) (u, v) & \assign & H A (u, v)\;,\\
  &  & \\
  (H \bullet A) (u, v) & \assign & A (H u, v)\;,\\
  &  & \\
  (\tmop{Alt} A) (u, v) & \assign & A (u, v) \;\,-\;\, A (v, u) \;.
\end{eqnarray*}
Assume now that $V$ is equipped with a metric $g$. Then we can define the $g$-transposed $A_g^T \in
(V^{\ast})^{\otimes 2} \otimes V$ as follows. For any $v \in V$
\begin{eqnarray*}
  v \;\neg\; A_g^T \;\; : = \;\; \left( v \;\neg\; A \right)_g^T \;.
\end{eqnarray*}
We remind (see \cite{Pal1}) that the geodesics in the space of Riemannian metrics
with respect to the scalar product (\ref{Glb-Rm-m}) are given by the solutions of the
equation $\dot{g}_t^{\ast}\assign g_t^{- 1} \dot{g}_t = g_0^{- 1}
\dot{g}_0$. Thus the geodesic curves write explicitly as
\begin{equation}
  \label{expr-geod} g_t \;\; = \;\; g_0 \hspace{0.25em}
  e^{\hspace{0.25em} tg_0^{- 1} \dot{g}_0} \; .
\end{equation}
With this notations we can show now the following fact.

\begin{lemma}
  \label{invar-F}Let $(g_t)_{t \in \mathbbm{R}}$ be a geodesic such that
  $\dot{g}_0 \in \mathbbm{F}^{\infty}_{g_0} .$ Then $\dot{g}_t \in
  \mathbbm{F}^{\infty}_{g_t}$ for all $t \in \mathbbm{R}.$ 
\end{lemma}

\begin{proof}
  Let $H \in C^{\infty} \left( X, \tmop{End} (T_X) \right) \text{}$ and let
  $(g_t)_{t \in \mathbbm{R}} \subset \mathcal{M}$ be an arbitrary smooth
  family. We expand first the time derivative
  \begin{eqnarray*}
    \dot{\nabla}_{_{T_X, g_t}} H (\xi, \eta) & = & \dot{\nabla}_{g_t} H (\xi,
    \eta) \;\,-\;\, \dot{\nabla}_{g_t} H (\eta, \xi)\\
    &  & \\
    & = & \dot{\nabla}_{g_t} (\xi, H \eta) \;\,-\;\, \dot{\nabla}_{g_t} (\eta, H
    \xi)\\
    &  & \\
    & - & H \left[ \dot{\nabla}_{g_t} (\xi, \eta) \;\,-\;\, \dot{\nabla}_{g_t} (\eta,
    \xi) \right]\\
    &  & \\
    & = & \dot{\nabla}_{g_t} (H \eta, \xi) \;\,-\;\, \dot{\nabla}_{g_t} (H \xi,
    \eta)\;,
  \end{eqnarray*}
  since $ \dot{\nabla}_{g_t} \in C^{\infty} (X, S_{_{\mathbbm{R}}}^2
  T^{\ast}_X \otimes T_X)$ thanks to the variation identity (see \cite{Bes})
\begin{eqnarray}
  \label{Bess-var-LC-conn} 
2\, g_t \left( \dot{\nabla}_{g_t} (\xi, \eta), \mu
  \right) \;\,=\;\, \nabla_{g_t}\, \dot{g}_t (\xi, \eta, \mu) \;\,+\;\, \nabla_{g_t}\, \dot{g}_t
  (\eta, \xi, \mu) \;\,-\;\, \nabla_{g_t} \,\dot{g}_t (\mu, \xi, \eta)\; .\;\;
\end{eqnarray}
  We observe now that the variation formula (\ref{Bess-var-LC-conn}) rewrites
  as
  \begin{eqnarray*}
    2 \,\dot{\nabla}_{g_t} (\xi, \eta) \;\; = \;\; \nabla_{g_t} \, \dot{g}^{\ast}_t
    (\xi, \eta) \;\,+\;\, \nabla_{g_t} \, \dot{g}^{\ast}_t (\eta, \xi) \;\,-\;\, \left(
    \nabla_{g_t} \, \dot{g}^{\ast}_t \, \eta \right)_{g_t}^T  \xi\; .
  \end{eqnarray*}
  Thus
  \begin{eqnarray*}
    2\, \dot{\nabla}_{_{T_X, g_t}} H (\xi, \eta) & = & \nabla_{g_t}\, 
    \dot{g}^{\ast}_t (H \eta, \xi) \;\,+\;\, \nabla_{g_t} \, \dot{g}^{\ast}_t (\xi, H
    \eta) \;\,-\;\, \left( \nabla_{g_t} \, \dot{g}^{\ast}_t \, \xi \right)_{g_t}^T
     H \eta\\
    &  & \\
    & - & \nabla_{g_t} \, \dot{g}^{\ast}_t (H \xi, \eta) \;\,-\;\, \nabla_{g_t} \,
    \dot{g}^{\ast}_t (\eta, H \xi) \;\,+\;\, \left( \nabla_{g_t} \, \dot{g}^{\ast}_t
    \, \eta \right)_{g_t}^T H \xi \;.
  \end{eqnarray*}
  Applying the identity
  \begin{eqnarray*}
    \left( \nabla_{g_t} \, \dot{g}^{\ast}_t \, \xi \right)_{g_t}^T & = & -\;\,
    \left( \xi \;\neg\; \nabla_{_{T_X, g_t}}  \dot{g}^{\ast}_t \right)_{g_t}^T \;\,+\;\,
    \xi \;\neg\; \nabla_{g_t}  \dot{g}^{\ast}_t\;,
  \end{eqnarray*}
  we obtain the equalities
  \begin{eqnarray*}
    2 \,\dot{\nabla}_{_{T_X, g_t}} H (\xi, \eta) & = & \nabla_{g_t} \,
    \dot{g}^{\ast}_t (H \eta, \xi) \;\,-\;\, \nabla_{g_t} \, \dot{g}^{\ast}_t (H \xi,
    \eta)\\
     \\
    & + & \left( \xi \;\neg\; \nabla_{_{T_X, g_t}}  \dot{g}^{\ast}_t
    \right)_{g_t}^T  H \eta \;\,-\;\, \left( \eta \;\neg\; \nabla_{_{T_X, g_t}} 
    \dot{g}^{\ast}_t \right)_{g_t}^T  H \xi\\
    \\
    & = & - \;\,\left( H \;\neg\; \nabla_{_{T_X, g_t}}  \dot{g}^{\ast}_t \right)
    (\xi, \eta)\\
\\
& +& \nabla_{g_t} \, \dot{g}^{\ast}_t (\xi, H \eta) \;\,-\;\, \nabla_{g_t} \,
    \dot{g}^{\ast}_t (\eta, H \xi)\\
     \\
    & + & \tmop{Alt} \left[ \left( \nabla_{_{T_X, g_t}}  \dot{g}^{\ast}_t
    \right)_{g_t}^T H \right] (\xi, \eta) \;.
  \end{eqnarray*}
  We infer the variation formula
  \begin{eqnarray}
    2\, \dot{\nabla}_{_{T_X, g_t}} H & = & - \;\,H \;\neg\; \nabla_{_{T_X, g_t}} 
    \dot{g}^{\ast}_t \;\,+\;\, \nabla_{_{T_X, g_t}}  \left( \dot{g}^{\ast}_t H \right)
   \;\, -\;\, \dot{g}^{\ast}_t \,\nabla_{_{T_X, g_t}} H\nonumber
\\\nonumber
     \\
    & + & \tmop{Alt} \left[ \left( \nabla_{_{T_X, g_t}}  \dot{g}^{\ast}_t
    \right)_{g_t}^T H \right] \;.\label{var-extD}
  \end{eqnarray}
  Thus along any geodesic hold the upper triangular type infinite dimensional
  ODE system
  \begin{eqnarray*}
    2 \,\frac{d}{d t}  \left[ \nabla_{_{T_X, g_t}} ( \dot{g}^{\ast}_t)^p \right]
    \text{} & = & -\;\, ( \dot{g}^{\ast}_t)^p \;\neg\; \nabla_{_{T_X, g_t}} 
    \dot{g}^{\ast}_t 
\\
\\
&+& \nabla_{_{T_X, g_t}}  \left( \dot{g}^{\ast}_t
    \right)^{p + 1} \;\,-\;\, \dot{g}^{\ast}_t \,\nabla_{_{T_X, g_t}} (
    \dot{g}^{\ast}_t)^p\\
    &  & \\
    & + & \tmop{Alt} \left[ \left( \nabla_{_{T_X, g_t}}  \dot{g}^{\ast}_t
    \right)_{g_t}^T ( \dot{g}^{\ast}_t)^p \right]\;,
  \end{eqnarray*}
  for all $p \in \mathbbm{Z}_{> 0}$. We remind now that $\dot{g}^{\ast}_t
  \equiv \dot{g}^{\ast}_0$ and we observe the formula
  \begin{eqnarray*}
    \frac{d}{d t}  \left( \nabla_{_{T_X, g_t}}  \dot{g}^{\ast}_t
    \right)_{g_t}^T & = & \left[ \left( \nabla_{_{T_X, g_t}}  \dot{g}^{\ast}_t
    \right)_{g_t}^T, \dot{g}^{\ast}_t \right] \;+\; \left[ \frac{d}{d t}  \left(
    \nabla_{_{T_X, g_t}}  \dot{g}^{\ast}_t \right) \right]_{g_t}^T\; .
  \end{eqnarray*}
  This combined once again with the identity $\dot{g}^{\ast}_t \equiv
  \dot{g}^{\ast}_0$ and with the previous variation formula implies that for
  all $k, p \in \mathbbm{Z}_{\geqslant 0}$ hold the identity
  \begin{eqnarray*}
    \frac{d^k}{d t^k} _{\mid_{t = 0}}  \left[ \nabla_{_{T_X, g_t}} (
    \dot{g}^{\ast}_t)^p \right] & = & 0 \;.
  \end{eqnarray*}
  Indeed this follows from an increasing induction in $k$. The conclusion
  follows from the fact that the curves
  \begin{eqnarray*}
    t & \longmapsto & \nabla_{_{T_X, g_t}} ( \dot{g}^{\ast}_t)^p\;,
  \end{eqnarray*}
  are real analytic over the real line.
\end{proof}

Let now $A \in (V^{\ast})^{\otimes p} \otimes V$, $B \in (V^{\ast})^{\otimes
q} \otimes V$ and let $k = 1, \ldots q$. We define the generalized product
operation
\begin{eqnarray*}
  (A \,B) (u_1, \ldots, u_{p - 1}, v_1, \ldots, v_q) & \assign & A (u_1, \ldots,
  u_{p - 1}, B (v_1, \ldots, v_q)) \;.
\end{eqnarray*}
With this notations we define the vector space
\begin{eqnarray*}
\mathbbm{E}_g \;\, \assign \;\, \Big \{ v \in C^{\infty} \left( X,
  S_{_{\mathbbm{R}}}^2 T^{\ast}_X \right) \mid \hspace{0.25em}  \left[
  \mathcal{R}_g, v^{\ast}_g  \right] \;=\; 0\,, \;\left[ \mathcal{R}_g, \nabla_{g,
  \xi} v^{\ast}_g  \right] \;=\; 0\,,\; \forall \xi \in T_X \Big\}\;,
\end{eqnarray*}
and we show the following crucial fact.

\begin{lemma}
  \label{invar-Rm}Let $(g_t)_{t \in \mathbbm{R}} \subset \mathcal{M}$ be a
  geodesic such that $\dot{g}_0 \in \mathbbm{F}^{\infty}_{g_0} \cap
  \mathbbm{E}_{g_0}$. Then $\dot{g}_t \in \mathbbm{F}_{g_t}^{\infty} \cap
  \mathbbm{E}_{g_t}$ for all $t \in \mathbbm{R}$.
\end{lemma}

\begin{proof}
  We observe first that the variation identity (\ref{Bess-var-LC-conn})
  combined with the fact that $\dot{g}_t \in \mathbbm{F}_{g_t}$ implies the
  variation identity
  \begin{equation}
    \label{Pal-var-LC-conn} 2\,\dot{\nabla}_{g_t} \;\;=\;\; \nabla_{g_t} 
    \dot{g}^{\ast}_t \;.
  \end{equation}
  Thus the variation formula (see \cite{Bes})
  \begin{equation}
    \label{var-Rm}  \dot{\mathcal{R}}_{g_t} (\xi, \eta) \mu \;\; =\;\;
     \nabla_{g_t} \dot{\nabla}_{g_t} (\xi, \eta, \mu)
    \;\,-\;\, \nabla_{g_t} \dot{\nabla}_{g_t} (\eta,
    \xi, \mu) \;,
  \end{equation}
  rewrites as
  \begin{eqnarray*}
    2\, \dot{\mathcal{R}}_{g_t} (\xi, \eta) & = & \nabla_{g_t, \xi} \nabla_{g_t,
    \eta} \, \dot{g}^{\ast}_t \;\,-\;\, \nabla_{g_t, \eta} \nabla_{g_t, \xi} \,
    \dot{g}^{\ast}_t \;\,-\;\, \nabla_{g_t, \left[ \xi, \eta \right]} \,
    \dot{g}^{\ast}_t 
\\\\
&=& \left[ \mathcal{R}_{g_t} (\xi, \eta), \dot{g}^{\ast}_t
    \right]\; .
  \end{eqnarray*}
  We remind in fact the general identity
  \begin{equation}
    \label{com-cov} \nabla_{g, \xi} \nabla_{g, \eta}\, H \;\,-\;\, \nabla_{g, \eta}
    \nabla_{g, \xi}\, H \;\;=\;\; \left[ \mathcal{R}_g (\xi, \eta), H \right] \;\,+\;\,
    \nabla_{g, \left[ \xi, \eta \right]} \,H\;,
  \end{equation}
  for any $H \in C^{\infty} \left( X, \tmop{End} (T_X) \right)$. We
  deduce the variation identity
  \begin{equation}
    \label{col-vr-Rm} 2 \,\dot{\mathcal{R}}_{g_t} \;\;=\;\; \left[ \mathcal{R}_{g_t},
    \dot{g}^{\ast}_t \right]\;,
  \end{equation}
  (for any smooth curve $(g_t)_t$ such that $\dot{g}_t \in
  \mathbbm{F}_{g_t}$), and the variation formula
  \begin{eqnarray*}
    2\, \frac{d}{d t}  \left[ \mathcal{R}_{g_t}, \dot{g}^{\ast}_t \right] & = &
    \Big[ \left[ \mathcal{R}_{g_t}, \dot{g}^{\ast}_t \right],
    \dot{g}^{\ast}_t \Big] \;.
  \end{eqnarray*}
  Thus the identity $\left[ \mathcal{R}_{g_t}, \dot{g}^{\ast}_t \right] = 0$
  hold for all times by Cauchy uniqueness. We infer in particular
  $\mathcal{R}_{g_t} =\mathcal{R}_{g_0}$ for all $t \in \mathbbm{R}$ thanks to the identity
  (\ref{col-vr-Rm}). Using the variation formula
  \begin{equation}
    \label{F-var-LC} 2\, \dot{\nabla}_{g_t} \,H \;\;=\;\; 2\, \dot{\nabla}_{g_t} \cdot H \;\,-\;\,
    2\, H \,\dot{\nabla}_{g_t} \;\;=\;\; \left[ \nabla_{g_t} \, \dot{g}^{\ast}_t, H \right]\;,
  \end{equation}
  we deduce
  \begin{eqnarray*}
    2\, \frac{d}{d t}  \left[ \mathcal{R}_{g_t}, \nabla_{g_t, \xi} \,
    \dot{g}^{\ast}_t \right] & = & \Big[ \mathcal{R}_{g_t}, \left[
    \nabla_{g_t, \xi} \, \dot{g}^{\ast}_t, \dot{g}^{\ast}_t \right] \Big]\\
    &  & \\
    & = &  -\;\, \Big[ \nabla_{g_t, \xi} \, \dot{g}^{\ast}_t, \left[
    \mathcal{R}_{g_t}, \dot{g}^{\ast}_t \right] \Big]
\;\, -\;\, \Big[
    \dot{g}^{\ast}_t, \left[ \mathcal{R}_{g_t}, \nabla_{g_t, \xi} \,
    \dot{g}^{\ast}_t \right] \Big]\\
    \\
    & = & \Big[ \left[ \mathcal{R}_{g_t}, \nabla_{g_t, \xi} \,
    \dot{g}^{\ast}_t\right], \dot{g}^{\ast}_t \Big]\;,
  \end{eqnarray*} 
  by the Jacobi identity and by the previous result. We infer the conclusion
  by Cauchy uniqueness.
\end{proof}

\subsection{Conservation of the pre-scattering condition}

This sub-section is the hart of the paper. We will show the conservation of the
pre-scattering condition along curves with variations in $\mathbbm{F}_g \cap
\mathbbm{E}_g$. We need to introduce first a few other product notations. Let
$(e_k)_k$ be a $g$-orthonormal basis. For any elements $A \in
(T_X^{\ast})^{\otimes 2} \otimes T_X$ and $B \in \Lambda^2 T^{\ast}_X \otimes
\tmop{End} (T_X)$ we define the generalized products
\begin{eqnarray*}
  (B \ast A) (u, v) & \assign & B (u, e_k) A (e_k, v)\;,\\
   \\
  (B \circledast A) (u, v) & \assign & \left[ B (u, e_k)\,, e_k \;\neg\; A \right] v\;,\\
   \\
  (A \ast B) (u, v) & \assign & A (e_k, B (u, v) e_k)\; .
\end{eqnarray*}
We observe that the algebraic Bianchi identity implies
\begin{equation}
  \label{curv-alg-id} \tmop{Alt} (\mathcal{R}_g \circledast A) \;\;=\;\; \tmop{Alt}
  (\mathcal{R}_g \ast A) \;\,-\;\, A \ast \mathcal{R}_g\; .
\end{equation}
Let also $H \in C^{\infty} (X, \tmop{End} (T_X)$. Then hold the identity
\begin{equation}
  \label{ext-alg-prod} \nabla_{_{T_X, g}} H \ast \mathcal{R}_g \;\;=\;\; 2\, \nabla_g\, H
  \ast \mathcal{R}_g \;.
\end{equation}
We observe in fact the equalities
\begin{eqnarray*}
  \nabla_{_{T_X, g}} H \ast \mathcal{R}_g & = & \nabla_g \,H \ast \mathcal{R}_g
  \;\,-\;\, \nabla_g\, H \,(\mathcal{R}_g e_k, e_k) \;\;=\;\; 2\, \nabla_g H \ast \mathcal{R}_g\; .
\end{eqnarray*}
This follows writing with respect to the $g$-orthonormal basis $(e_k)$ the
identity \
\begin{eqnarray*}
  \mathcal{R}_g (\xi, \eta) & = & - \;\,\left( \mathcal{R}_g (\xi, \eta)
  \right)_g^T\;,
\end{eqnarray*}
which is a consequence of the alternating property of the $(4, 0)$-Riemann
curvature operator.

For any $A \in C^{\infty} (X, (T_X^{\ast})^{\otimes p + 1} \otimes T_X)$ we
define the divergence type operations
\begin{eqnarray*}
  \underline{\tmop{div} }_g \,A (u_1, \ldots, u_p) & \assign & \tmop{Tr}_g \big[
  \nabla_g A (\cdot, u_1, \ldots, u_p, \cdot) \big]\;,\\
  &  & \\
  \underline{\tmop{div}}^{^{_{_{\Omega}}}}_g A (u_1, \ldots, u_p) & : = &
  \underline{\tmop{div}}_g \,A (u_1, \ldots, u_p) \;\,-\;\, A (u_1, \ldots, u_p,
  \nabla_g f) \;.
\end{eqnarray*}
We remind that the once contracted differential Bianchi identity writes often as
$\underline{\tmop{div}}_g \mathcal{R}_g = - \nabla_{_{T_X, g}}
\tmop{Ric}^{\ast}_g$. This combined with the identity $\nabla_{_{T_X, g}}
\nabla^2_g f =\mathcal{R}_g \cdot \nabla_g f$ implies
\begin{equation}
  \label{Om-cntr-Bianc}  \underline{\tmop{div}}^{^{_{_{\Omega}}}}_g
  \mathcal{R}_g \;\;=\;\; -\;\, \nabla_{_{T_X, g}} \tmop{Ric}^{\ast}_g (\Omega) \;.
\end{equation}
With the previous notations hold the following lemma.

\begin{lemma}
  \label{F-var-pre-scat}Let $(g_t)_{t \in \mathbbm{R}} \subset \mathcal{M}$ be
  a smooth family such that $\dot{g}_t \in \mathbbm{F}_{g_t}$ for all $t \in
  \mathbbm{R}$. Then hold the variation formula
  \begin{eqnarray*}
    2\, \frac{d}{d t}  \left[ \nabla_{_{T_X, g_t}} \tmop{Ric}^{\ast}_{g_t}
    (\Omega) \right] & = & \underline{\tmop{div}}^{^{_{_{\Omega}}}}_{g_t} 
    \left[ \mathcal{R}_{g_t}, \dot{g}^{\ast}_t \right] \;\,+\;\, \tmop{Alt} \left(
    \mathcal{R}_{g_t} \circledast \nabla_{g_t}  \dot{g}^{\ast}_t \right) 
\\
\\
&-& 2\,
    \dot{g}_t^{\ast} \,\nabla_{_{T_X, g_t}} \tmop{Ric}^{\ast}_{g_t} (\Omega)\; .
  \end{eqnarray*}
\end{lemma}

\begin{proof}
  We will show the above variation formula by means of the identity
  (\ref{Om-cntr-Bianc}). Consider any $B \in C^{\infty} (X, \Lambda^2
  T^{\ast}_X \otimes \tmop{End} (T_X))$. Time deriving the definition of the
  covariant derivative $\nabla_{g_t} B$ we deduce the formula
  \begin{eqnarray*}
    \dot{\nabla}_{g_t} B (\xi, u, v)\, w & = & \dot{\nabla}_{g_t}  \left( \xi, B
    (u, v) \,w \right) \;\,-\;\, B \left( \dot{\nabla}_{g_t} (\xi, u), v \right) w\\
    &  & \\
    & - & B \left( u, \dot{\nabla}_{g_t} (\xi, v) \right) w \;\,-\;\, B (u, v)
    \dot{\nabla}_{g_t} (\xi, w) \;.
  \end{eqnarray*}
  We infer the expression
  \begin{eqnarray}
    2\, \dot{\nabla}_{g_t} B (\xi, u, v)\,w & = & \nabla_{g_t, \xi} \,
    \dot{g}^{\ast}_t \,B (u, v) \,w \;\,-\;\, B \left( \nabla_{g_t, \xi}  \,\dot{g}^{\ast}_t
    u, v \right) w\nonumber
\\\nonumber
    &  & \\
    & - & B \left( u, \nabla_{g_t, \xi}  \,\dot{g}^{\ast}_t \,v \right) w \;\,-\;\, B
    \left( u, v \right) \nabla_{g_t, \xi} \, \dot{g}^{\ast}_t \,w\;,\qquad\label{var-covB}
  \end{eqnarray}
  thanks to the formula (\ref{Pal-var-LC-conn}). We fix now an arbitrary
  space-time point $(x_0, t_0)$ and we pick a local tangent frame $(e_k)_k$ in
  a neighborhood of $x_0$ which is $g_{t_0} (x_0)$-orthonormal at the point
  $x_0$ and satisfies $\nabla_{g_t} e_j (x_0) = 0$ at the time $t_0$ for all
  $j$. Then time deriving the therm
  \begin{eqnarray*}
    ( \underline{\tmop{div}}^{^{_{_{\Omega}}}}_{g_t} B) (\xi, \eta) & = &
    \nabla_{g_t, e_k} B \left( \xi, \eta \right) g^{- 1}_t e^{\ast}_k \;\,-\;\, B
    (\xi, \eta)\, \nabla_{g_t} f_t\;,
  \end{eqnarray*}
  and using the expression (\ref{var-covB}) we obtain the identity
  \begin{eqnarray}
    2\, \frac{d}{d t} \,( \underline{\tmop{div}}^{^{_{_{\Omega}}}}_{g_t} B) (\xi,
    \eta) & = & \nabla_{g_t, e_k} \, \dot{g}^{\ast}_t \,B \left( \xi, \eta \right)
    e_k \;\,-\;\, B \left( \nabla_{g_t, e_k} \, \dot{g}^{\ast}_t \,\xi, \eta \right) e_k\nonumber
\\\nonumber
    &  & \\
    & - & B \left( \xi, \nabla_{g_t, e_k} \, \dot{g}^{\ast}_t \,\eta \right) e_k
    \;\,-\;\, B \left( \xi, \eta \right) \nabla_{g_t, e_k} \, \dot{g}^{\ast}_t \,e_k\nonumber
\\\nonumber
    &  & \\
    & - & 2\, \nabla_{g_t, e_k} B \left( \xi, \eta \right) \dot{g}^{\ast}_t e_k
    \;\,-\;\, 2\, B (\xi, \eta)\; \frac{d}{d t} \; \mathcal{\nabla}_{g_t} f_t\;,\label{var-divB}\qquad\quad
  \end{eqnarray}
  at the space-time $(x_0, t_0)$. Moreover hold the elementary formula
  \begin{eqnarray*}
    2\, \frac{d}{d t}  \,\mathcal{\nabla}_{g_t} f_t & = & \mathcal{\nabla}_{g_t}
    \tmop{Tr}_{g_t}  \dot{g}_t \;\,-\;\, 2\, \dot{g}^{\ast}_t  \,\mathcal{\nabla}_{g_t}
    f_t\; .
  \end{eqnarray*}
  We observe also that at the space time point $(x_0, t_0)$ hold the trivial equalities
  \begin{eqnarray*}
    \mathcal{\nabla}_{g_t} \tmop{Tr}_{g_t}  \dot{g}_t & = & e_k \,. \left(
    \tmop{Tr}_{_{\mathbbm{R}}} \dot{g}^{\ast}_t \right) e_k\\
    &  & \\
    & = & e_k \,.\, g_t \left( \dot{g}^{\ast}_t \, e_j, e_j \right) e_k\\
    &  & \\
    & = & g_t \left( \nabla_{g_t, e_k} \, \dot{g}^{\ast}_t \, e_j, e_j
    \right) e_k\\
    &  & \\
    & = &  \nabla_{g_t} \, \dot{g}_t \,(e_k, e_j, e_j) \,e_k\\
    &  & \\
    & = &  \nabla_{g_t} \, \dot{g}_t \,(e_j, e_j, e_k) e_k\\
    &  & \\
    & = & g_t\left( \nabla_{g_t, e_j} \, \dot{g}^{\ast}_t \, e_j, e_k
    \right) e_k\\
    &  & \\
    & = & - \;\,\nabla_{g_t}^{\ast} \, \dot{g}^{\ast}_t\;,
  \end{eqnarray*}
thanks to the assumption $\dot{g}_t \in \mathbbm{F}_{g_t}$. We deduce the identity
  \begin{equation}
    \label{col-vr-Grad} 2\, \frac{d}{d t} \, \mathcal{\nabla}_{g_t} f_t \;\;=\;\; -\;\,
    \nabla_{g_t}^{\ast} \, \dot{g}^{\ast}_t \;\,-\;\, 2\, \dot{g}^{\ast}_t \,
    \mathcal{\nabla}_{g_t} f_t \;.
  \end{equation}
  Thus the identity (\ref{var-divB}) rewrites as follows;
  \begin{eqnarray*}
    2 \,\frac{d}{d t} \,( \underline{\tmop{div}}^{^{_{_{\Omega}}}}_{g_t} B) (\xi,
    \eta) & = & \left( \nabla_{g_t} \, \dot{g}^{\ast}_t \ast B \right) (\xi,
    \eta) 
\\
\\
&-& B \left( \nabla_{g_t, e_k} \, \dot{g}^{\ast}_t \,\xi, \eta \right) e_k
    \;\,-\;\, B \left( \xi, \nabla_{g_t, e_k} \, \dot{g}^{\ast}_t \,\eta \right) e_k\\
    &  & \\
    & - & 2 \,\nabla_{g_t, e_k} B \left( \xi, \eta \right) \,\dot{g}^{\ast}_t e_k
    \;\,+\;\, 2\, B (\xi, \eta) \,\nabla^{\ast_{_{\Omega}}}_{g_t}  \dot{g}^{\ast}_t\; .
  \end{eqnarray*}
  The assumption $\dot{g}_t \in \mathbbm{F}_{g_t}$ implies that the
  endomorphism $\nabla_{g_t, \bullet} \, \dot{g}^{\ast}_t \,\xi$ 
is
  $g_t$- symmetric. Thus we can choose a $g_{t_0}(x_0)$-orthonormal basis $(e_k) \subset
  T_{X, x_0}$ which diagonalize it at the space-time point $(x_0, t_0)$. It is
  easy to see that with respect to this basis hold the identity
  \begin{eqnarray*}
    B \left( \nabla_{g_t, e_k} \, \dot{g}^{\ast}_t \,\xi, \eta \right) e_k \;\; = \;\; -\;\,
    \left( B \ast \nabla_{g_t} \, \dot{g}^{\ast}_t \right) (\eta, \xi)\;,
  \end{eqnarray*}
  by the alternating property of $B$. But the therm on the left hand side is
  independent of the choice of the $g_{t_0}(x_0)$-orthonormal basis. In a similar way
  choosing a $g_{t_0}(x_0)$-orthonormal basis $(e_k)_k \subset T_{X, x_0}$ which
  diagonalizes $\nabla_{g_t, \bullet} \, \dot{g}^{\ast}_t \,\eta$ at the
  space-time point $(x_0, t_0)$ we obtain the identity
  \begin{eqnarray*}
    B \left( \xi, \nabla_{g_t, e_k} \, \dot{g}^{\ast}_t \,\eta \right) e_k \;\; = \;\;
    \left( B \ast \nabla_{g_t} \, \dot{g}^{\ast}_t \right) (\xi, \eta) \;.
  \end{eqnarray*}
  We infer the equality
  \begin{eqnarray*}
    2 \,\left( \frac{d}{d t} \, \underline{\tmop{div}}^{^{_{_{\Omega}}}}_{g_t}
    \right) B & = & \nabla_{g_t} \, \dot{g}^{\ast}_t \ast B \;\,-\;\, \tmop{Alt} \left(
    B \ast \nabla_{g_t} \, \dot{g}^{\ast}_t \right) 
\\
\\
&-& 2\, \nabla_{g_t, e_k} B\,
    \dot{g}^{\ast}_t e_k \;\,+\;\, 2\, B \,\nabla^{\ast_{_{\Omega}}}_{g_t} 
    \dot{g}^{\ast}_t\;,
  \end{eqnarray*}
  with respect to any $g_{t_0}(x_0)$-orthonormal basis $(e_k)$ at the arbitrary space-time point $(x_0,t_0)$. This combined with
  (\ref{col-vr-Rm}) and (\ref{curv-alg-id}) implies the equalities
  \begin{eqnarray*}
    2\, \frac{d}{d t}  \left( \underline{\tmop{div}}^{^{_{_{\Omega}}}}_{g_t}
    \mathcal{R}_{g_t} \right) & = & 2 \left( \frac{d}{d t} \,
    \underline{\tmop{div}}^{^{_{_{\Omega}}}}_{g_t} \right) \mathcal{R}_{g_t} \;\,+\;\,
    \underline{\tmop{div}}^{^{_{_{\Omega}}}}_{g_t}  \left[ \mathcal{R}_{g_t},
    \dot{g}^{\ast}_t \right]\\
    \\
    & = & - \;\,\tmop{Alt} \left( \mathcal{R}_{g_t} \circledast \nabla_{g_t} \,
    \dot{g}^{\ast}_t \right) \;\,-\;\, 2\, \nabla_{g_t, e_k} \mathcal{R}_{g_t} \,
    \dot{g}^{\ast}_t e_k  \\
     \\
    & + & 2\,\mathcal{R}_{g_t} \nabla^{\ast_{_{\Omega}}}_{g_t} 
    \dot{g}^{\ast}_t \;\,+\;\,\underline{\tmop{div}}^{^{_{_{\Omega}}}}_{g_t}  \left[
    \mathcal{R}_{g_t}, \dot{g}^{\ast}_t \right] \;.
  \end{eqnarray*}
  We observe now that for any smooth curve $(g_t)_{t \in \mathbbm{R}} \subset
  \mathcal{M}$ hold the identity
  \begin{eqnarray}
    \label{div-id-curv}  \underline{\tmop{div}}^{^{_{_{\Omega}}}}_{g_t} 
    \left[ \mathcal{R}_{g_t}, \dot{g}^{\ast}_t \right] 
&=&
 \nabla_{g_t, e_k}
    \mathcal{R}_{g_t} \, \dot{g}^{\ast}_t e_k \;\,-\;\,\mathcal{R}_{g_t} 
    \nabla^{\ast_{_{\Omega}}}_{g_t}  \dot{g}^{\ast}_t \nonumber
\\\nonumber
\\
&-& \nabla_{g_t} \,
    \dot{g}^{\ast}_t \ast \mathcal{R}_{g_t} \;\,-\;\, \dot{g}^{\ast}_t \,
    \underline{\tmop{div}}^{^{_{_{\Omega}}}}_{g_t} \mathcal{R}_{g_t} \;.
  \end{eqnarray}
But our assumption $\dot{g}_t \in \mathbbm{F}_{g_t}$ implies $\nabla_{g_t} \,
  \dot{g}^{\ast}_t \ast \mathcal{R}_{g_t} \equiv 0$ thanks to the identity
  (\ref{ext-alg-prod}). Thus we obtain the formula
\begin{eqnarray*}
2 \,\frac{d}{d t}  \left( \underline{\tmop{div}}^{^{_{_{\Omega}}}}_{g_t}
    \mathcal{R}_{g_t} \right) 
& = &
 -\;\, \tmop{Alt} \left( \mathcal{R}_{g_t}
    \circledast \nabla_{g_t} \, \dot{g}^{\ast}_t \right) \;\,-\;\,
    \underline{\tmop{div}}^{^{_{_{\Omega}}}}_{g_t}  \left[ \mathcal{R}_{g_t},
    \dot{g}^{\ast}_t \right] 
\\
\\
&-& 2\, \dot{g}^{\ast}_t \,
    \underline{\tmop{div}}^{^{_{_{\Omega}}}}_{g_t} \mathcal{R}_{g_t}\;,
\end{eqnarray*}
  which implies the required conclusion thanks to the identity (\ref{Om-cntr-Bianc}).
\end{proof}

\begin{corollary}
  \label{vr-Om-EndRc}{\tmstrong{$($Conservation of the pre-scattering
  condition$)$}} \\
Let $(g_t)_{t \in \mathbbm{R}} \subset \mathcal{M}$ be a
  smooth family such that $\dot{g}_t \in \mathbbm{F}_{g_t} \cap
  \mathbbm{E}_{g_t}$ for all $t \in \mathbbm{R}$. If $g_0 \in
  \mathcal{S}_{_{^{\Omega}}}$ then $g_t \in \mathcal{S}_{_{^{\Omega}}}$ for
  all $t \in \mathbbm{R}$.
\end{corollary}

\begin{proof}
  We observe that the assumption $\dot{g}_t \in \mathbbm{E}_{g_t}$ implies in
  particular the identity $\mathcal{R}_{g_t} \circledast \nabla_{g_t}  \dot{g}^{\ast}_t \equiv
  0$. By lemma \ref{F-var-pre-scat} we infer the variation formula
  \begin{eqnarray*}
    \frac{d}{d t}  \left[ \nabla_{_{T_X, g_t}} \tmop{Ric}^{\ast}_{g_t}
    (\Omega) \right] \;\; = \;\; -\;\, \dot{g}_t^{\ast} \,\nabla_{_{T_X, g_t}}
    \tmop{Ric}^{\ast}_{g_t} (\Omega)\;,
  \end{eqnarray*}
  and thus the conclusion by Cauchy uniqueness.
\end{proof}

The total variation of the pre-scattering operator is given in lemma
\ref{var-pre-scat} in the appendix. It provides in particular an alternative
proof of the conservation of the pre-scattering condition.

\subsection{Higher order conservative differential symmetries}

In this sub-section we will show that some higher order differential symmetries are
conserved along the geodesics. This type of higher order differential
symmetries is needed in order to stabilize the scattering conditions with
respect to the variations produced by the SRF. We observe first that given any
diagonal $n \times n$-matrix $\Lambda$ it hold the identity
\begin{eqnarray*}
  \left[ \Lambda, M \right] \;\;=\;\; \left( (\lambda_i - \lambda_j) M_{i, j}
  \right)\;,
\end{eqnarray*}
for any other $n \times n$-matrix $M$. Thus if the values $\lambda_j$ are all
distinct then $\left[ \Lambda, M \right] = 0$ if and only if $M$ is also a
diagonal matrix.

In this sub-section and in the sections \ref{sec-scat-dat}, \ref{Integ-F} that will follow 
we will always denote by $K \in \Gamma (X, \tmop{End} (T_X))$
an element with point wise $n$-distinct real eigenvalues, where 
$n =\dim_{_{\mathbbm{R}}} X$. 
\\
The previous remark shows that if $(e_k) \subset
T_{X, p}$ is a basis diagonalizing $K (p)$ then it diagonalizes any element $M
\in \tmop{End} (T_{X, p})$ such that $\left[ K (p), M \right] = 0$. 
\\
We deduce
that if also $N \in \tmop{End} (T_{X, p})$ satisfies $\left[ K (p), N \right]
= 0$ then $\left[ M, N \right] = 0$.
\\
We define now the vector space inside $\mathbbm{F}_g^{\infty}$
\begin{eqnarray*}
  \mathbbm{F}_g (K) \;\; \assign \;\; \Big\{ v \in \mathbbm{F}_g \mid
  \hspace{0.25em}  \left[ K, v^{\ast}_g \right] \;=\; 0\,, \;\left[ K, \nabla_g\,
  v^{\ast}_g  \right] \;=\; 0 \Big\} \;\;\subset\;\; \mathbbm{F}_g^{\infty}\;	 .
\end{eqnarray*}
We observe in fact the definition implies $\left[ \nabla_g \,v^{\ast}_g,
v^{\ast}_g \right] = 0$, and thus the last inclusion. With this notations hold
the following corollary analogue to lemma \ref{invar-F}.

\begin{corollary}
  \label{invar-FK}Let $(g_t)_{t \in \mathbbm{R}}$ be a geodesic such that
  $\dot{g}_0 \in \mathbbm{F}_{g_0} (K) .$ Then $\dot{g}_t \in
  \mathbbm{F}_{g_t} (K)$ for all $t \in \mathbbm{R}.$
\end{corollary}

\begin{proof}
  By lemma \ref{invar-F} we just need to show the identity $\left[ K,
  \nabla_{g_t} \, \dot{g}^{\ast}_t \right] \equiv 0$. In fact using the variation formula
  (\ref{F-var-LC}) we obtain
  \begin{eqnarray*}
    2 \,\frac{d}{d t} \, \left[ K, \nabla_{g_t} \, \dot{g}^{\ast}_t \right] \;\; = \;\;
    \Big[ K, \left[ \nabla_{g_t} \,
    \dot{g}^{\ast}_t, \dot{g}^{\ast}_t \right] \Big]\;\; =\;\; -\;\, \Big[
    \dot{g}^{\ast}_t, \left[ K, \nabla_{g_t} \, \dot{g}^{\ast}_t \right]
    \Big]\;,
  \end{eqnarray*}
  since $\left[ K, \dot{g}^{\ast}_t \right] \equiv 0$. Then the conclusion
  follows by Cauchy uniqueness.
\end{proof}

We define now the sub-vector space $\mathbbm{F}^K_g\subset \mathbbm{F}_g (K)$,
\begin{eqnarray*}
  \mathbbm{F}^K_g & \assign & \left \{ v \in \mathbbm{F}_g \mid
  \hspace{0.25em}  \left[ T, \nabla^p_{g, \xi}\, v^{\ast}_g \right] \;=\; 0\,,\; T \;=\; K\,,
  \mathcal{R}_g\,,\; \forall \xi \in T^{\otimes p}_X,\; \forall p \in
  \mathbbm{Z}_{\geqslant 0} \right\} \;,
\end{eqnarray*}
and we show the following elementary lemmas

\begin{lemma}
  \label{mut-exp}If $u, v \in \mathbbm{F}^K_g$ then $u\, v^{\ast}_g, u\,
  e^{v^{\ast}_g} \in \mathbbm{F}^K_g$.
\end{lemma}

\begin{proof}
  By assumption follows that $u^{\ast}_g$ commutes with $v^{\ast}_g$. This
  shows that $u\, v^{\ast}_g$ is a symmetric form. Again by assumption we infer
  $[\nabla_g \,u^{\ast}_g, v^{\ast}_g] = [\nabla_g \,v^{\ast}_g, u^{\ast}_g] = 0$,
  and thus $u^{\ast}_g \,v^{\ast}_g \in \mathbbm{F}_g$. We observe now that for
  any $A, B, C \in \tmop{End} (V)$ such that $\left[ C, A \right] = 0$ hold
  the identity
  \begin{equation}
    \label{triv-com} \left[ C, A\, B \right] \;\;=\;\; A \left[ C, B \right] \;.
  \end{equation}
  Thus if also $\left[ C, B \right] = 0$ then hold the identity
  \begin{equation}
    \label{nul-triv-com} \left[ C, A\, B \right] \;\;=\;\; 0\; .
  \end{equation}
  Applying (\ref{nul-triv-com}) with $C = T$ and with $A = \nabla^r_{g, \eta}\,
  u^{\ast}_g$, $\eta \in T^{\otimes r}_X$, $B = \nabla^{p - r}_{g, \mu}\,
  v^{\ast}_g$, $\mu \in T^{\otimes p - r}_X$ we infer the identity 
$$
[T,
  \nabla^p_{g, \xi} (u^{\ast}_g v^{\ast}_g)] \;\;=\;\; 0\;,
$$
thus the conclusion $u\,
  v^{\ast}_g \in \mathbbm{F}^K_g$. The fact $u \,e^{v^{\ast}_g} \in
  \mathbbm{F}^K_g$ follow directly from the previous one.
\end{proof}

For all $\xi \equiv (\xi_1, \ldots, \xi_p) \in C^{\infty} (X, T_X)^{\oplus p}$
we denote 
$$
\nabla^{(p)}_{g, \xi} \,v^{\ast}_g \;\;\assign\;\; \nabla_{g, \xi_1} \ldots
\nabla_{g, \xi_p} v^{\ast}_g\;,
$$ 
and we observe that a simple induction based on
the formula
\begin{eqnarray*}
  \nabla^p_{g, \xi} \,v^{\ast}_g & = & \nabla^{(p)}_{g, \xi} \,v^{\ast}_g \;\,-\;\,
  \sum_{r = 1}^{p - 1}\, \sum_{I \in J^{p - 1}_{p - r}} \varepsilon_I  \left(
  \nabla_{g, \xi_I}\, \xi_{\complement I} \right) \;\neg\; \nabla^r_g \,v^{\ast}_g\;,
\end{eqnarray*}
with $J^{p - 1}_{p - r} \assign \left\{ I \subset \left\{ 1, \ldots, p - 1
\right\} : | I| = p - r \right\}$, $\varepsilon_I = 0, 1$, $\complement I
\equiv (j_1, \ldots j_r)\assign \left\{ 1, \ldots, p \right\} \smallsetminus
I$ and with
\begin{eqnarray*}
  \nabla_{g, \xi_I}\, \xi_{\complement I} \;\;\assign\;\; \nabla_{g, \xi_{i_1}}
  \ldots \nabla_{g, \xi_{i_{p - r}}} \left( \xi_{j_1} \otimes \cdots \otimes
  \xi_{j_r} \right)\;,
\end{eqnarray*}
$I \equiv (i_1, \ldots, i_{p - r})$, shows the identity
\begin{eqnarray*}
  \mathbbm{F}^K_g \; = \; \left \{ v \in \mathbbm{F}_g \mid \hspace{0.25em} 
  \left[ T, \nabla^{(p)}_{g, \xi} \,v^{\ast}_g \right] = 0,\; T = K,
  \mathcal{R}_g, \forall \xi \in C^{\infty} (X, T_X)^{\oplus p}, \forall p \in
  \mathbbm{Z}_{\geqslant 0} \right\} .
\end{eqnarray*}
\begin{lemma}
  \label{invar-F-K}Let $(g_t)_{t \in \mathbbm{R}}$ be a geodesic such that
  $\dot{g}_0 \in \mathbbm{F}^K_{g_0} .$ Then $\dot{g}_t \in
  \mathbbm{F}^K_{g_t}$ for all $t \in \mathbbm{R}$.
\end{lemma}

\begin{proof}
We observe that $\dot{g}_t \in
  \mathbbm{F}_{g_t}(K)\cap \mathbbm{E}_{g_t}$ for all $t \in \mathbbm{R}$, thanks to corollary \ref{invar-FK} and lemma
  \ref{invar-Rm}. This implies in particular that $\mathcal{R}_{g_t} =\mathcal{R}_{g_0}$ for
  all $t \in \mathbbm{R}$, by the variation formula (\ref{col-vr-Rm}). Then the conclusion will follow from the property
\begin{equation}
  \label{pcov-cst-Geo} 
\nabla^{(p)}_{g_0, \xi} \,\dot{g}^{\ast}_0 \;\;\equiv\;\; \nabla^{(p)}_{g_t, \xi} \,\dot{g}^{\ast}_t\,,
\quad \forall \xi \in C^{\infty} (X, T_X)^{\oplus p},\quad \forall t\;\in
\; \mathbbm{R}\;.
\end{equation}
This certainly hold true for $p=0$ by the geodesic equation $\dot{g}^{\ast}_0\equiv\dot{g}^{\ast}_t$.
We assume now the statement (\ref{pcov-cst-Geo}) true for $p - 1$ and we show it for $p >0$. We observe first that thanks to the variation formula (\ref{F-var-LC})
  hold the identity
  \begin{equation}
    \label{flat-vr-LC}  \dot{\nabla}_{g_t} H \;\;\equiv\;\; 0\;,
  \end{equation}
  along any smooth curve $(g_t)_t \subset \mathcal{M}$ such that $\dot{g}_t
  \in \mathbbm{F}_g$ and $\left[ \nabla_{g_t} \, \dot{g}^{\ast}_t, H \right] =
  0$. In our situation the identity $\left[ K,\nabla_{g_t} \, \dot{g}^{\ast}_t \right] \equiv
  0$ combined with the assumption on the initial data
  \begin{eqnarray*}
    \left[ K, \nabla^{(p-1)}_{g_0, \xi} \, \dot{g}^{\ast}_0 \right] \;\; = \;\; 0\;,
  \end{eqnarray*}
  for all $\xi \in C^{\infty} (X,T_X)^{\oplus (p-1)}$, implies thanks to (\ref{flat-vr-LC}) the equalities
  \begin{eqnarray*}
\nabla_{g_0}\nabla^{(p-1)}_{g_0, \xi} \,
    \dot{g}^{\ast}_0 \;\; = \;\; \nabla_{g_t}\nabla^{(p-1)}_{g_0, \xi} \,
    \dot{g}^{\ast}_0 \;\; \equiv\;\;\nabla_{g_t}\nabla^{(p-1)}_{g_t, \xi} \,
    \dot{g}^{\ast}_t \;,
  \end{eqnarray*}
by the inductive hypothesis. We infer the conclusion of the induction.
\end{proof}

\section{The set of scattering data}\label{sec-scat-dat}

We define the {\tmstrong{set of scattering data}} with center $K$ as the set
of metrics
\begin{eqnarray*}
  \mathcal{S}_{_{^{\Omega}}}^K & \assign & \left \{ g \in \mathcal{M} \mid
  \tmop{Ric}_g (\Omega) \in \mathbbm{F}^K_g \right\}\\
  &  & \\
  & = & \left \{ g \in \mathcal{S}_{_{^{\Omega}}} \mid 
  \left[ T, \nabla^p_{g, \xi} \tmop{Ric}^{\ast}_g (\Omega) \right] = 0\,,\; T = K, \mathcal{R}_g, \forall \xi \in T^{\otimes p}_X, \forall p \in
  \mathbbm{Z}_{\geqslant 0} \right\} .
\end{eqnarray*}
We observe that $\mathcal{S}_{_{^{\Omega}}}^K \neq \emptyset$ if the manifold
$X$ admit a $\Omega$-$\tmop{ShRS}$. We introduce now a few new product
notations. For any $A \in (V^{\ast})^{\otimes p} \otimes V$, $B \in
(V^{\ast})^{\otimes q} \otimes V$ and for any $k = 1, \ldots, q$ we define the
products $A \bullet_k B$ as
\begin{eqnarray*}
  (A \bullet_k B) (u, v) & \assign & B \Big(v_1, \ldots, v_{k - 1}, A (u, v_k),
  v_{k + 1}, \ldots, v_q\Big)\;,
\end{eqnarray*}
for all $u \equiv (u_1, \ldots, u_{p - 1})$ and $v \equiv (v_1, \ldots, v_q)$.
We note $\bullet \assign \bullet_1$ for simplicity. For any $\sigma \in S_{p +
k - 2}$ we define $A \bullet^{\sigma}_k B$ as
\begin{eqnarray*}
  (A \bullet^{\sigma}_k B) (u, v) & \assign & (A \bullet_k B) (\xi_{\sigma},
  v_k, \ldots, v_q)\;,
\end{eqnarray*}
where $\xi \equiv (\xi_1, \ldots, \xi_{p + k - 2}) \assign (u_1, \ldots, u_{p
- 1}, v_1, \ldots, v_{k - 1})$. We should notice that $A \bullet^{\sigma}_k B
\equiv A \bullet_k B$ if $p + k - 2 \leqslant 1$. We define
\begin{eqnarray*}
  A \;\hat{\neg}\; B & \assign & \sum^{q - 1}_{k = 1} A \bullet_k B \;.
\end{eqnarray*}
For $p > 1$ and \ $k = 1, \ldots p - 1$ we define the trace operation
\begin{eqnarray*}
  (\tmop{Tr}_{g, k} A) \left( u_1, \ldots, u_{p - 2} \right) & \assign &
  \tmop{Tr}_g \Big[ A \left( u_1, \ldots, u_{k - 1}, \cdot, \cdot, u_k,
  \ldots, u_{p - 2} \right) \Big] \;.
\end{eqnarray*}
For any $v \in T_X$ and $k = 1, \ldots, p$ we define the contraction operation
\begin{eqnarray*}
  ( v \;\neg_k\; A) \left( u_1, \ldots, u_{p - 1} \right)  & \assign &
  A \left( u_1, \ldots, u_{k - 1}, v, u_k, \ldots, u_{p - 1} \right) \;.
\end{eqnarray*}
For any $B \in (V^{\ast})^{\otimes q} \otimes V$ and $k = 1, \ldots, q
- 1$ we define the generalized type products
\begin{eqnarray*}
&&(A \ast_k B) \left( u_1, \ldots, u_p, v_1, \ldots, v_{q - 1} \right) 
\\
\\
&
  \assign & B \Big( v_1, \ldots, v_{k - 1}, A (u_1, \ldots, u_p), v_k,
  \ldots, v_{q - 1} \Big)\;,\\
   \\
  &&(A \ast^{\sigma}_k B) \left( u_1, \ldots, u_p, v_1, \ldots, v_{q - 1}
  \right) 
\\
\\
& \assign & (A \ast_k B) \left( \xi_{\sigma}, v_k, \ldots, v_{q - 1}
  \right)\;,\quad \forall \sigma \in S_{p + k - 1}\;,
\end{eqnarray*}
and $\xi \equiv (\xi_1, \ldots, \xi_{p + k - 1}) :=(u_1, \ldots, u_p,
  v_1, \ldots, v_{k - 1})$.
We observe now that if $A \in C^{\infty} (X, (T^{\ast}_X)^{\otimes p} \otimes
T_X)$ and $(g_t)_{t \in \mathbbm{R}} \subset \mathcal{M}$ is a smooth family
such that $[\nabla_{g_t, \xi} \; \dot{g}^{\ast}_t, A] \equiv 0$ for all $\xi \in
T_X$ then
\begin{eqnarray*}
  2\, \dot{\nabla}_{g_t} A & = & - \;\,\nabla_{g_t} \, \dot{g}^{\ast}_t  \;\hat{\neg}\; A\;,
\end{eqnarray*}
thanks to the variation formula (\ref{Pal-var-LC-conn}). We infer by this and
by the variation formula (\ref{flat-vr-LC}) that if $H \in C^{\infty} (X,
\tmop{End} (T_X))$ satisfies $[\nabla_{g_t, \xi} \, \dot{g}^{\ast}_t,
\nabla^p_{g_t} H] \equiv 0$, for all $\xi \in T_X$ and $p = 0, 1$ then
\begin{eqnarray*}
  2 \,\dot{\nabla}^2_{g_t} H & = & - \;\,\nabla_{g_t} \, \dot{g}^{\ast}_t  \;\hat{\neg}\;
  \nabla_{g_t} H \;.
\end{eqnarray*}
A simple induction shows that if $[\nabla_{g_t, \xi} \, \dot{g}^{\ast}_t,
\nabla^r_{g_t} H] \equiv 0$, for all $\xi \in T_X$ and $r = 0, \ldots p - 1$
then
\begin{equation}
  \label{var-Pder} 2\, \dot{\nabla}^p_{g_t} H \;\;=\;\; -\;\, \sum_{r = 1}^{p - 1}\,  \sum_{k
  = 1}^r \,\sum_{\sigma \in S_{p - r + k - 1}} C_{k, \sigma}^{p, r}
  \nabla_{g_t}^{p - r}  \dot{g}^{\ast}_t \bullet^{\sigma}_k \nabla_{g_t}^r H\;,
\end{equation}
with $C_{k, \sigma}^{p, r} = 0, 1$. We show now the following fundamental result.

\begin{corollary}
  \label{FK-vr-Om-EndRc}Let $(g_t)_{t \in \mathbbm{R}} \subset \mathcal{M}$ be
  a smooth family such that $\dot{g}_t \in \mathbbm{F}^K_{g_t}$ for all $t \in
  \mathbbm{R}$. If $g_0 \in \mathcal{S}^K_{_{^{\Omega}}}$ then $g_t \in
  \mathcal{S}^K_{_{^{\Omega}}}$ for all $t \in \mathbbm{R}$.
\end{corollary}

\begin{proof}
  Thanks to corollary \ref{vr-Om-EndRc} we just need to show the condition on
  the brackets. We will proceed by induction on the order of covariant differentiation $p$. For notations simplicity
  we set $\rho_t \assign \tmop{Ric}_{g_t} \left( \Omega \right)$. We observe
  now that the assumption implies $\mathcal{R}_{g_t} =\mathcal{R}_{g_0}$ for
  all $t \in \mathbbm{R}$, thanks to the variation formula (\ref{col-vr-Rm}).
  This combined with the variation formula (\ref{col-vrEnOmRc}) gives
  \begin{eqnarray*}
    2\, \frac{d}{d t}\,  \left[ T, \rho^{\ast}_t \right] &=& -\;\, \left[ T,
    \Delta^{^{_{_{\Omega}}}}_{g_t}  \dot{g}^{\ast}_t \;\,+\;\, 2\, \dot{g}^{\ast}_t
    \rho^{\ast}_t \right] 
\\
\\
&=& -\;\, 2\, \left[ T, \dot{g}^{\ast}_t \rho^{\ast}_t
    \right] 
\\
\\
&=& -\;\, 2\, \dot{g}^{\ast}_t \left[ T, \rho^{\ast}_t \right]\;,
  \end{eqnarray*}
  thanks to the assumption $[T, \nabla^p_{g_t, \xi} \, \dot{g}^{\ast}_t] \equiv
  0$ and (\ref{triv-com}). By Cauchy uniqueness we infer $\left[ T,
  \rho^{\ast}_t \right] \equiv 0$. We assume now as inductive hypothesis
  $\left[ T, \nabla^r_{g_t, \eta}\, \rho^{\ast}_t \right] \equiv 0$, for all $r
  = 0, \ldots, p - 1$, $p > 1$ and $\eta \in T^{\otimes r}_X$. We
  deduce thanks to this for $T = K$ and thanks to the assumption, the identity
  \begin{eqnarray*}
    \left[ \nabla_{g_t} \, \dot{g}^{\ast}_t, \nabla^r_{g_t, \eta} \,\rho^{\ast}_t
    \right] & = & 0\;,
  \end{eqnarray*}
  This combined with the variation formula (\ref{var-Pder}) with $H =
  \rho_t^{\ast}$, combined with the inductive hypothesis and with
  (\ref{col-vrEnOmRc}) provides the identity
  \begin{eqnarray*}
    2\, \frac{d}{d t}  \left[ T, \nabla^p_{g_t, \xi} \,\rho^{\ast}_t \right] & = &
    -\;\, \left[ T, \nabla^p_{g_t, \xi}  \left( \Delta^{^{_{_{\Omega}}}}_{g_t} 
    \dot{g}^{\ast}_t \;\,+\;\, 2\, \dot{g}^{\ast}_t \rho^{\ast}_t \right) \right] \;.
  \end{eqnarray*}
  We can express the $p$-derivative of the $\Omega$-Laplacian as
  \begin{eqnarray}
    \nabla^p_{g_t} \Delta^{^{_{_{\Omega}}}}_{g_t}  \dot{g}^{\ast}_t & = & -\;\,
    \tmop{Tr}_{g_t, p + 1} \nabla_{g_t}^{p + 2}  \dot{g}^{\ast}_t \;\,+\;\,
    \nabla_{g_t} f_t \;\neg_{p + 1} \nabla^{p + 1}_{g_t}  \dot{g}^{\ast}_t 
    \nonumber\\
    &  &  \nonumber\\
    & + & \sum_{r = 1}^p \sum_{\sigma \in S_{p + 1}} C_{\sigma}^{p, r}\,
    \nabla^{p + 2 - r}_{g_t} f_t \ast^{\sigma}_r \,\nabla^r_{g_t} 
    \dot{g}^{\ast}_t\;, \label{p-derLap} 
  \end{eqnarray}
  with $C_{\sigma}^{p, r} = 0, 1$. Moreover the assumption combined with the
  expression of the $p$-derivative of the $\Omega$-Laplacian and with the
  identity
  \begin{eqnarray*}
    \left[ T, \xi \;\neg\; \left( \nabla^{p + 2 - r}_{g_t} f_t \ast^{\sigma}_r
    \nabla^r_{g_t} \, \dot{g}^{\ast}_t \right) \right] & \equiv & 0\;,
  \end{eqnarray*}
  implies the variation formula
  \begin{eqnarray*}
    2\, \frac{d}{d t}  \left[ T, \nabla^p_{g_t, \xi} \,\rho^{\ast}_t \right] & = &
    - \;\,2 \left[ T, \nabla^p_{g_t, \xi}  \left( \dot{g}^{\ast}_t \rho^{\ast}_t
    \right) \right] 
\\
\\
&=&- \;\,2\sum_{r = 0}^p\,
  \sum_{| I| = r} \left[ T,\nabla^{p - r}_{g_0, \xi_{\complement I}} \dot{g}^{\ast}_t\,
   \nabla^{r}_{g_0, \xi_I} \rho^{\ast}_t \right] \;.
  \end{eqnarray*}
  Using the assumption $[T, \nabla^{p - r}_{g_t, \mu} \, \dot{g}^{\ast}_t]
  \equiv 0$, $\mu \in T^{\otimes p - r}_X$ and the inductive hypothesis we can
  apply the identity (\ref{nul-triv-com}) to the products of type 
$\nabla^{p - r}_{g_0, \xi_{\complement I}} \dot{g}^{\ast}_t\,
   \nabla^{r}_{g_0, \xi_I} \rho^{\ast}_t$
in order
  to obtain the identity
  \begin{eqnarray*}
    2\, \frac{d}{d t}  \left[ T, \nabla^p_{g_t, \xi} \,\rho^{\ast}_t \right] & = &
    - \;\,2\, \dot{g}^{\ast}_t \left[ T, \nabla^p_{g_t, \xi} \,\rho^{\ast}_t \right] \;.
  \end{eqnarray*}
  Then the conclusion follows by Cauchy uniqueness.
\end{proof}

\section{Integrability of the distribution $\mathbbm{F}^K$}\label{Integ-F}

We start first with a basic calculus fact.

\begin{lemma}
  \label{cov-der-log}Let $B > 0$ be a $g$-symmetric endomorphism smooth section of $T_X$ such
  that $\left[ B, \nabla_{g, \xi} B \right] = 0$. Then hold the identity
  \begin{eqnarray*}
    \nabla_{g, \xi} \log B & = & B^{- 1} \nabla_{g, \xi} B \;.
  \end{eqnarray*}
\end{lemma}

\begin{proof}
  We set $A \assign \log B$ and we observe that by definition $\left[ B, A
  \right] = 0$, i.e $\left[ e^A, A \right] = 0$. The assumption $\left[ e^A,
  \nabla_{g, \xi} \,e^A \right] = 0$ is equivalent to the condition $\left[ A,
  \nabla_{g, \xi} \,e^A \right] = 0$ since the endomorphisms $e^A, \nabla_{g,
  \xi} \,e^A$ can be diagonalized simultaneously. Thus deriving the identity
  $\left[ e^A, A \right] = 0$ we infer $\left[ \nabla_{g, \xi} \,A, e^A \right]
  = 0$, which is equivalent to $\left[ \nabla_{g, \xi} \,A, A \right] = 0$. But
  this last implies
  \begin{eqnarray*}
    \nabla_{g, \xi}\, e^A \;\; = \;\; \nabla_{g, \xi} \,A \,e^A \;\;=\;\; e^A \nabla_{g, \xi}\, A\; .
  \end{eqnarray*}
  We infer $\nabla_{g, \xi} \,A = e^{- A} \nabla_{g, \xi} \,e^A$, i.e the required
  conclusion.
\end{proof}

We show now the following key lemma.

\begin{lemma}
  \label{flat-Scat-Space}For any $g_0 \in \mathcal{M}$ hold the identities
  \begin{eqnarray}
    \label{Iflat-Scat-Space} \Sigma_K (g_0) &\assign& \mathbbm{F}^K_{g_0} \cap
    \mathcal{M}\;\;=\;\; \exp_{G, g_0} \left( \mathbbm{F}^K_{g_0} \right)\;,
\\\nonumber
\\
    \label{2flat-Scat-Space} T_{\Sigma_{^{_K}} (g_0), g} &=&\mathbbm{F}^K_g\,,\quad
    \forall g \in \Sigma_K (g_0) \;.
  \end{eqnarray}
  Moreover $\Sigma_K (g_0)$ is a totally geodesic and flat sub-variety inside the
  non-positively curved Riemannian manifold $(\mathcal{M}, G) .$
\end{lemma}

\begin{proof}
  {\tmstrong{Step I (A)}}
  
  We observe first the inclusion $\Sigma_K (g_0) \supseteq \exp_{G, g_0}
  \left( \mathbbm{F}^K_{g_0} \right) .$ In fact let $(g_t)_{t \in
  \mathbbm{R}}$ be a geodesic such that $\dot{g}_0 \in \mathbbm{F}_{g_0}^K$.
  Then using the expression (\ref{expr-geod}) of the geodesics we obtain
  \begin{eqnarray*}
    \nabla_{_{T_X, g_0}} (g^{- 1}_0 g_t) \;\;=\;\; \nabla_{_{T_X, g_0}} e^{t
    \dot{g}^{\ast}_0} \;\;=\;\; 0\,,\quad \forall t \in \mathbbm{R}\;,
  \end{eqnarray*}
  since the last equality is equivalent to the condition $\dot{g}_0 \in
  \mathbbm{F}_{g_0}^{\infty}$. Moreover the condition
  \begin{eqnarray*}
    \left[ T, \nabla^p_{g_0, \xi} \,e^{t \dot{g}^{\ast}_0} \right] \;\;=\;\; 0\,,\quad
    \forall t \in \mathbbm{R}\,,\; \forall p \in \mathbbm{Z}_{\geqslant 0}\;,
  \end{eqnarray*}
  is equivalent to the identity
$$
\left[ T, \nabla^p_{g_0, \xi} ( \dot{g}^{\ast}_0)^q
  \right] \;\;=\;\; 0\;,
$$
 for all $p, q \in \mathbbm{Z}_{\geqslant 0}$ and this last
  follows from a repetitive use of the identity (\ref{nul-triv-com}). We
  conclude the inclusion $\Sigma_K (g_0) \supseteq \exp_{G, g_0} \left(
  \mathbbm{F}^K_{g_0} \right) .$
  
  {\tmstrong{Step I (B1)}}
  
  We observe now that for any smooth curve $(u_t)_{t \in (- \varepsilon,
  \varepsilon)} \in \mathbbm{F}^K_{g_0}$ the identity $[K, g^{- 1}_0 u_t] = 0$
  implies $[K, g^{- 1}_0 \dot{u}_t] = 0$ and thus $[g^{- 1}_0 u_t, g^{- 1}_0
  \dot{u}_t] = 0$. We infer that the differential 
$$
D_u \exp_{G, g_0} :
  \mathbbm{F}^K_{g_0} \longrightarrow \mathbbm{F}^K_{g_0}\;,
$$ of the exponential map
  $\exp_{G, g_0} : \mathbbm{F}^K_{g_0} \longrightarrow \Sigma_K (g_0)$ at a point
  $u \in \mathbbm{F}^K_{g_0}$ is given by the formula
  \begin{eqnarray*}
    D_u \exp_{G, g_0} (v) \;\; = \;\; v\, e^{g^{- 1}_0 u}\;,
  \end{eqnarray*}
  for all $v \in \mathbbm{F}^K_{g_0}$. We deduce by lemma \ref{mut-exp} that
  this differential map is an isomorphism. This combined with the fact that
  the exponential map is injective implies that $\exp_{G, g_0}
  (\mathbbm{F}^K_{g_0})$ is an open subset of $\Sigma_K (g_0)$. But $\exp_{G,
  g_0} (\mathbbm{F}^K_{g_0})$ is also a closed set of $\mathcal{M}$ and thus a
  closed set of $\Sigma_K (g_0)$. The fact that this last is connected implies
  the required equality (\ref{Iflat-Scat-Space}).
  
  {\tmstrong{Step I (B2)}}
  
  We give now an explicit proof of the inclusion $\Sigma_K (g_0) \subseteq
  \exp_{G, g_0} \left( \mathbbm{F}^K_{g_0} \right) .$ (This is useful also
  for other considerations.) Indeed we show that if $g \in \Sigma_K (g_0)$
  then $g_0 \log (g_0^{- 1} g) \in \mathbbm{F}^K_{g_0}$. We observe first that
  the assumptions $[K, g^{- 1}_0 g] = 0$, and $[K, \nabla_{g_0} (g^{- 1}_0 g)]
  = 0$, imply
  \begin{eqnarray*}
    \left[ g^{- 1}_0 g, \nabla_{g_0, \xi} (g^{- 1}_0 g) \right] \;\; = \;\; 0\;,
  \end{eqnarray*}
  which allows to apply lemma \ref{cov-der-log} in order to obtain the formula
  \begin{equation}
    \label{der-logG} \nabla_{g_0, \xi} \log (g^{- 1}_0 g) \;\;=\;\; g^{- 1} g_0
    \nabla_{g_0, \xi} (g^{- 1}_0 g) \;.
  \end{equation}
  Thus the assumption $\nabla_{_{T_{X, g_0}}} (g_0^{- 1} g) = 0$ implies the
  identity
$$
\nabla_{_{T_{X, g_0}}} \log (g_0^{- 1} g) \;\;=\;\; 0\; . 
$$
  Let $\varepsilon > 0$ be sufficiently small such that $\varepsilon g < 2
  g_0$. Then the expansion
  \begin{eqnarray*}
    \log (\varepsilon\, g_0^{- 1} g) \;\; =\;\; \sum_{p = 1}^{+ \infty}  \frac{(-
    1)^{p + 1}}{p}  \left( \varepsilon\, g_0^{- 1} g \;\,-\;\,\mathbbm{I} \right)^p\,,
  \end{eqnarray*}
  implies $[T, \log (\varepsilon g_0^{- 1} g)] = 0$, and thus $[T, \log
  (g_0^{- 1} g)] = 0$. Moreover the identity (\ref{triv-com}) implies
  \begin{eqnarray*}
    g^{- 1}_0 g \left[ T, g^{- 1} g_0 \right] \;\;=\;\; \left[ T, \mathbbm{I}
    \right] \;\;=\;\; 0\,,
  \end{eqnarray*}
  and thus $[T, g^{- 1} g_0] = 0$ which combined with the formula
  \begin{eqnarray*}
    \nabla_{g_0, \xi} (g^{- 1} g_0)\;\;=\;\; - g^{- 1} g_0 \,\nabla_{g_0, \xi}
    (g^{- 1}_0 g) \,g^{- 1} g_0\;,
  \end{eqnarray*}
  and with the identity (\ref{nul-triv-com}) implies the equality $[T, \nabla_{g_0} (g^{-
  1}_0 g)] = 0$. We infer $[T, \nabla^p_{g_0, \xi} (g^{- 1}_0 g)] = 0$ by a
  simple induction based on a repetitive use of the identity
  (\ref{nul-triv-com}). We conclude
  \begin{eqnarray*}
    \left[ T, \nabla^p_{g_0, \xi} \log (g^{- 1}_0 g) \right] \;\; = \;\; 0\;,
  \end{eqnarray*}
  by deriving the identity (\ref{der-logG}) and using (\ref{nul-triv-com}).
  
  {\tmstrong{Step II}}
  
We show now the identity (\ref{2flat-Scat-Space}), i.e the identity $\mathbbm{F}^K_{g_0}=\mathbbm{F}^K_g$. We can consider,
  thanks to the equality (\ref{Iflat-Scat-Space}), a geodesic $(g_t)_{t \in
  \mathbbm{R}} \subset \Sigma_K (g_0)$ joining $g = g_1$ with $g_0$. We
  observe also that lemma \ref{invar-F-K} combined with the variation formula
  (\ref{col-vr-Rm}) implies the identity $\mathcal{R}_{g_t} \equiv \mathcal{R}_{g_0}
  =\mathcal{R}_g$. 
Moreover 
$\left[ K,\nabla_{g_t} \, \dot{g}^{\ast}_t \right] \equiv 0$, thanks to lemma \ref{invar-F-K}.
Then the variation formula (\ref{flat-vr-LC}) implies
\begin{eqnarray}
  \nabla_g H \;\; = \;\; \nabla_{g_0} H\,,\quad \forall H \in C^{\infty} (X, \tmop{End}
  (T_X)) \;:\; [K, H] \;\;=\;\; 0\;.  \label{inv-cov}
\end{eqnarray}
On the other hand using (\ref{triv-com}) we
obtain the equalities
\begin{eqnarray*}
   \left[ T, g^{- 1}_0 v \right] \;\;=\;\; \left[ T, (g^{- 1}_0 g) (g^{- 1} v)
  \right]\;\; =\;\; (g^{- 1}_0 g) \left[ T, g^{- 1} v \right]\; .
\end{eqnarray*}
Thus $[T, g^{- 1}_0 v] = 0$ iff $[T, g^{- 1} v] = 0$. This last for $T = K$
implies
\begin{equation}
  \label{cov-cstH} \nabla_g (g^{- 1} v) \;\;=\;\; \nabla_{g_0} (g^{- 1} v)\;,
\end{equation}
thanks to (\ref{inv-cov}). We consider the identities
  \begin{eqnarray*}
\nabla_{_{T_{X, g_0}}} (g_0^{- 1}  v)
 & = & \nabla_{_{T_{X, g_0}}}  \big[ (g^{- 1}_0 g) (g^{- 1}  v)\big]\\
    &  & \\
    & = & g^{- 1}  v \,\nabla_{_{T_{X, g_0}}} (g^{- 1}_0 g) \;\,+\;\, g^{- 1}_0
    g\, \nabla_{_{T_{X, g_0}}} (g^{- 1}  v)\\
    &  & \\
    & = & g^{- 1}_0 g \,\nabla_{_{T_{X, g}}} (g^{- 1}  v)\;,
  \end{eqnarray*}
since $g \in \Sigma_K (g_0)$ and thanks to (\ref{cov-cstH}). We deduce $v \in \mathbbm{F}_{g_0}$ iff $v \in \mathbbm{F}_g$ provided that $[K, g^{- 1} v] = 0$.
We show now by induction on $p \geqslant 0$ the
properties
\begin{equation}
  \label{ind-com-cov}  \left[ T, \nabla^{(r)}_{g_0, \xi} (g_0^{- 1} v) \right]
  \;=\; 0\; \Longleftrightarrow \;\left[ T, \nabla^{(r)}_{g, \xi} (g^{- 1} v) \right]
  \;=\; 0\,, \;\forall \xi \in C^{\infty} (X, T_X)^{\oplus r},
\end{equation}
and
\begin{equation}
  \label{pcov-cst-vrH} \nabla^{(r + 1)}_{g, \xi} (g^{- 1} v) \;\;=\;\; \nabla^{(r
  + 1)}_{g_0, \xi} (g^{- 1} v)\,,\quad \forall \xi \in C^{\infty} (X, T_X)^{\oplus (r
  + 1)}\,,
\end{equation}
for all $r = 0, \ldots, p$. This properties hold true for $p = 0$ as we
observed previously. The assumption $g \in \Sigma_K (g_0)$ combined with
(\ref{triv-com}) implies the identity.
\begin{eqnarray*}
  \left[ T, \nabla^{(p)}_{g_0, \xi} (g_0^{- 1} v) \right] \;\; = \;\; \sum_{r = 0}^p\,
  \sum_{| I| = r} \nabla^{(p - r)}_{g_0, \xi_{\complement I}} (g^{- 1}_0 g)
  \left[ T, \nabla^{(r)}_{g_0, \xi_I} (g^{- 1} v) \right] \;.
\end{eqnarray*}
We assume that the step $p - 1$ of the induction hold true. We infer the
equality
\begin{eqnarray*}
  &  & \left[ T, \nabla^{(p)}_{g_0, \xi} (g_0^{- 1} v) \right]\;\;=\;\; g_0^{- 1} g
  \left[ T, \nabla^{(p)}_{g, \xi} (g^{- 1} v) \right]\;,
\end{eqnarray*}
which implies (\ref{ind-com-cov}) for $r = p$. Moreover the identity
$$
\left[ K, \nabla^{(p)}_{g, \xi} (g^{- 1} v)\right] \;\;=\;\; 0\;, 
$$
implies the equalities
\begin{eqnarray*}
  \nabla_g \nabla^{(p)}_{g, \xi} (g^{- 1} v) \;\;= \;\; \nabla_{g_0}
  \nabla^{(p)}_{g, \xi} (g^{- 1} v) \;\;=\;\; \nabla_{g_0} \nabla^{(p)}_{g_0,
  \xi} (g^{- 1} v)\;,
\end{eqnarray*}
thanks to (\ref{inv-cov}) and to the inductive assumption. We obtain
(\ref{pcov-cst-vrH}) for $r = p$ and thus the conclusion of the induction.

{\tmstrong{Step III}}
  
  We show now the last statement of the lemma. We observe indeed that the
  identities $\Sigma_{^{_K}} (g_0) =\mathcal{M} \cap \mathbbm{F}^K_g =
  \exp_{G, g} \left( \mathbbm{F}^K_g \right)$, hold thanks to the equalities
  (\ref{2flat-Scat-Space}) and (\ref{Iflat-Scat-Space}). But this implies that
  the second fundamental form of $\Sigma_K (g_0)$ inside $(\mathcal{M}, G)$
  vanishes identically. Thus using Gauss equation, the identity
  (\ref{2flat-Scat-Space}) and the expression of the curvature tensor
\begin{equation*}
  \label{curvat-formula} \mathcal{R}_{\mathcal{M}} (g) (u, v) w\;\; =\;\; -\;\, \frac{1}{4}\; g \Big[
  \left[ u_g^{\ast} \hspace{0.25em}, v^{\ast}_g \right], w^{\ast}_g \Big]\; ,
\end{equation*}
 we infer the equalities of the curvature forms $R_{\Sigma_{^{_K}}
  (g_0)} (g) = R_{\mathcal{M}} (g)_{\mid \mathbbm{F}^K_g } \equiv 0$ for all
  $g \in \Sigma_{^{_K}} (g_0)$. This concludes the proof of lemma
  \ref{flat-Scat-Space}.
\end{proof}

\section{Reinterpretation of the space $\mathbbm{F}^K_g$}

In this section we conciliate the definition of the vector space $\mathbbm{F}^K_g$ given
in the section \ref{main-res} with the definition so far used. 
\\
We consider indeed
$K \in C^{\infty} (X, \tmop{End} (T_X))$ with $n$-distinct real eigenvalues
almost everywhere over $X$. If $A, B \in C^{\infty} (X, \tmop{End} (T_X))$
commute with $K$ then $[A, B] = 0$ over $X$. We observe that this is all we
need in order to make work the previous arguments. Thus all the previous
results hold true if we use such $K$. In this case hold an equivalent
definition of the vector space $\mathbbm{F}^K_g$. We show in fact the
following lemma.

\begin{lemma}
  Let $K \in C^{\infty} (X, \tmop{End} (T_X))$ with $n$-distinct real
  eigenvalues almost everywhere over $X$. Then hold the identity
\begin{equation}
   \label{eq-def-FK} \mathbbm{F}^K_g \;\;=\;\; \Big\{ v \in \mathbbm{F}_g \mid
  \hspace{0.25em}  \left[ \nabla^p_g \,T, v^{\ast}_g \right] \;=\; 0\,,\; T
  \;=\;\mathcal{R}_g\,,\, K\,,\; \forall p \in \mathbbm{Z}_{\geqslant 0} \Big\}\; .
\end{equation}
\end{lemma}
\begin{proof}
  It is sufficient to show by induction on $p \geqslant 1$ that
  \begin{equation}
    \label{iduc-eqFK-I} \forall r \;=\; 0, \ldots, p \,;\; \left[ \nabla^r_g T,
    v^{\ast}_g \right] \;=\; 0\; \Longleftrightarrow \;\left[ T, \nabla^r_{g, \xi}\,
    v^{\ast}_g \right] \;=\; 0\,,\; \forall \xi \in T^{\otimes r}_X \;.
  \end{equation}
  We assume true this statement for $p - 1$ and we show it for $p$. The
  inductive hypothesis implies
  \begin{equation}
    \label{ind-FK-I} \forall r \;=\; 0, \ldots, p - 1 \,;\; \left[ \nabla^{r - s}_g \,T,
    \nabla^s_{g, \xi} \,v^{\ast}_g \right] \;=\; 0\,,\; \forall \xi \in T^{\otimes s}_X,\;
    \forall s \;=\; 0, \ldots, r \;.
  \end{equation}
  We show the statement (\ref{ind-FK-I}) by a finite increasing induction on
  $s$. The statement (\ref{ind-FK-I}) hold true obviously for $s = 0, 1$. We
  assume (\ref{ind-FK-I}) true for $s$ and we show it for $s + 1$. Indeed by
  the inductive assumption on $s$ hold the identity
  \begin{eqnarray*}
    \left[ \nabla^{r - 1 - s}_g \,T, \nabla^s_{g, \xi} \,v^{\ast}_g \right] \;\;=\;\;
    0\;,
  \end{eqnarray*}
  for all $\xi \in C^{\infty} (X, T^{\otimes s}_X)$. We take in particular
  $\xi$ such that $\nabla_g \,\xi (x) = 0$, at some arbitrary point $x \in X$.
  Thus hold the equalities
  \begin{eqnarray*}
    0 & = & \nabla_{g, \eta} \left[ \nabla^{r - 1 - s}_g \,T, \nabla^s_{g, \xi}\,
    v^{\ast}_g \right] \\
    &  & \\
    & = & \left[ \eta \;\neg\; \nabla^{r - s}_g \,T, \nabla^s_{g, \xi} \,v^{\ast}_g
    \right] \;\,+\;\, \left[ \nabla^{r - 1 - s}_g \,T, \nabla^{s + 1}_{g, \eta \otimes
    \xi} \,v^{\ast}_g \right]\\
    &  & \\
    & = & \left[ \nabla^{r - 1 - s}_g \,T, \nabla^{s + 1}_{g, \eta \otimes \xi}\,
    v^{\ast}_g \right]\;,
  \end{eqnarray*}
  thanks to the inductive assumption on $s$. This completes the proof of
  (\ref{ind-FK-I}). The conclusion of the induction on $p$ for
  (\ref{iduc-eqFK-I}) will follow from the statement; for all $s = 0, \ldots, p
  - 1$ hold the equivalence
  \begin{equation}
    \label{ind-FK-D} \left[ \nabla^{p - s}_g \,T, \nabla^s_{g, \xi} \,v^{\ast}_g
    \right] \,=\, 0\,,\, \forall \xi \in T^{\otimes s}_X \,\Longleftrightarrow \,\left[
    \nabla^{p - s - 1}_g \,T, \nabla^{s + 1}_{g, \eta} \,v^{\ast}_g \right] \,=\, 0\,,\,
    \forall \eta \in T^{\otimes s + 1}_X .
  \end{equation}
  By (\ref{ind-FK-I}) for $r = p - 1$ and $s = 0, \ldots, p - 1$ hold the identity
  \begin{eqnarray*}
    \left[ \nabla^{p - 1 - s}_g \,T, \nabla^s_{g, \xi} \,v^{\ast}_g \right] \;\; = \;\;
    0\;,
  \end{eqnarray*}
  for all $\xi \in C^{\infty} (X, T^{\otimes s}_X)$. We take as before $\xi$
  such that $\nabla_g \,\xi (x) = 0$, at some arbitrary point $x \in X$. Thus
  \begin{eqnarray*}
    0 & = & \nabla_{g, \eta} \left[ \nabla^{p - 1 - s}_g \,T, \nabla^s_{g, \xi}\,
    v^{\ast}_g \right] \\
    &  & \\
    & = & \left[ \eta \;\neg\; \nabla^{p - s}_g \,T, \nabla^s_{g, \xi} \,v^{\ast}_g
    \right] \;\,+\;\, \left[ \nabla^{p - 1 - s}_g \,T, \nabla^{s + 1}_{g, \eta \otimes
    \xi} \,v^{\ast}_g \right]\;,
  \end{eqnarray*}
  which shows (\ref{ind-FK-D}) and thus the conclusion of the induction on $p$
  for (\ref{iduc-eqFK-I}). 
\end{proof}

\section{Representation of the $\Omega$-SRF as the gradient flow of the
functional $\mathcal{W}_{\Omega}$ over $\Sigma_K (g_0)$}

The following proposition enlightens the properties of the sub-variety $\Sigma_K
(g_0)$.

\begin{proposition}
  \label{fund-scattering}For any $g_0 \in \mathcal{S}^K_{_{^{\Omega}}}$ and
  any $g \in \Sigma_K (g_0)$ hold the identities
\begin{eqnarray*}
    \nabla_G \mathcal{W}_{\Omega} \,(g) & = & g \;\,-\;\, \tmop{Ric}_g (\Omega) \;\in\;
    T_{\Sigma_K (g_0), g}\;,\\
    &  & \\
    \nabla^{_{^{\Sigma_K (g_0)}}}_G D\, \mathcal{W}_{\Omega} \,(g)\, (v, v) & = &
    \int_X \left[ \left\langle v \tmop{Ric}^{\ast}_g (\Omega), v
    \right\rangle_g \;\,+\;\, \frac{1}{2}\, | \nabla_g \,v|^2_g  \right] \Omega\;,
  \end{eqnarray*}
  for all $v \in T_{\Sigma_K (g_0), g}$. Moreover for all $g_0 \in \mathcal{S}^K_{_{^{\Omega}}}$ the
  $\Omega$-$\tmop{SRF}$ $(g_t)_{t \in \left[ 0, T) \right.} \subset
  \Sigma^{}_K (g_0)$ with initial data $g_0$ represents the formal gradient flow of
  the functional $\mathcal{W}_{\Omega}$ over the totally geodesic and flat
  sub-variety $\Sigma_K (g_0)$ inside the non-positively curved Riemannian
  manifold $(\mathcal{M}, G)$.
\end{proposition}

\begin{proof}
  By the identity (\ref{Iflat-Scat-Space}) in lemma \ref{flat-Scat-Space}
  there exist a geodesic $(g_t)_{t \in \mathbbm{R}} \subset \Sigma_K (g_0)$,
  $\dot{g}_0 \in \mathbbm{F}^K_{g_0}$, joining $g = g_1$ with $g_0$. Then
  lemma \ref{invar-F-K} combined with corollary \ref{FK-vr-Om-EndRc} and with
  the identity (\ref{2flat-Scat-Space}) in lemma \ref{flat-Scat-Space} implies
  $\tmop{Ric}_g (\Omega) \in T_{\Sigma_K (g_0), g}$.
  
  Moreover $g \in T_{\Sigma_K (g_0), g}$ thanks to the identity
  (\ref{2flat-Scat-Space}). We conclude the fundamental property $\nabla_G \mathcal{W}_{\Omega} (g) \in
  T_{\Sigma_K (g_0), g}$ for all $g\in \Sigma_K (g_0)$.
\\
The second variation formula in the statement
  follows directly from corollary 1 in \cite{Pal1} and from the fact that
  $\Sigma_K (g_0)$ is a totally geodesic sub-variety inside the Riemannian
  manifold $(\mathcal{M}, G)$. We observe however that it follows also from
  the tangency property $\nabla_G \mathcal{W}_{\Omega} (g) \in T_{\Sigma_K (g_0),
  g}$.
  Indeed we remind that for any sub-variety $\Sigma \subset \mathcal{M}$ of a
  general Riemannian manifold $(\mathcal{M}, G)$ and for any $f \in C^2
  (\mathcal{M}, \mathbbm{R})$ hold the identity
  \begin{eqnarray*}
    \nabla^{_{^{\mathcal{M}}}}_G d \,f \,(\xi, \eta) \;\; = \;\; \nabla^{_{^{\Sigma}}}_G\,
    d \,f \,(\xi, \eta) \;\,-\;\, G \left( \nabla^{_{^{\mathcal{M}}}}_G f,
    \tmop{II}_G^{_{^{\Sigma}}} (\xi, \eta) \right), \forall \xi, \eta \in
    T_{\Sigma}\;,
  \end{eqnarray*}
  where $\tmop{II}_G^{_{^{\Sigma}}} \in C^{\infty} \left( \Sigma, S^2
  T^{\ast}_{\Sigma} \otimes N_{\Sigma /\mathcal{M}, G} \right)$ denotes the
  second fundamental form of $\Sigma$ inside $(\mathcal{M}, G)$. 
\end{proof}

The result so far obtained does not allow to see yet the $\Omega$-SRF as the
gradient flow of a convex functional inside a flat metric space. In order to
see the required convexity picture we need to make a key change of variables
that we explain in the next sections.

\section{Explicit representations of the $\Omega$-SRF
equation}\label{rep-SRF-PM}

In this section we show the following fundamental expression of the
$\Omega$-BER-tensor over the variety $\Sigma_K (g_0)$.

\begin{lemma}
  \label{Cool-expr-OmRic}Let $g_0 \in \mathcal{M}$. Then for any metric $g \in
  \Sigma_K (g_0)$ hold the expression
  \begin{eqnarray*}
    \tmop{Ric}_g (\Omega) & = & \tmop{Ric}_{g_0} (\Omega) \;\,-\;\, \frac{1}{2}\,
    \Delta^{^{_{_{\Omega}}}}_{g_0}  \left[ g_0 \log (g^{- 1}_0 g) \right] \\
    &  & \\
    & - & \frac{1}{4} \,g_0 \,\tmop{Tr}_{g_0}  \Big[ \nabla_{g_0, \bullet} \log
    (g^{- 1}_0 g) \nabla_{g_0, \bullet} \log (g^{- 1}_0 g) \Big]\;.
  \end{eqnarray*}
\end{lemma}

\begin{proof}
  The fact that $\mathcal{R}_g =\mathcal{R}_{g_0}$ for all $g \in \Sigma_K
  (g_0)$ implies also $\tmop{Ric}_g = \tmop{Ric}_{g_0}$. Moreover we observe the elementary identities
  \begin{eqnarray*}
    \log \frac{dV_g}{\Omega} & = & \log \frac{dV_g}{d V_{g_0}} \;\,+\;\, \log
    \frac{dV_{g_0}}{\Omega}\;,\\
    &  & \\
    \frac{dV_g}{d V_{g_0}} & = & \left[ \frac{\det g}{\det g_0} \right]^{1 /
    2} \;\;=\;\; \left[ \det (g^{- 1}_0 g) \right]^{1 / 2}\;,\\
    &  & \\
    \log \frac{dV_g}{\Omega} & = & \frac{1}{2} \,\log \det (g^{- 1}_0 g) \;\,+\;\, \log
    \frac{dV_{g_0}}{\Omega} \\
    &  & \\
    & = & \frac{1}{2} \tmop{Tr}_{_{\mathbbm{R}}} \log (g^{- 1}_0 g) \;\,+\;\, \log
    \frac{dV_{g_0}}{\Omega}\;.
  \end{eqnarray*}
We set $f_0 \assign \log \frac{dV_{g_0}}{\Omega}$ and $\nabla_g = \nabla_{g_0}
  + \Gamma^{T_X}_g$. We infer the equality
  \begin{eqnarray*}
    \tmop{Ric}_g (\Omega) & = & \tmop{Ric}_{g_0} (\Omega) \;\,+\;\,
    \Gamma^{T^{^{_{\ast}}}_X}_g d\, f_0
\\
\\
& +& \frac{1}{2}\, \nabla_{g_0} d
    \tmop{Tr}_{_{\mathbbm{R}}} \log (g^{- 1}_0 g) \;\,+\;\, \frac{1}{2}\,
    \Gamma^{T^{^{_{\ast}}}_X}_g d \tmop{Tr}_{_{\mathbbm{R}}} \log (g^{- 1}_0
    g) \;.
  \end{eqnarray*}
  We observe now that $\Gamma^{T^{^{_{\ast}}}_X}_{g, \xi}  =  -\, \big( \Gamma^{T_X}_{g, \xi}
    \big)^{\ast}$, with
\begin{eqnarray*}
    2 \,\Gamma^{T_X}_{g, \xi} \,\eta & = & g^{- 1} \Big[ \nabla_{g_0} \,g \,(\xi, \eta,
    \cdot) \;\,+\;\, \nabla_{g_0} \,g \,(\eta, \xi, \cdot) \;\,-\;\, \nabla_{g_0} \,g \,(\cdot, \xi,
    \eta)\Big]\\
    &  & \\
    & = & \big(g^{- 1} \nabla_{g_0, \xi} \,g\big) \,\eta\;,
  \end{eqnarray*}
  since $\nabla_{_{T_{X, g_0}}} (g_0^{- 1} g) = 0$. Using lemma
  \ref{cov-der-log} we deduce the identity
  \begin{eqnarray*}
2\, \Gamma^{T_X}_{g, \xi} & = & (g^{- 1}_0 g)^{- 1} \nabla_{g_0, \xi} (g^{-1}_0 g)
\\
    &  & \\
    & = & \nabla_{g_0, \xi} \log (g^{- 1}_0 g)\;,\label{Gamma}
  \end{eqnarray*}
since $\left[ g^{- 1}_0 g, \nabla_{g_0, \xi} (g^{- 1}_0 g) \right]=0$.
Indeed we observe that this last equality follows from the fact that $[K, g^{- 1}_0 g] = 0$ and 
$[K,\nabla_{g_0} (g^{- 1}_0 g)] = 0$.
Thus for any function $u$ hold the
  identity
  \begin{eqnarray*}
    2 \left( \Gamma^{T^{^{_{\ast}}}_X}_{g, \xi} \,d\, u \right) \eta \;\; = \;\; -\;\, g_0
    \left( \nabla_{g_0} u, \nabla_{g_0, \xi} \log (g^{- 1}_0 g) \,\eta \right) \;.
  \end{eqnarray*}
  Using the fact that the endomorphism $\nabla_{g_0, \xi} \log (g^{- 1}_0 g)$
  is $g_0$-symmetric (since $\log (g^{- 1}_0 g)$ is also $g_0$-symmetric) we
  deduce
  \begin{equation}
    \label{form-Gamm} 2 \left( \Gamma^{T^{^{_{\ast}}}_X}_{g, \xi} \,d\, u \right)
    \eta \;\;=\;\; - \;\,g_0 \left( \nabla_{g_0, \xi} \log (g^{- 1}_0 g) \nabla_{g_0} u,
    \eta \right) \;.
  \end{equation}
  We observe that $g \in \Sigma_K (g_0)$ if and only if $g_0 \log (g_0^{- 1}
  g) \in \mathbbm{F}^K_{g_0}$ by the identity (\ref{Iflat-Scat-Space}). In
  particular
  \begin{equation}
    \label{clos-log} \nabla_{_{T_{X, g_0}}} \log (g_0^{- 1} g) \;\;=\;\; 0 \;.
  \end{equation}
  We deduce the equalities
  \begin{eqnarray*}
    2 \left( \Gamma^{T^{^{_{\ast}}}_X}_{g, \xi} \,d\, f_0 \right) \eta & = & - \;\,g_0
    \left( \nabla_{g_0, \nabla_{g_0} f_0} \log (g^{- 1}_0 g) \,\xi, \eta
    \right)\\
    &  & \\
    & = & - \;\,\nabla_{g_0, \nabla_{g_0} f_0}  \left[ g_0 \log (g^{- 1}_0 g)
    \right] (\xi, \eta) \;.
  \end{eqnarray*}
  Let $(e_k)_k$ be a local frame of $T_X$ in a neighborhood of an arbitrary
  point $x \in X$ which is $g$-orthonormal at $x$ and such that $\nabla_{g_0}
  e_k (x) = 0$. Deriving the identity
  \begin{eqnarray*}
    \tmop{Tr}_{_{\mathbbm{R}}} \log (g^{- 1}_0 g) \;\; = \;\; g_0 \big(\log (g^{- 1}_0
    g) \,e_k, g^{- 1}_0 e_k^{\ast}\big)\;,
  \end{eqnarray*}
  we infer at the point $x$
  \begin{eqnarray*}
    \xi . \tmop{Tr}_{_{\mathbbm{R}}} \log (g^{- 1}_0 g) & = & g_0 \left(
    \nabla_{g_0, \xi} \log (g^{- 1}_0 g) \,e_k, e_k \right)\\
    &  & \\
    & = & g_0 \left( \nabla_{g_0, e_k} \log (g^{- 1}_0 g) \,\xi, e_k \right)\\
    &  & \\
    & = & g_0 \left( \xi, \nabla_{g_0, e_k} \log (g^{- 1}_0 g) \,e_k \right)\;,
  \end{eqnarray*}
  thanks to the identity (\ref{clos-log}) and thanks to the fact that the
  endomorphism $\nabla_{g_0, \xi} \log (g^{- 1}_0 g)$ is $g_0$-symmetric. We
  infer the formula
  \begin{eqnarray*}
    \label{der-tr-log} d \tmop{Tr}_{_{\mathbbm{R}}} \log (g^{- 1}_0 g) \;\;=\;\; -\;\, g_0
    \nabla^{\ast}_{g_0} \log (g^{- 1}_0 g) \;.
  \end{eqnarray*}
  Thus using (\ref{clos-log}) we infer the equalities
  \begin{eqnarray*}
    \nabla_{g_0} d \tmop{Tr}_{_{\mathbbm{R}}} \log (g^{- 1}_0 g) & = & - \;\,g_0
    \nabla_{g_0} \nabla^{\ast}_{g_0} \log (g^{- 1}_0 g)\\
    &  & \\
    & = & - \;\,g_0 \,\Delta_{_{T_{X, g_0}}} \log (g_0^{- 1} g)\\
    &  & \\
    & = & - \;\,g_0 \,\Delta_{g_0} \log (g_0^{- 1} g)\\
    &  & \\
    & = & - \;\,\Delta_{g_0}  \left[ g_0 \log (g^{- 1}_0 g) \right]\;,
  \end{eqnarray*}
  by the Weitzenb\"ock formula in lemma \ref{TX-Lap-RmLap} in the appendix and by the identity
  $[\mathcal{R}_g, \log (g_0^{- 1} g)] = 0$. (We remind that $g_0 \log (g_0^{-
  1} g) \in \mathbbm{F}^K_{g_0}$.) We apply now the identity (\ref{form-Gamm})
  to the function $u \assign \tmop{Tr}_{_{\mathbbm{R}}} \log (g^{- 1}_0 g)$.
  We infer by the formula (\ref{der-tr-log}) the equality
  \begin{eqnarray*}
    \nabla_{g_0} u & = & \nabla_{g_0, e_k} \log (g^{- 1}_0 g) e_k .
  \end{eqnarray*}
  Thus we obtain the equality
\begin{eqnarray*}
2 \left( \Gamma^{T^{^{_{\ast}}}_X}_{g, \xi} \,d \tmop{Tr}_{_{\mathbbm{R}}}
    \log (g^{- 1}_0 g) \right) \eta 
\;\;= \;\; -\;\, g_0 \left( \nabla_{g_0, \xi} \log
    (g^{- 1}_0 g) \nabla_{g_0, e_k} \log (g^{- 1}_0 g) \,e_k, \eta \right) \;.
\end{eqnarray*}
  Using the identity $[K, \nabla_{g_0} \log (g^{- 1}_0 g)] = 0$, we deduce the
  expression at the point $x$
  \begin{eqnarray*}
    2 \left( \Gamma^{T^{^{_{\ast}}}_X}_{g, \xi} \,d \tmop{Tr}_{_{\mathbbm{R}}}
    \log (g^{- 1}_0 g) \right) \eta & = & - \;\,g_0 \left( \nabla_{g_0, e_k} \log
    (g^{- 1}_0 g) \nabla_{g_0, \xi} \log (g^{- 1}_0 g) \,e_k, \eta \right) \\
    &  & \\
    & = & - \;\,g_0 \left( \nabla_{g_0, e_k} \log (g^{- 1}_0 g) \nabla_{g_0, e_k}
    \log (g^{- 1}_0 g) \,\xi, \eta \right)\;,
  \end{eqnarray*}
  thanks to (\ref{clos-log}). Combining the expressions obtained so far we
  infer the required formula.
\end{proof}

We remind that by proposition \ref{fund-scattering} follows that if $g_0 \in
\mathcal{S}^K_{_{^{\Omega}}}$ then
\begin{equation}
  \label{fund-scat} \tmop{Ric}_g (\Omega) \;\in\; \mathbbm{F}_{g_0}^K\,,\quad \forall g
  \;\in\; \Sigma_K (g_0) \;.
\end{equation}
We give an other proof (not necessarily shorter but more explicit) of this fundamental fact
based on the expression of $\tmop{Ric}_g (\Omega)$ in lemma
\ref{Cool-expr-OmRic}.

\begin{proof}
  We remind that $g \in \Sigma_K (g_0)$ if and only if $g_0 \log (g_0^{- 1} g)
  \in \mathbbm{F}^K_{g_0}$ by the identity (\ref{Iflat-Scat-Space}). We set
  for notation simplicity 
$$
U \;\;\assign\;\; \log (g^{- 1}_0 g) \;\in\; g^{- 1}_0
  \mathbbm{F}^K_{g_0}\;,
$$ 
and we show first the equality
  \begin{equation}
    \label{closed-endRic} \nabla_{_{T_X, g_0}} g^{- 1}_0 \tmop{Ric}_g (\Omega)
    \;\;=\;\; 0\; .
  \end{equation}
  The assumption $g_0 \in \mathcal{S}^K_{_{^{\Omega}}}$ combined with the
  expression of $\tmop{Ric}_g (\Omega)$ in lemma \ref{Cool-expr-OmRic} implies the identity
  \begin{eqnarray*}
    \nabla_{_{T_X, g_0}} g^{- 1}_0 \tmop{Ric}_g (\Omega) \;\; = \;\; -\;\, \frac{1}{2}\,
    \nabla_{_{T_X, g_0}} \Delta^{^{_{_{\Omega}}}}_{g_0} U \;.
  \end{eqnarray*}
  Consider now $(x_1, ..., x_n)$ be $g_0$-geodesic coordinates centered at an
  arbitrary point $p \in X$ and set $e_k : = \frac{\partial}{\partial x_k}$.
  The local tangent frame $(e_k)_k$ is $g_t (p)$-orthonormal at the point $p$
  and satisfies $\nabla_{g_0} e_j (p) = 0$ for all $j$. 
\\
We take now two vector
  fields $\xi$ and $\eta$ with constant coefficients with respect to the
  $g_0$-geodesic coordinates $(x_1, ..., x_n)$. Therefore $\nabla_{g_0} \xi
  (p) = \nabla_{g_0} \eta (p) = 0$. Commuting derivatives by means of
  (\ref{com-cov}) we infer the identities at the point $p$
  \begin{eqnarray*}
    \nabla_{g_0, \xi} \tmop{Tr}_{g_0} \left( \nabla_{g_0, \bullet} \,U\,
    \nabla_{g_0, \bullet} \,U \right) \eta & = & \nabla_{g_0, \xi}  \left(
    \nabla_{g_0, e_k} U \,\nabla_{g_0, g^{- 1}_0 e^{\ast}_k} U \right) \eta\\
    &  & \\
    & = & \nabla_{g_0, e_k} \nabla_{g_0, \xi}\, U\, \nabla_{g_0, e_k} U \eta\\
    &  & \\
    & + & \nabla_{g_0, e_k} U \,\nabla_{g_0, e_k} \nabla_{g_0, \xi} \,U \eta\\
    &  & \\
    & = & 2\, \nabla_{g_0, e_k} U \,\nabla_{g_0, e_k} \left[ \nabla_{g_0, \xi} \,U
    \eta \right]\;,
  \end{eqnarray*}
  since $[\mathcal{R}_{g_0}, U] = 0$, $[\xi, e_k] \equiv 0$, $[\xi, g^{- 1}_0
  e^{\ast}_k] = 0$ at the point $p$ and because $[\nabla_{g_0, e_k} \nabla_{g_0, \xi}\, U,
  \nabla_{g_0, e_k} U] = 0$. We infer the required identity
  \begin{eqnarray*}
    \nabla_{_{T_X, g_0}} \tmop{Tr}_{g_0} \left( \nabla_{g_0, \bullet}\, U
    \,\nabla_{g_0, \bullet} \,U \right) \;\; = \;\; 0\; .
  \end{eqnarray*}
  We expand now at the point $p$ the therm
  \begin{eqnarray*}
    \nabla_{_{T_X, g_0}} \Delta_{g_0} U (\xi, \eta) & = &  -\;\,
    \nabla_{g_0, \xi}  \big[ \nabla_{g_0, e_k} \nabla_{g_0} U (e_k, \eta) \big]
    \\
\\
&+& \nabla_{g_0, \eta}  \big[ \nabla_{g_0, e_k} \nabla_{g_0} U (e_k, \xi)
    \big] 
\end{eqnarray*}
\begin{eqnarray*}
    & = & - \;\,\nabla_{g_0, \xi} \nabla_{g_0, e_k} \nabla_{g_0, e_k} U \eta \;\,+\;\,
    \nabla_{g_0} U (\nabla_{g_0, \xi} \nabla_{g_0, e_k} e_k, \eta)\\
    &  & \\
    & + & \nabla_{g_0, \eta} \nabla_{g_0, e_k} \nabla_{g_0, e_k} U \xi \;\,-\;\,
    \nabla_{g_0} U (\nabla_{g_0, \eta} \nabla_{g_0, e_k} e_k, \xi)\\
    &  & \\
    & = & - \;\,\nabla_{g_0, e_k} \nabla_{g_0, \xi} \nabla_{g_0, e_k} U \eta \;\,+\;\,
    \nabla_{g_0} U (\nabla_{g_0, \xi} \nabla_{g_0, e_k} e_k, \eta)\\
    &  & \\
    & + & \nabla_{g_0, e_k} \nabla_{g_0, \eta} \nabla_{g_0, e_k} U \xi \;\,-\;\,
    \nabla_{g_0} U (\nabla_{g_0, \eta} \nabla_{g_0, e_k} e_k, \xi)\;,
  \end{eqnarray*}
  since
  \begin{eqnarray*}
    \nabla_{g_0, e_k} \nabla_{g_0} U (e_k, \eta) & = & \nabla_{g_0, e_k} 
    \left[ \nabla_{g_0, e_k} U \eta \right] 
\\
\\
&-& \nabla_{g_0} U (\nabla_{g_0,
    e_k} e_k, \eta) \;\,-\;\, \nabla_{g_0} U (e_k, \nabla_{g_0, e_k} \eta)\\
    &  & \\
    & = & \nabla_{g_0, e_k} \nabla_{g_0, e_k} U \eta \;\,-\;\, \nabla_{g_0} U
    (\nabla_{g_0, e_k} e_k, \eta)\;,
  \end{eqnarray*}
  and $\left[ \xi, e_k \right] = \left[ \eta, e_k \right] \equiv 0$ in a
  neighborhood of $p$ and since $\left[ \mathcal{R}_{g_0}, \nabla_{g_0, e_k} U
  \right] \equiv 0$. (We use here the \ identity (\ref{com-cov}).) Moreover
  using the fact that $\left[ \mathcal{R}_{g_0}, U \right] = 0$ we infer the
  identity at the point $p$
  \begin{eqnarray*}
    \nabla_{_{T_X, g_0}} \Delta_{g_0} U (\xi, \eta) & = & -\;\, \nabla_{g_0, e_k}
    \nabla_{g_0, e_k} \nabla_{g_0, \xi}\, U \eta \;\,+\;\, \nabla_{g_0} U (\nabla_{g_0,
    \xi} \nabla_{g_0, e_k} e_k, \eta)\\
    &  & \\
    & + & \nabla_{g_0, e_k} \nabla_{g_0, e_k} \nabla_{g_0, \eta} \,U \xi \;\,-\;\,
    \nabla_{g_0} U (\nabla_{g_0, \eta} \nabla_{g_0, e_k} e_k, \xi) \;.
  \end{eqnarray*}
  We expand now at the point $p$ the therm
  \begin{eqnarray*}
    0 \;\;=\;\; \Delta_{g_0} \nabla_{_{T_X, g_0}} U (\xi, \eta) \;\; = \;\; - \;\,\nabla_{g_0,
    e_k}  [ \nabla_{g_0, e_k} \nabla_{_{T_X, g_0}} U (\xi, \eta)] \;.
  \end{eqnarray*}
  Thus we expand first the therm
  \begin{eqnarray*}
    \nabla_{g_0, e_k} \nabla_{_{T_X, g_0}} U (\xi, \eta) & = & \nabla_{g_0,
    e_k}  \big[ \nabla_{_{T_X, g_0}} U (\xi, \eta)\big] \\
    &  & \\
    & - & \nabla_{_{T_X, g_0}} U (\nabla_{g_0, e_k} \xi, \eta) \;\,-\;\,
    \nabla_{_{T_X, g_0}} U (\xi, \nabla_{g_0, e_k} \eta)\\
    &  & \\
    & = & \nabla_{g_0, e_k}  \big[ \nabla_{g_0, \xi} \,U \eta \;\,-\;\, \nabla_{g_0,
    \eta} \,U \xi \big]\\
    &  & \\
    & - & \nabla_{g_0} U (\nabla_{g_0, e_k} \xi, \eta) \;\,+\;\, \nabla_{g_0} U
    (\eta, \nabla_{g_0, e_k} \xi)\\
    &  & \\
    & - & \nabla_{g_0} U (\xi, \nabla_{g_0, e_k} \eta) \;\,+\;\, \nabla_{g_0} U
    (\nabla_{g_0, e_k} \eta, \xi)
\\
\\
    & = & \nabla_{g_0, e_k} \nabla_{g_0, \xi} \,U \eta \;\,-\;\, \nabla_{g_0, e_k}
    \nabla_{g_0, \eta} \,U \xi\\
    &  & \\
    & - & \nabla_{g_0} U (\nabla_{g_0, e_k} \xi, \eta) \;\,+\;\, \nabla_{g_0} U
    (\nabla_{g_0, e_k} \eta, \xi)\;,
  \end{eqnarray*}
  in a neighborhood of $p$. We deduce the identity at the point $p$
  \begin{eqnarray*}
    0 \;\;=\;\; \Delta_{g_0} \nabla_{_{T_X, g_0}} U (\xi, \eta) & = & -\;\, \nabla_{g_0,
    e_k} \nabla_{g_0, e_k} \nabla_{g_0, \xi}\, U \eta 
\\
\\
&+& \nabla_{g_0, e_k}
    \nabla_{g_0, e_k} \nabla_{g_0, \eta} \,U \xi\\
    &  & \\
    & + & \nabla_{g_0} U (\nabla_{g_0, e_k} \nabla_{g_0, \xi} \,e_k, \eta) 
\\
\\
&-&\nabla_{g_0} U (\nabla_{g_0, e_k} \nabla_{g_0, \eta} \,e_k, \xi)\;,
  \end{eqnarray*}
  since $\left[ \xi, e_k \right] = \left[ \eta, e_k \right] \equiv 0$. Thus we
  obtain the identity 
  \begin{eqnarray*}
    \nabla_{_{T_X, g_0}} \Delta_{g_0} U (\xi, \eta) & = & \nabla_{g_0} U
    (\mathcal{R}_{g_0} (\xi, e_k) e_k, \eta) \;\,-\;\, \nabla_{g_0} U
    (\mathcal{R}_{g_0} (\eta, e_k) e_k, \xi)\\
    &  & \\
    & = & \nabla_{g_0, \eta} \,U \tmop{Ric}^{\ast}_{g_0} \xi \;\,-\;\, \nabla_{g_0,
    \xi} \,U \tmop{Ric}^{\ast}_{g_0} \eta\;,
  \end{eqnarray*}
  since $g_0 U \in \mathbbm{F}_{g_0}$. We expand now at the point $p$ the
  therm 
  \begin{eqnarray*}
 \nabla_{_{T_X, g_0}}  \big[ \nabla_{g_0} f_0 \;\neg\; \nabla_{g_0} U
    \big] (\xi, \eta) & = & \nabla_{_{T_X, g_0}}  \big[ \nabla_{g_0} U\,
    \nabla_{g_0} f_0 \big] (\xi, \eta)\\
    &  & \\
    & = &  \nabla_{g_0, \xi}  \big[ \nabla_{g_0, \eta} \,U
    \,\nabla_{g_0} f_0 \big] 
\\
\\
&-& \nabla_{g_0, \eta}  \big[ \nabla_{g_0, \xi} \,U
    \,\nabla_{g_0} f_0 \big]\\
    &  & \\
    & = & \nabla_{g_0, \xi} \nabla_{g_0, \eta}\, U \,\nabla_{g_0} f_0 \;\,+\;\,
    \nabla_{g_0, \eta}\, U \,\nabla^2_{g_0} f_0 \,\xi\\
    &  & \\
    & - & \nabla_{g_0, \eta} \nabla_{g_0, \xi} \,U \,\nabla_{g_0} f_0 \;\,-\;\,
    \nabla_{g_0, \xi} \,U\, \nabla^2_{g_0} f_0 \,\eta\\
    &  & \\
    & = & \nabla_{g_0, \eta}\, U\, \nabla^2_{g_0} f_0 \,\xi \;\,-\;\, \nabla_{g_0, \xi} \,U\,
    \nabla^2_{g_0} f_0 \,\eta\;,
  \end{eqnarray*}
  since $\left[ \mathcal{R}_{g_0}, U \right] = 0$ and $\left[ \xi, \eta
  \right] \equiv 0$. We deduce the identity
  \begin{eqnarray*}
    \nabla_{_{T_X, g_0}} \Delta^{^{_{_{\Omega}}}}_{g_0} U \;\; =\;\, -\;\, \tmop{Alt} \big[
    \nabla_{g_0} U \tmop{Ric}^{\ast}_{g_0} (\Omega)\big] \;.
  \end{eqnarray*}
  But $[ \nabla_{g_0} U, \tmop{Ric}^{\ast}_{g_0} (\Omega)] = 0$ since $g_0 \in
  \mathcal{S}^K_{_{^{\Omega}}}$. We infer
  \begin{eqnarray*}
    \nabla_{_{T_X, g_0}} \Delta^{^{_{_{\Omega}}}}_{g_0} U \;\; = \;\; -\;
    \tmop{Ric}^{\ast}_{g_0} (\Omega) \,\nabla_{_{T_X, g_0}} U \;\;=\;\; 0\; .
  \end{eqnarray*}
Moreover 
$$
\big[T,\nabla^p_{g_0, \xi} \,\Delta^{^{_{_{\Omega}}}}_{g_0} U\big]\;\;=\;\;0\;,
$$ 
thanks to the identity (\ref{p-derLap}) applied to $U$.
We observe now the decomposition
  \begin{eqnarray*}
    \nabla^p_{g_0, \xi} \tmop{Tr}_{g_0} \left( \nabla_{g_0, \bullet} \,U\,
    \nabla_{g_0, \bullet} \,U \right) & = & \sum_{r = 0}^p  \,\sum_{| I| = r}
    \nabla^{r + 1}_{g_0, e_k, \xi_I} U \,\nabla^{p - r + 1}_{g_0, g^{- 1}_0
    e^{\ast}_k, \xi_{\complement I}} U \;.
  \end{eqnarray*}
Then the conclusion follows from the identity 
  (\ref{nul-triv-com}) with $C=T$, $A=\nabla^{r + 1}_{g_0, e_k, \xi_I} U$ and $B=\nabla^{p - r + 1}_{g_0, g^{- 1}_0 e^{\ast}_k, \xi_{\complement I}} U$.
\end{proof}

In the following lemma we introduce a fundamental change of variables.

\begin{lemma}
  \label{Cool-expr-EndOmRic}Let $g_0 \in \mathcal{M}$ and set $H \equiv H_g
  \assign (g^{- 1} g_0)^{1 / 2}$ for any metric $g \in \Sigma_K (g_0)$. Then
  hold the expression
  \begin{eqnarray*}
    \tmop{Ric}^{\ast}_g (\Omega) \;\;=\;\; H \Delta^{^{_{_{\Omega}}}}_{g_0} H \;\,+\;\,
    H^2 \tmop{Ric}^{\ast}_{g_0} (\Omega) \;.
  \end{eqnarray*}
\end{lemma}

\begin{proof}
  By lemma \ref{Cool-expr-OmRic} we obtain the formula
  \begin{eqnarray*}
    \tmop{Ric}^{\ast}_g (\Omega) & = & g^{- 1} g_0 \tmop{Ric}^{\ast}_{g_0}
    (\Omega) \;\,-\;\, \frac{1}{2}\, g^{- 1} g_0 \,\Delta^{^{_{_{\Omega}}}}_{g_0} \log
    (g^{- 1}_0 g)\\
    &  & \\
    & - & \frac{1}{4}\, g^{- 1} g_0 \tmop{Tr}_{g_0}  \left[ \nabla_{g_0,
    \bullet} \log (g^{- 1}_0 g) \nabla_{g_0, \bullet} \log (g^{- 1}_0 g)
    \right]\;.
  \end{eqnarray*}
  We set 
$$
A \;\;\assign\;\; \log H \;\;=\;\; -\;\, \frac{1}{2}\, \log (g^{- 1}_0 g) \;\,\in\;\, g^{- 1}_0
  \mathbbm{F}^K_{g_0}\;.
$$ 
and we observe that $H = e^A \in g^{- 1}_0 \Sigma_K
  (g_0)$ thanks to lemma \ref{mut-exp}. With this notations the previous
  formula rewrites as
  \begin{equation}
    \label{quad-RIC} \tmop{Ric}^{\ast}_g (\Omega) \;\;=\;\; e^{2 A}  \left[
    \Delta^{^{_{_{\Omega}}}}_{g_0} A \;\,-\;\, \tmop{Tr}_{g_0}  \left( \nabla_{g_0,
    \bullet} A \nabla_{g_0, \bullet} A \right) \;\,+\;\, \tmop{Ric}^{\ast}_{g_0}
    (\Omega) \right]\; .
  \end{equation}
  We expand now, at an arbitrary center of geodesic coordinates, the therm
  \begin{eqnarray*}
    \Delta^{^{_{_{\Omega}}}}_{g_0} e^A & = & - \;\,\nabla_{g_0, e_k} \nabla_{g_0,
    e_k} e^A \;\,+\;\, \nabla_{g_0} f_0 \;\neg\; \nabla_{g_0} e^A\\
    &  & \\
    & = & - \;\,\nabla_{g_0, e_k}  \left( e^A \nabla_{g_0, e_k} A \right) \;\,+\;\, e^A 
    \left( \nabla_{g_0} f_0 \;\neg\; \nabla_{g_0} A \right)\\
    &  & \\
    & = & e^A \Delta^{^{_{_{\Omega}}}}_{g_0} A \;\,-\;\, e^A \tmop{Tr}_{g_0}  \left(
    \nabla_{g_0, \bullet} A \,\nabla_{g_0, \bullet} A \right) \;.
  \end{eqnarray*}
  We infer the expression
  \begin{equation}
    \label{exp-exprRic} \tmop{Ric}^{\ast}_g (\Omega) \;\;=\;\; e^A
    \Delta^{^{_{_{\Omega}}}}_{g_0} e^A \;\,+\;\, e^{2 A} \tmop{Ric}^{\ast}_{g_0}
    (\Omega)\;,
  \end{equation}
  i.e the required conclusion.
\end{proof}
We deduce the following corollary.
\begin{corollary}
  \label{Porous-Medium}The $\Omega$-SRF $(g_t)_t \subset \Sigma_K (g_0)$ is
  equivalent to the porous medium type equation
  \begin{eqnarray*}
    2 \,\dot{H}_t & = & -\;\, H_t^2 \Delta^{^{_{_{\Omega}}}}_{g_0} H_t \;\,-\;\, H_t^3
    \tmop{Ric}^{\ast}_{g_0} (\Omega) \;\,+\;\, H_t\;,
  \end{eqnarray*}
  with initial data $H_0 =\mathbbm{I}$, via the identification $H_t = (g_t^{-
  1} g_0)^{1 / 2} \in g^{- 1}_0 \Sigma_K (g_0)$.
\end{corollary}

\begin{proof}
  Let $U_t \assign \log (g^{- 1}_0 g_t) \in g^{- 1}_0 \mathbbm{F}^K_{g_0}$ and
  observe that the $\Omega$-SRF equation $(g_t)_t \subset \Sigma_K (g_0)$ is
  equivalent to the evolution equation
  \begin{eqnarray*}
    \dot{U}_t \;\; = \;\; \dot{g}_t^{\ast} \;\;=\;\; \tmop{Ric}^{\ast}_{g_t} (\Omega)
    \;\,-\;\,\mathbbm{I}\;.
  \end{eqnarray*}
  Then the conclusion follows combining the identity $2 \dot{H}_t = - H_t \,
  \dot{U}_t$ with the previous lemma.
\end{proof}

We observe that the change of variables
$$ 
g \;\in\; \Sigma_K (g_0) \longmapsto H = (g^{- 1} g_0)^{1 / 2} \in g^{-
   1}_0 \Sigma_K (g_0)\;, 
$$
is the one which linearizes as much as possible the expression of the SRF
equation. Indeed we can rewrite it as the porous medium equation in corollary
\ref{Porous-Medium}. However we will see in the next section that the change
of variables $g \mapsto A = \log H$ would fit us in a gradient flow picture of
a convex functional over convex sets in a Hilbert space. So from now on we
will consider the change of variables
\begin{equation}
  \label{chg-var} g \in \Sigma_K (g_0) \longmapsto A \;=\; \log H \;=\; -\;
  \frac{1}{2}\, \log (g^{- 1}_0 g) \in \mathbbm{T}_{g_0}^K \assign g^{- 1}_0
  \mathbbm{F}^K_{g_0}\; .
\end{equation}
We define the $\Omega$-divergence operator of a tensor $\alpha$ as
\begin{eqnarray*}
  \tmop{div}^{^{_{_{\Omega}}}}_g \alpha \;\; \assign \;\; e^f \tmop{div}_g \left(
  e^{- f} \alpha \right) \;\;=\;\; \tmop{div}_g \alpha \;\,-\;\, \nabla_g f \;\neg\; \alpha\;,
\end{eqnarray*}
with $f \assign \log \frac{d V_g}{\Omega}$. We observe that with this notation
formula (\ref{exp-exprRic}) rewrites as
\begin{equation}
  \label{exp-expr-OmRic} \tmop{Ric}^{\ast}_{g_A} (\Omega) \;\;=\;\; -\;\, e^A
  \tmop{div}^{^{_{_{\Omega}}}}_{g_0} \left( e^A \nabla_{g_0} A \right) \;\,+\;\, e^{2
  A} \tmop{Ric}^{\ast}_{g_0} (\Omega)\;,
\end{equation}
with $g_A \assign g_0 e^{- 2 A}$, for all $A \in \mathbbm{T}_{g_0}^K$. With
this notations hold the analogue of corollary \ref{Porous-Medium}.

\begin{corollary}
  The $\Omega$-SRF $(g_t)_t \subset \Sigma_K (g_0)$ is equivalent to the
  solution $(A_t)_{t \geqslant 0} \subset \mathbbm{T}_{g_0}^K$ of the forward
  evolution equation
  \begin{equation}
    \label{non-lin-GrdFw} 2\, \dot{A}_t \;\;=\;\; e^{A_t}
    \tmop{div}^{^{_{_{\Omega}}}}_{g_0} \left( e^{A_t} \nabla_{g_0} A_t \right)
    \;\,-\;\, e^{2 A_t} \tmop{Ric}^{\ast}_{g_0} (\Omega) \;\,+\;\,\mathbbm{I}\;,
  \end{equation}
  with initial data $A_0 = 0$, via the identification $( \ref{chg-var})$.
\end{corollary}

\begin{proof}
  We observe first the identity $\dot{g}_t \;=\; -\, 2\, g_t\,  \dot{A}_t$. Then the
  conclusion follows combining the identity $- \,2\, \dot{A}_t \;=\; \dot{g}_t^{\ast}
  \;=\; \tmop{Ric}^{\ast}_{g_t} (\Omega) \;-\;\mathbbm{I}$, with formula
  (\ref{exp-expr-OmRic}).
\end{proof}

We observe that the assumption $g_0 \in \mathcal{S}^K_{_{^{\Omega}}}$ implies
\begin{eqnarray*}
  \tmop{Ric}^{\ast}_{g_A} (\Omega) \;\; = \;\; e^{2 A}\, g^{- 1}_0 \tmop{Ric}_{g_A}
  (\Omega) \;\in\; \mathbbm{T}_{g_0}^K\,,\quad \forall A \;\in\; \mathbbm{T}_{g_0}^K\,,
\end{eqnarray*}
thanks to the fundamental identity (\ref{fund-scat}) combined with lemma
\ref{mut-exp}. We deduce
\begin{equation}
  \label{exp-SCAT} e^A \tmop{div}^{^{_{_{\Omega}}}}_{g_0} \left( e^A
  \nabla_{g_0} A \right) \;\,-\;\, e^{2 A} \tmop{Ric}^{\ast}_{g_0} (\Omega) \;\in\;
  \mathbbm{T}_{g_0}^K\,,\quad \forall A \in \mathbbm{T}_{g_0}^K\,,
\end{equation}
thanks to the expression (\ref{exp-expr-OmRic}).

\section{Convexity of $\mathcal{W}_{\Omega}$ over convex subsets inside
$(\Sigma_K (g_0), G)$}\label{sec-PM}

We define the functional ${\bf W}_{\Omega}$ over $\mathbbm{T}_{g_0}^K$ by
the formula ${\bf W}_{\Omega} (A) \assign \mathcal{W}_{\Omega}
(g_A)$, via the identification (\ref{chg-var}). We remind now the identity
\begin{eqnarray*}
  \mathcal{W}_{\Omega} (g) & = & \int_X \left[ \tmop{Tr}_g  \left(
  \tmop{Ric}_g (\Omega) \;\,-\;\, g \right) \;\,+\;\, 2 \log
  \frac{dV_g}{\Omega} \right] \Omega\\
  &  & \\
  & = & \int_X \tmop{Tr}_{_{\mathbbm{R}}} \Big[ \tmop{Ric}^{\ast}_g (\Omega)
  \;\,+\;\, \log (g^{- 1}_0 g) \Big] \Omega \;\,+\;\, \int_X \left[ 2 \log \frac{d
  V_{g_0}}{\Omega} \;\,-\;\, n \right] \Omega \;.
\end{eqnarray*}
Plunging (\ref{exp-exprRic}) and integrating by parts we infer the expression
\begin{eqnarray*}
  {\bf W}_{\Omega} (A) & = & \int_X \Big[ \big| \nabla_{g_0} e^A \big|^2_{g_0} \;\,+\;\,
  \tmop{Tr}_{_{\mathbbm{R}}} \left( e^{2 A} \tmop{Ric}^{\ast}_{g_0} (\Omega) \;\,-\;\,
  2 \,A \right)\Big] \Omega\\
  &  & \\
  & + & \int_X \left[ 2 \log \frac{d V_{g_0}}{\Omega} \;\,-\;\, n \right] \Omega \;.
\end{eqnarray*}
We define now the vector space $\overline{\mathbbm{T}}^K_{g_0}$ as the
$L^2$-closure of $\mathbbm{T}_{g_0}^K$ and we equip it with the constant
$L^2$-product $4 \int_X \left\langle \cdot, \cdot \right\rangle_{g_0} \Omega$.
From now on all $L^2$-products are defined by this formula.

\begin{lemma}
  \label{Grad-Por-Med}Let $g_0 \in \mathcal{S}^K_{_{^{\Omega}}}$. Then the
  forward equation equation $( \ref{non-lin-GrdFw})$ with initial data $A_0 =
  0$ is equivalent to a smooth solution of the gradient flow equation
  $\dot{A}_t = - \,\nabla_{L^2} {\bf W}_{\Omega} (A_t)$.
\end{lemma}

\begin{proof}
  We compute first the $L^2$-gradient of the functional
  ${\bf W}_{\Omega}$. For this purpose we consider a line $t \mapsto
  A_t \assign A + t \,V$ with $A, V \in \mathbbm{T}^K_{g_0}$ arbitrary. Then
  hold the identity
  \begin{eqnarray}
    \label{first-varW}  \frac{d}{d t} \,{\bf W}_{\Omega} (A_t) &=& 2 \int_X
     \left\langle \nabla_{g_0} e^{A_t}, \nabla_{g_0} (e^{A_t} V)
    \right\rangle_{g_0}\Omega\nonumber
\\\nonumber
\\
& +& 2\int_X\tmop{Tr}_{_{^{\mathbbm{R}}}} \left[ \left( e^{2
    A_t} \tmop{Ric}^{\ast}_{g_0} (\Omega) \;\,-\;\,\mathbbm{I} \right) V \right] 
    \Omega \;.
  \end{eqnarray}
Integrating by parts we obtain the first variation formula
  \begin{eqnarray*}
    \frac{d}{d t} _{\mid_{t = 0}} {\bf W}_{\Omega} (A_t) \; = \; -\; 2 \int_X
    \Big\langle e^A \tmop{div}^{^{_{_{\Omega}}}}_{g_0} \left( e^A
    \nabla_{g_0} A \right) \;-\; e^{2 A} \tmop{Ric}^{\ast}_{g_0} (\Omega)
    \;+\;\mathbbm{I}\,, V \Big\rangle_{g_0} \Omega \;.
  \end{eqnarray*}
  The assumption $g_0 \in \mathcal{S}^K_{_{^{\Omega}}}$ implies the expression of the gradient
  \begin{eqnarray*}
    2 \,\nabla_{L^2} {\bf W}_{\Omega} (A) \;\; = \;\; -\;\, e^A
    \tmop{div}^{^{_{_{\Omega}}}}_{g_0} \left( e^A \nabla_{g_0} A \right) \;\,+\;\,
    e^{2 A} \tmop{Ric}^{\ast}_{g_0} (\Omega) \;\,-\;\,\mathbbm{I} \;\;\in\;\;
    \mathbbm{T}_{g_0}^K\,,
  \end{eqnarray*}
  thanks to the identity (\ref{exp-SCAT}). We infer the required conclusion.
\end{proof}

We show now the following convexity results.

\begin{lemma}
  \label{convex-lm} For any $g_0 \in \mathcal{M}$ and for all $A, V \in
  \mathbbm{T}^K_{g_0}$ hold the second variation formula
  \begin{eqnarray*}
    &&\nabla_{L^2} D\,{\bf W}_{\Omega} (A) (V, V) 
\\
\\
& = & 2 \int_X \Big\langle
    \left[ - \;\,e^A \tmop{div}^{^{_{_{\Omega}}}}_{g_0} \left( e^A \nabla_{g_0} A
    \right) \;\,+\;\, 2\, e^{2 A} \tmop{Ric}^{\ast}_{g_0} (\Omega) \right] V, V
    \Big\rangle_{g_0} \Omega\\
    &  & \\
    & + & 2 \int_X \big| \nabla_{g_0} (e^A V) \big|^2_{g_0} \Omega
  \end{eqnarray*}
  Moreover if $\tmop{Ric}_{g_0} (\Omega) > 0$ then the functional
  ${\bf W}_{\Omega}$ is convex over the convex set
  \begin{eqnarray*}
    \mathbbm{T}^{K, +}_{g_0} \;\assign \; \left\{ A \in \mathbbm{T}_{g_0}^K
    \mid \int_X |U \nabla_{g_0} A|^2_{g_0} \Omega \,\leqslant\, \int_X
    \tmop{Tr}_{_{^{\mathbbm{R}}}} \left[ U^2 \tmop{Ric}^{\ast}_{g_0} (\Omega)
    \right] \Omega\,, \forall U \in \mathbbm{T}^K_{g_0} \right\} .
  \end{eqnarray*}
\end{lemma}

\begin{proof}
  We observe first that the convexity of the set $\mathbbm{T}^{K, +}_{g_0}$
  follows directly by the convexity of the $L^2$-norm squared. We compute now
  the second variation of the functional ${\bf W}_{\Omega}$ along any
  line $t \mapsto A_t \assign A + t\, V$ with $A, V \in \mathbbm{T}^K_{g_0}$.
  Differentiating the formula (\ref{first-varW}) we infer the expansion
  \begin{eqnarray*}
    \nabla_{L^2} D\,{\bf W}_{\Omega} (A) (V, V) & = & \frac{d^2}{d t^2}
    _{\mid_{t = 0}} {\bf W}_{\Omega} (A_t)\\
    &  & \\
    & = & 2 \int_X \Big[ \big| \nabla_{g_0} (e^A V) \big|^2_{g_0} \;\,+\;\, \left\langle
    \nabla_{g_0} e^A, \nabla_{g_0} (e^A V^2) \right\rangle_{g_0}\Big] \Omega\\
    &  & \\
    & + & 4 \int_X \tmop{Tr}_{_{\mathbbm{R}}} \left[ e^{2 A} V^2
    \tmop{Ric}^{\ast}_{g_0} (\Omega) \right] \Omega \;.
  \end{eqnarray*}
  The required second variation formula follows by an integration by parts. We
  expand now the therm
  \begin{eqnarray*}
    \left\langle \nabla_{g_0} e^A, \nabla_{g_0} (e^A V^2) \right\rangle_{g_0}
    & = & \left\langle \nabla_{g_0} e^A, \nabla_{g_0} (e^A V) V \;\,+\;\, e^A V\,
    \nabla_{g_0} V \right\rangle_{g_0}\\
    &  & \\
    & = & \left\langle V \,\nabla_{g_0} e^A, \nabla_{g_0} (e^A V) \;\,+\;\, e^A
    \nabla_{g_0} V \right\rangle_{g_0}\\
    &  & \\
    & = & 2 \left\langle V \,\nabla_{g_0} e^A, \nabla_{g_0} (e^A V)
    \right\rangle_{g_0} \;\,-\;\, \big|V\, \nabla_{g_0} e^A \big|^2_{g_0} \;.
  \end{eqnarray*}
  Using the Cauchy-Schwarz and Jensen's inequalities we obtain
  \begin{eqnarray*}
    2\, \big| \left\langle V \,\nabla_{g_0} e^A, \nabla_{g_0} (e^A V)
    \right\rangle_{g_0}  \big| & \leqslant & 2\, \big|V \,\nabla_{g_0} e^A \big|_{g_0}\, \big|
    \nabla_{g_0} (e^A V) \big|_{g_0} \\
    &  & \\
    & \leqslant & \big|V \,\nabla_{g_0} e^A \big|^2_{g_0} \;\,+\;\, \big| \nabla_{g_0} (e^A V)
    \big|^2_{g_0}\;,
  \end{eqnarray*}
  and thus the inequality
  \begin{eqnarray*}
    \left\langle \nabla_{g_0} e^A, \nabla_{g_0} (e^A V^2) \right\rangle_{g_0} 
    \;\; \geqslant \;\; -\;\, 2\, \big|V \,\nabla_{g_0} e^A \big|^2_{g_0} \;\,-\;\, \big| \nabla_{g_0} (e^A V)
    \big|^2_{g_0}\; .
  \end{eqnarray*}
  We infer the estimate
  \begin{eqnarray*}
    \nabla_{L^2} D\,{\bf W}_{\Omega} (A) (V, V) \;\; \geqslant \;\; 4 \int_X \Big\{
    \tmop{Tr}_{_{^{\mathbbm{R}}}} \left[ (V e^A)^2 \tmop{Ric}^{\ast}_{g_0}
    (\Omega) \right] \;\,-\;\,\big|V e^A \nabla_{g_0} A\big|^2_{g_0} \Big\}\, \Omega\;,
  \end{eqnarray*}
  which implies the required convexity statement over the convex set
  $\mathbbm{T}_{g_0}^{K, +}$. Indeed $V e^A \in \mathbbm{T}_{g_0}^K$ thanks to
  lemma \ref{mut-exp}.
\end{proof}

\begin{corollary}
  \label{coro-Pos-Ric}Let $g_0 \in \mathcal{M}$. The functional
  $\mathcal{W}_{\Omega}$ is $G$-convex over the $G$-convex set
  \begin{eqnarray*}
    \Sigma^-_K (g_0) \;\;\assign\;\;\Big\{ g \in \Sigma_K (g_0) \mid
    \tmop{Ric}_g (\Omega) \;\geqslant\; -\, \tmop{Ric}_{g_0} (\Omega) \Big\}\;,
  \end{eqnarray*}
  inside the totally geodesic and flat sub-variety $\Sigma_K (g_0)$ of the
  non-positively curved Riemannian manifold $(\mathcal{M}, G)$. Moreover if
  $\tmop{Ric}_{g_0} (\Omega) \geqslant \varepsilon g_0$, for some $\varepsilon
  \in \mathbbm{R}_{> 0}$ then the functional $\mathcal{W}_{\Omega}$ is
  $G$-convex over the $G$-convex and non-empty sets
  \begin{eqnarray*}
    \Sigma^{\delta}_K (g_0) & \assign & \Big\{ g \in \Sigma_K (g_0) \mid
    \tmop{Ric}_g (\Omega) \;\geqslant\; \delta\, g \Big\}\,,\quad \forall \delta \in [0,
    \varepsilon)\;,
\\
\\
    \Sigma^+_K (g_0) & \assign & \Big\{ g \in \Sigma_K (g_0) \mid 2\,
    \tmop{Ric}_g (\Omega) \;+\; g_0 \,\Delta^{^{_{_{\Omega}}}}_{g_0} \log (g^{- 1}_0
    g) \;\geqslant\; 0 \Big\} \;.
  \end{eqnarray*}
\end{corollary}

\begin{proof}
  {\tmstrong{STEP I (G-convexity of the sets $\Sigma^{\ast}_K (g_0)$)}}. We
  observe first that the change of variables (\ref{chg-var}) send
  geodesics in to lines. Indeed the image of any geodesic $t \mapsto
  g_t = g\, e^{t v^{\ast}_g}$ via this map is the line $t \mapsto A_t : = A - t\,
  v^{\ast}_g / 2 \in \mathbbm{T}_{g_0}^K$. We infer that the $G$-convexity of
  the set $\Sigma^-_K (g_0)$ is equivalent to the (linear) convexity of the
  set
  \begin{eqnarray*}
    \mathbbm{T}_{g_0}^{K, -} \;\; \assign \;\; \Big\{ A \in \mathbbm{T}_{g_0}^K
    \mid \tmop{Ric}_{g_A} (\Omega) \;\geqslant\; -\; \tmop{Ric}_{g_0} (\Omega)
    \Big\}\;,
  \end{eqnarray*}
  with $g_A \assign g_0 e^{- 2 A}$. In the same way the $G$-convexity of the
  sets $\Sigma^{\delta}_K (g_0)$ and $\Sigma^+_K (g_0)$ is equivalent
  respectively to the convexity of the sets
  \begin{eqnarray*}
    \mathbbm{T}_{g_0}^{K, \delta} & \assign & \Big\{ A \in
    \mathbbm{T}_{g_0}^K \mid \tmop{Ric}_{g_A} (\Omega) \;\geqslant\; \delta\, g_A
    \Big\}\;,
\\
\\
\mathbbm{T}_{g_0}^{K, + +} & \assign & \Big\{ A \in \mathbbm{T}_{g_0}^K
    \mid \tmop{Ric}_{g_A} (\Omega) \;\geqslant \;g_0\,
    \Delta^{^{_{_{\Omega}}}}_{g_0} A \Big\}\; .
  \end{eqnarray*}
  Given any metric $g \in \mathcal{M}$ and any sections $A, B \in C^{\infty}
  (X, \tmop{End}_g (T_X))$ we define the bilinear product operation
\begin{eqnarray*}
\left\{ A, B \right\}_g \;\;\assign \; g\, \tmop{Tr}_g  \left( \nabla_{g,
\bullet} \,A \,\nabla_{g, \bullet} \,B \right)\;,
\end{eqnarray*}
and we observe the inequality $\left\{ A, A \right\}_g \geqslant 0$. This
implies the convex inequality
\begin{eqnarray*}
    \left\{ A_t, A_t \right\}_g & \leqslant & (1 \;-\; t) \left\{ A_0, A_0
    \right\}_g \;\,+\;\, t \left\{ A_1, A_1 \right\}_g\;,\\
    &  & \\
    A_t & \assign & (1 \;-\; t)\, A_0 \;\,+\;\, t \,A_1\,,\quad t \;\in\; [0, 1]\; .
\end{eqnarray*}
  for any $A_0, A_1 \in C^{\infty} (X, \tmop{End}_g (T_X))$. Indeed we observe
  the expansion
  \begin{eqnarray*}
    \left\{ A_t, A_t \right\}_g & = & \{A_t, A_0 \}_g \;\,+\;\, \left\{ A_t, t (A_1 \;-\;
    A_0) \right\}_g\\
    &  & \\
    & = & (1 \;-\; t) \left\{ A_0, A_0 \right\}_g \;\,+\;\, t \,\left\{ A_1, A_0
    \right\}_g\\
    &  & \\
    & + & \left\{ A_1 \;\,-\;\, (1 \;-\; t) (A_1 \;-\; A_0), t \,(A_1 \;-\; A_0) \right\}_g\\
    &  & \\
    & = & (1 \;-\; t) \left\{ A_0, A_0 \right\}_g \;\,+\;\, t \left\{ A_1, A_1 \right\}_g
    \\
    &  & \\
    & - & t \,(1 \;-\; t) \left\{ A_1 \;-\; A_0, A_1 \;-\; A_0 \right\}_g \;.
  \end{eqnarray*}
  Using this notation in formula (\ref{quad-RIC}) we infer the expression
  \begin{equation}
    \label{brk-RIC} \tmop{Ric}_{g_A} (\Omega) \;\;=\;\; g_0\,
    \Delta^{^{_{_{\Omega}}}}_{g_0} A \;\,-\;\, \left\{ A, A \right\}_{g_0} \;\,+\;\,
    \tmop{Ric}_{g_0} (\Omega) \;.
  \end{equation}
  We deduce the identities
  \begin{eqnarray*}
    \mathbbm{T}_{g_0}^{K, -} & = & \left\{ A \in \mathbbm{T}_{g_0}^K \mid g_0\,
    \Delta^{^{_{_{\Omega}}}}_{g_0} A \;\geqslant\; \left\{ A, A \right\}_{g_0} \;-\; 2
    \tmop{Ric}_{g_0} (\Omega) \right\}\;,\\
    &  & \\
    \mathbbm{T}_{g_0}^{K, \delta} & = & \left\{ A \in \mathbbm{T}_{g_0}^K \mid
    g_0 \,\Delta^{^{_{_{\Omega}}}}_{g_0} A \;\geqslant\; \left\{ A, A \right\}_{g_0}
    \;-\; \tmop{Ric}_{g_0} (\Omega) \;+\; \delta\, g_0\, e^{- 2 A} \right\}\;,\\
    &  & \\
    \mathbbm{T}_{g_0}^{K, + +} & = & \left\{ A \in \mathbbm{T}_{g_0}^K \mid
    \left\{ A, A \right\}_{g_0} \;\leqslant\; \tmop{Ric}_{g_0} (\Omega) \right\}\; .
  \end{eqnarray*}
  Let now $A_0$, $A_1 \in \mathbbm{T}_{g_0}^{K, \delta}$. The fact that $[A_0,
  A_1] = 0$ implies the existence of a $g_0$-orthonormal basis which
  diagonalizes simultaneously $A_0$ and $A_1$. Then the convexity of the
  exponential function implies the convex inequality
  \begin{eqnarray*}
    g_0 \,e^{- 2 A_t} \;\;\leqslant \;\; (1\; -\; t)\, g_0\; e^{- \,2\, A_0} \;\,+\;\, t\, g_0\, e^{- \,2\, A_1}\;,
  \end{eqnarray*}
  for all $t \in [0, 1]$. Using the previous convex inequalities we obtain
  \begin{eqnarray*}
    g_0 \,\Delta^{^{_{_{\Omega}}}}_{g_0} A_t & = & (1 \;-\; t)\, g_0\,
    \Delta^{^{_{_{\Omega}}}}_{g_0} A_0 \;\,+\;\, t\, g_0\, \Delta^{^{_{_{\Omega}}}}_{g_0}
    A_1 \\
    &  & \\
    & \geqslant & (1 \;-\; t) \left\{ A_0, A_0 \right\}_{g_0} \;\,+\;\, t \left\{ A_1,
    A_1 \right\}_{g_0} \;\,-\;\, \tmop{Ric}_{g_0} (\Omega)\\
    &  & \\
    & + & (1 \;-\; t)\, \delta\, g_0\, e^{- \,2\, A_0} \;\,+\;\, t\, \delta\, g_0\, e^{-\, 2\, A_1}\\
    &  & \\
    & \geqslant & \left\{ A_t, A_t \right\}_{g_0} \;\,-\;\, \tmop{Ric}_{g_0} (\Omega)
    \;\,+\;\, \delta\, g_0\, e^{-\, 2\, A_t}\;,
  \end{eqnarray*}
  for all $t \in [0, 1]$. We infer the convexity of the set
  $\mathbbm{T}_{g_0}^{K, \delta}$. The proof of the convexity of the sets
  $\mathbbm{T}_{g_0}^{K, -}$ and $\mathbbm{T}_{g_0}^{K, + +}$ is quite
  similar.
  
  {\tmstrong{STEP II (G-convexity of the functional
  $\mathcal{W}_{\Omega}$)}}. Using again the fact that the change of variables
  (\ref{chg-var}) send geodesics in to lines we infer that the
  $G$-convexity of the functional $\mathcal{W}_{\Omega}$ over the $G$-convex
  sets $\Sigma^-_K (g_0)$, $\Sigma^{\delta}_K (g_0)$, $\Sigma^+_K (g_0)$ is
  equivalent to the convexity of the functional ${\bf W}_{\Omega}$ over
  the convex sets $\mathbbm{T}_{g_0}^{K, -}$, $\mathbbm{T}_{g_0}^{K, \delta}$,
  $\mathbbm{T}_{g_0}^{K, + +}$. Let now $g \in \Sigma^-_K (g_0)$ and observe
  that
\begin{eqnarray*}
0 & \leqslant & \tmop{Ric}_g (\Omega) \;\,+\;\, \tmop{Ric}_{g_0} (\Omega)\\
    &  & \\
& = & - \;\,g_0 \,e^{- A} \tmop{div}^{^{_{_{\Omega}}}}_{g_0} \left( e^A
    \nabla_{g_0} A \right) \;\,+\;\, 2\, \tmop{Ric}_{g_0} (\Omega)\;,
\end{eqnarray*}
  thanks to the identity (\ref{exp-expr-OmRic}). Then the second variation
  formula in lemma \ref{convex-lm} implies the convexity of the functional
  ${\bf W}_{\Omega}$ over the convex set $\mathbbm{T}_{g_0}^{K, -}$. The
  convexity of the functional ${\bf W}_{\Omega}$ over
  $\mathbbm{T}_{g_0}^{K, \delta}$ is obvious at this point. We observe now the
  inclusion $\mathbbm{T}_{g_0}^{K, + +} \subset \mathbbm{T}_{g_0}^{K, +}$.
  Indeed for all $A \in \mathbbm{T}_{g_0}^{K, + +}$ and for all $U \in
  \mathbbm{T}_{g_0}^K$ hold the trivial identities
  \begin{eqnarray*}
    |U \,\nabla_{g_0} A|^2_{g_0} & = & \sum_k |U \,\nabla_{g_0, e_k} A|^2_{g_0} \\
    &  & \\
    & = & \sum_k \Big\langle \nabla_{g_0, e_k} A \,\nabla_{g_0, e_k} A \,U, U
    \Big\rangle_{g_0}\\
    &  & \\
    & = & \Big\langle \tmop{Tr}_{g_0} \big( \nabla_{g_0, \bullet} \,A\,
    \nabla_{g_0, \bullet} \,A \big) U, U \Big\rangle_{g_0} \\
    &  & \\
    & = & \tmop{Tr}_{_{^{\mathbbm{R}}}} \Big[ \tmop{Tr}_{g_0} \big( \nabla_{g_0,
    \bullet} A \nabla_{g_0, \bullet} A\big) U^2\Big] \\
    &  & \\
    & \leqslant & \tmop{Tr}_{_{^{\mathbbm{R}}}} \left[ U^2
    \tmop{Ric}^{\ast}_{g_0} (\Omega) \right]\;,
  \end{eqnarray*}
  where $(e_k)_k \subset T_{X, x}$ is a $g_0$-orthonormal basis at an
  arbitrary point $x \in X$. Then lemma \ref{convex-lm} implies the convexity
  of the functional ${\bf W}_{\Omega}$ over the convex set
  $\mathbbm{T}_{g_0}^{K, + +}$.
\end{proof}

\section{The extension of the functional ${\bf W}_{\Omega}$ to
$\overline{\mathbbm{T}}^K_{g_0}$}

We denote by $\tmop{End}_{g_0} (T_X)$ the space of $g_0$-symmetric
endomorphisms. We define the natural integral extension 
$$
{\bf W}_{\Omega}
: \overline{\mathbbm{T}}^K_{g_0} \longrightarrow (- \infty, + \infty]\;,
$$ 
of the
functional ${\bf W}_{\Omega}$ by the integral expression in the beginning
of section \ref{sec-PM} if $e^A \in \overline{\mathbbm{T}}^K_{g_0} \cap H^1
(X, \tmop{End}_{g_0} (T_X))$ and ${\bf W}_{\Omega} (A) = + \infty$
otherwise. We show now the following elementary fact.

\begin{lemma}
  \label{lsc-W}Let $g_0 \in \mathcal{M}$ such that $\tmop{Ric}_{g_0} (\Omega)
  \geqslant \varepsilon g_0$, for some $\varepsilon \in \mathbbm{R}_{> 0}$.
  Then the natural integral extension ${\bf W}_{\Omega} :
  \overline{\mathbbm{T}}^K_{g_0} \longrightarrow (- \infty, + \infty]$ of the
  functional ${\bf W}_{\Omega}$ is lower semi-continuous and bounded from
  below. Indeed for any $A \in \overline{\mathbbm{T}}^K_{g_0}$ hold the
  uniform estimate
  \begin{eqnarray*}
    {\bf W}_{\Omega} (A) \;\; \geqslant \;\; 2 \int_X \log \frac{d
    V_{\varepsilon g_0}}{\Omega} \,\Omega\; .
  \end{eqnarray*}
\end{lemma}
\begin{proof}
  We observe that a function $f$ over a metric space is l.s.c. iff $f (x)
  \leqslant \liminf_{k \rightarrow + \infty} f (x_k)$ for any convergent
  sequence $x_k \rightarrow x$ such that $\sup_k f (x_k) < + \infty$. So let
  $(A_k)_k \subset \overline{\mathbbm{T}}_{g_0}^K$ and $A \in
  \overline{\mathbbm{T}}_{g_0}^K$ such that $A_k \rightarrow A$ in $L^2 (X)$
  with $\sup_k {\bf W}_{\Omega} (A_k) < + \infty$. This combined with the assumption $\tmop{Ric}_{g_0} (\Omega) \geqslant
  \varepsilon g_0$, implies the estimates
  \begin{eqnarray}
&&    \label{unifH1}  \int_X \left[ \big| \nabla_{g_0} e^{A_k} \big|^2_{g_0} \;\,+\;\,
    \varepsilon \big|e^{A_k} \big|^2_{g_0} \right] \Omega \nonumber
\\\nonumber
\\
&\leqslant& \int_X \left[ \big|\nabla_{g_0} e^{A_k} \big|^2_{g_0} \;\,+\;\, \left\langle \tmop{Ric}^{\ast}_{g_0}
    (\Omega) e^{A_k}, e^{A_k} \right\rangle_{g_0} \right] \Omega \nonumber
\\\nonumber
\\
&\leqslant& C\;,
  \end{eqnarray}
  for some uniform constant $C$. The first uniform estimate combined with the
  Relich-Kondrachov compactness result, $H^1 (X) \subset \subset L^2 (X)$,
  implies that for every subsequence of $(e^{A_k})_k$ there exists a
  sub-subsequence convergent to $e^A$ in $L^2 (X)$. We observe indeed that the
  assumption on the $L^2$-convergence $A_k \rightarrow A$ implies that for
  every subsequence of $(e^{A_k})_k$ there exists a sub-subsequence convergent
  to $e^A$ a.e. over $X$. We infer that $e^{A_k} \longrightarrow e^A$ in $L^2
  (X)$. Then the uniform estimates (\ref{unifH1}) imply that $\nabla_{g_0}
  e^{A_k} \longrightarrow \nabla_{g_0} e^A$ weakly in $L^2 (X)$. We deduce
  \begin{eqnarray*}
    {\bf W}_{\Omega} (A) \;\; \leqslant \;\; \liminf_{k \rightarrow + \infty}\,
    {\bf W}_{\Omega} (A_k)\;,
  \end{eqnarray*}
  thanks to the weak lower semi-continuity of the $L^2$-norm. We show now the
  lower bound in the statement. For this purpose we observe the estimates
  \begin{eqnarray*}
    {\bf W}_{\Omega} (A) & \geqslant & \int_X \left[
    \tmop{Tr}_{_{\mathbbm{R}}} \left( \varepsilon\, e^{2 A} \;\,-\;\, 2\, A \right) \;\,+\;\, 2
    \log \frac{d V_{g_0}}{\Omega} \;\,-\;\, n \right] \Omega\\
    &  & \\
    & \geqslant & \int_X \left[ n (1 \;\,+\;\, \log \varepsilon) \;\,+\;\, 2 \log \frac{d
    V_{g_0}}{\Omega} \;\,-\;\, n \right] \Omega \;.
  \end{eqnarray*}
  The last estimate follows from the fact that the convex function $x \mapsto
  \varepsilon e^{2 x} - 2 x$ admits a global minimum over $\mathbbm{R}$ at the
  point $- (\log \varepsilon) / 2$ in which takes the value $1 + \log
  \varepsilon$. We infer the required lower bound.
\end{proof}

From now on we will assume that the polarization endomorphism $K$ is smooth.
We define the vector spaces
\begin{eqnarray*}
  W^{1, \infty} (\mathbbm{T}_{g_0}) & \assign & \Big\{ A \in W^{1, \infty} (X,
  \tmop{End}_{g_0} (T_X)) \mid \nabla_{_{T_{X, g_0}}} A = 0 \Big\}\;,\\
  &  & \\
  W^{1, \infty} (\mathbbm{T}_{g_0}^K) & \assign & \Big\{ A \in W^{1, \infty}
  (\mathbbm{T}_{g_0}) \mid \left[ \nabla^p_{g_0} T, A \right] = 0\,, \;T
  =\mathcal{R}_{g_0}, K, \;\forall p \in \mathbbm{Z}_{\geqslant 0} \Big\}\;,
\end{eqnarray*}
We denote by $\overline{\mathbbm{T}}^{K, + +}_{g_0}$ the $L^2$-closure of the
set convex set $\mathbbm{T}^{K, + +}_{g_0}$. The following quite elementary
lemma will be useful for convexity purposes.

\begin{lemma}
  \label{bound-M}Let $g_0 \in \mathcal{M}$ such that $\tmop{Ric}_{g_0}
  (\Omega) > 0$. Then the $L^2$-closed and convex set
  $\overline{\mathbbm{T}}^{K, + +}_{g_0}$ satisfies the inclusion
$$
\overline{\mathbbm{T}}^{K, + +}_{g_0} \;\;\subset\;\; W^{1, \infty}
  (\mathbbm{T}_{g_0}^{K, +})\;,
$$
where $W^{1, \infty} (\mathbbm{T}_{g_0}^{K, +})$ is the set of points $A \in W^{1, \infty}
     (\mathbbm{T}^K_{g_0})$ such that 
$$ \int_X |U \,\nabla_{g_0} A|^2_{g_0} \Omega
     \;\;\leqslant\;\; \int_X \tmop{Tr}_{_{^{\mathbbm{R}}}} \left[ U^2
     \tmop{Ric}^{\ast}_{g_0} (\Omega) \right] \Omega\,,\quad \forall U \;\in\;
     \overline{\mathbbm{T}}^K_{g_0}\; .
$$
\end{lemma}

\begin{proof}
  We remind that for all $A \in \mathbbm{T}_{g_0}^{K, + +}$ hold the
  inequality
  \begin{eqnarray*}
    | \nabla_{g_0} A|^2_{g_0} \;\; \leqslant \;\; \tmop{Tr}_{g_0} \tmop{Ric}_{g_0}
    (\Omega)\;,
  \end{eqnarray*}
  which implies the uniform estimate
  \begin{equation}
    \label{p-iter-M}  \left[ \int_X | \nabla_{g_0} A|^{2 p}_{g_0} \,\Omega
    \right]^{\frac{1}{2 p}} \;\;\leqslant\;\; \left[ \int_X \tmop{Tr}_{g_0}
    \tmop{Ric}_{g_0} (\Omega) \Omega \right]^{\frac{1}{2 p}}\,,
  \end{equation}
  for all $p \in \mathbbm{N}_{> 1}$. Let now $A \in \overline{\mathbbm{T}}^{K,
  + +}_{g_0}$ arbitrary and let $(A_k)_k \subset \mathbbm{T}^{K, + +}_{g_0}$
  be a sequence $L^2$-convergent to $A$. Applying the uniform estimate
  (\ref{p-iter-M}) to $A_k$ we infer that (\ref{p-iter-M}) hold also for $A$,
  by the weak $L^{2 p}$-compactness and the weak lower semi-continuity of the
  $L^{2 p}$-norm. Furthermore taking the limit as $p \rightarrow + \infty$ in
  (\ref{p-iter-M}) we infer the uniform estimate
  \begin{eqnarray*}
    \| \nabla_{g_0} A\|_{L^{\infty} (X, g_0)} \;\; \leqslant \;\; \sup_X \big[
    \tmop{Tr}_{g_0} \tmop{Ric}_{g_0} (\Omega) \big]^{\frac{1}{2}}\,,
  \end{eqnarray*}
  for all $A \in \overline{\mathbbm{T}}^{K, + +}_{g_0}$. It is clear at this
  point that $A \in L^{\infty} (X, \tmop{End}_{g_0} (T_X))$. Indeed consider
  an arbitrary coordinate ball $\mathcal{B} \subset X$ with center a point $x
  \equiv 0$ such that $|A|_{g_0} (0) < + \infty$. Then for all \ $v \in
  \mathcal{B}$ hold the inequalities
  \begin{eqnarray*}
    |A|_{g_0} (v) & \leqslant & |A|_{g_0} (0) \;\,+\;\, \int^1_0 \big| \left\langle
    \nabla_{g_0} |A|_{g_0} (t v), v \right\rangle_{g_0} \big|\, d t\\
    &  & \\
    & \leqslant & |A|_{g_0} (0) \;\,+\;\, \int^1_0 | \nabla_{g_0, v} A|_{g_0} (t v) \,d
    t\;,
  \end{eqnarray*}
  which show that $A$ is bounded. In the last inequality we used the estimate
  \begin{eqnarray*}
    | \left\langle \nabla_{g_0} |A|_{g_0}, \xi \right\rangle_{g_0} | \;\;
    \leqslant \;\; | \nabla_{g_0, \xi} A|_{g_0}\;,
  \end{eqnarray*}
  for all $\xi \in T_{X, x}$. This last follows combining the elementary
  identities
  \begin{eqnarray*}
    \xi . |A|^2_{g_0} & = & 2 \left\langle \nabla_{g_0, \xi} A, A
    \right\rangle_{g_0}\;,\\
    &  & \\
    \xi . |A|^2_{g_0} & = & 2\, |A|_{g_0} \,\xi . |A|_{g_0} \;\;=\;\; 2\, |A|_{g_0} 
    \left\langle \nabla_{g_0} |A|_{g_0}, \xi \right\rangle_{g_0} \;.
  \end{eqnarray*}
with the Cauchy-Schwartz inequality.
  It is also clear by the definition of the convex set
  $\overline{\mathbbm{T}}^{K, + +}_{g_0}$ that $A \in W^{1, \infty}
  (\mathbbm{T}_{g_0}^K)$. Moreover the inclusion $\mathbbm{T}_{g_0}^{K, + +}
  \subset \mathbbm{T}_{g_0}^{K, +}$ implies that for all $A \in
  \overline{\mathbbm{T}}^{K, + +}_{g_0}$ hold the inequality
  \begin{eqnarray*}
    \int_X |U \,\nabla_{g_0} A|^2_{g_0} \Omega \;\;\leqslant \;\; \int_X
    \tmop{Tr}_{_{^{\mathbbm{R}}}} \left[ U^2 \tmop{Ric}^{\ast}_{g_0} (\Omega)
    \right] \Omega\,,\quad \forall U \;\in\; \mathbbm{T}_{g_0}^K\; .
  \end{eqnarray*}
  Indeed this follows by the weak $L^2$-compactness and the weak lower
  semi-continuity of the $L^2$-norm. By $L^2$-density we infer the inclusion in
  the statement of lemma \ref{bound-M}. 
\end{proof}

We can show that the same result hold also for the closure of the set
$\mathbbm{T}_{g_0}^{K, +}$. However the proof of this case is slightly more
complicated and we omit it since we will not use it.

\begin{lemma}
  \label{fin-sc-Ext}Consider any $g_0 \in \mathcal{M}$ such that
  $\tmop{Ric}_{g_0} (\Omega) > 0$. Then the natural integral extension
  ${\bf W}_{\Omega} : \overline{\mathbbm{T}}^{K, + +}_{g_0}
  \longrightarrow \mathbbm{R}$ of the functional ${\bf W}_{\Omega}$ is
  lower semi-continuous, uniformly bounded from below and convex over the
  $L^2$-closed and convex set $\overline{\mathbbm{T}}^{K, + +}_{g_0}$.
\end{lemma}

\begin{proof}
  Thanks to lemma \ref{lsc-W} we just need to show the convexity of the
  natural integral extension ${\bf W}_{\Omega} :
  \overline{\mathbbm{T}}^{K, + +}_{g_0} \longrightarrow \mathbbm{R}$. We
  consider for this purpose an arbitrary segment $t \in [0, 1] \longmapsto
  A_t \assign A + t \,V \in \overline{\mathbbm{T}}^{K, + +}_{g_0} \subset W^{1,
  \infty} (\mathbbm{T}_{g_0}^{K, +})$. In particular the fact that $A_t \in
  W^{1, \infty} (\mathbbm{T}_{g_0}^K)$ combined with the expression
  \begin{eqnarray*}
    {\bf W}_{\Omega} (A_t) & = & \int_X \Big[ \,\big|e^{t V} e^A \nabla_{g_0} A_t
    \big|^2_{g_0} \;\,+\;\, \tmop{Tr}_{_{\mathbbm{R}}} \left( e^{2 t V} e^{2 A}
    \tmop{Ric}^{\ast}_{g_0} (\Omega) \;\,-\;\, 2\, A_t \right)\Big] \Omega\\
    &  & \\
    & + & \int_X \left[ 2 \log \frac{d V_{g_0}}{\Omega} \;\,-\;\, n \right] \Omega\;,
  \end{eqnarray*}
  implies that the function $t \in [0, 1] \longmapsto
  {\bf W}_{\Omega} (A_t) \in \mathbbm{R}$ is of class $C^{\infty}$
  over the time interval $[0, 1]$. Moreover $e^{A_t} \in W^{1, \infty}
  (\mathbbm{T}_{g_0}^K)$ for all $t \in [0, 1]$ thanks to the argument in the
  proof of lemma \ref{mut-exp}. This is all we need in order to apply to $A_t
  \in W^{1, \infty} (\mathbbm{T}_{g_0}^K)$ the first order computations in the
  proof of lemma \ref{convex-lm} which provide the second variation formula
  \begin{eqnarray*}
&&    \frac{d^2}{d t^2} {\bf W}_{\Omega} (A_t) 
\\
\\
& = & 4 \int_X \Big\{
    \tmop{Tr}_{_{^{\mathbbm{R}}}} \left[ (V e^{A_t})^2 \tmop{Ric}^{\ast}_{g_0}
    (\Omega) \right] \;\,-\;\, \big|V e^{A_t} \nabla_{g_0} A_t \big|^2_{g_0} \Big\} \,\Omega\\
    &  & \\
    & + & 2 \int_X \Big[ \big| \nabla_{g_0} (e^{A_t} V) \big|^2_{g_0} \;\,+\;\, \big|V\, \nabla_{g_0}
    e^{A_t} \big|^2_{g_0} \;\,+\;\, 2 \left\langle V\, \nabla_{g_0} e^{A_t}, \nabla_{g_0}
    (e^{A_t} V) \right\rangle_{g_0}\Big] \Omega\\
    &  & \\
    & \geqslant & 4 \int_X \Big\{ \tmop{Tr}_{_{^{\mathbbm{R}}}} \left[ (V
    e^{A_t})^2 \tmop{Ric}^{\ast}_{g_0} (\Omega) \right] \;\,-\;\, \big|V e^{A_t}
    \nabla_{g_0} A_t \big|^2_{g_0} \Big\}\, \Omega\;,
  \end{eqnarray*}
  thanks to the Cauchy-Schwarz and Jensen's inequalities. We show now that
  \begin{equation}
    \label{L2-alg-T} U \,e^A \;\;\in\;\; \overline{\mathbbm{T}}_{g_0}^K\; .
  \end{equation}
  for all $U \in \overline{\mathbbm{T}}_{g_0}^K$ and all $A \in
  \overline{\mathbbm{T}}_{g_0}^{K, + +}$. Indeed let $(A_j)_j \subset
  \mathbbm{T}_{g_0}^{K, + +}$ and $(U_j)_j \subset \mathbbm{T}_{g_0}^K$ be two
  sequences convergent respectively to $A$ and $U$ in the $L^2$-topology. From a
  computation in the proof of lemma \ref{bound-M} we know that the uniform
  estimate
$$
| \nabla_{g_0} A_j |^2_{g_0} \;\;\leqslant \;\;\tmop{Tr}_{_{^{\mathbbm{R}}}}
     \tmop{Ric}^{\ast}_{g_0} (\Omega)\;, 
$$
  combined with the convergence a.e implies that the sequence $(A_j)_j$ is
  bounded in norm $L^{\infty}$. We infer the $L^2$-convergence
  $\mathbbm{T}_{g_0}^K \ni U_k e^{A_k} \longrightarrow U e^A \in \text{}
  \overline{\mathbbm{T}}_{g_0}^K$ thanks to the dominated
  convergence theorem. We observe now that the property (\ref{L2-alg-T})
  applied to $V e^{A_t}$ combined with the fact that $A_t \in W^{1, \infty}
  (\mathbbm{T}_{g_0}^{K, +})$ for all $t \in [0, 1]$ provides the inequality
  \begin{eqnarray*}
    \frac{d^2}{d t^2} {\bf W}_{\Omega} (A_t) \;\; \geqslant \;\; 0\;,
  \end{eqnarray*}
  over the time interval $[0, 1]$, which shows the required convexity
  statement.
\end{proof}

The convexity statement over the $d_G$-convex set $\overline{\Sigma}^+_K
(g_0)$ in the main theorem \ref{Main-Teo} follows directly from lemma
\ref{fin-sc-Ext} due to the fact that the change of variables
$$
g \;\in\; \overline{\Sigma}^+_K (g_0)\; \longmapsto \;A\; =\; -\, \frac{1}{2}\,
   \log (g^{- 1}_0 g) \;\in\; \overline{\mathbbm{T}}^{K, + +}_{g_0}\;, 
$$
represents a $(d_G, L^2)$-isometry map (where $L^2$ denotes the constant
$L^2$-product $4 \int_X \left\langle \cdot, \cdot \right\rangle_{g_0} \Omega$)
which in particular send all $d_G$-geodesics segments 
$$
t \;\in\; [0, 1]
\;\longmapsto\; g_t \;=\; g\, e^{t v^{\ast}_g} \;\in\; \overline{\Sigma}^+_K (g_0)\;,
$$
in to linear segments $t \in [0, 1] \longmapsto A_t : = A - t \,v^{\ast}_g / 2
\in \overline{\mathbbm{T}}^{K, + +}_{g_0}$.

\section{On the exponentially fast convergence of
the Soliton-Ricci-flow }

\begin{lemma}
  Let $g_0 \in \mathcal{S}^K_{_{^{\Omega,+}}}$ and let $(g_t)_{t \geqslant 0} \subset \Sigma_K (g_0)$ be
  a solution of the $\Omega$-SRF with initial data $g_0$. If there exist $\delta \in \mathbbm{R}_{> 0}$ such that
 $\tmop{Ric}_{g_t} (\Omega) \geqslant \delta
  g_t$ for all times $t \geqslant 0$, then the $\Omega$-SRF converges
  exponentially fast with all its space derivatives to a
  $\Omega$-$\tmop{ShRS}$ $g_{\tmop{RS}} \in \Sigma_K (g_0)$ as $t \rightarrow
  + \infty$.
\end{lemma}

\begin{proof}
Time deriving the $\Omega$-SRF equation by
  means of (\ref{col-vr-OmRc}) we infer the evolution formula
  \begin{eqnarray*}
    2\, \ddot{g}_t \;\; = \;\; -\;\, \Delta^{^{_{_{\Omega}}}}_{g_t}  \dot{g}_t \;\,-\;\, 2\,
    \dot{g}_t,
  \end{eqnarray*}
  and thus the evolution equation
  \begin{equation}
    \label{evol-vr-G} 2 \,\frac{d}{d t}\,  \dot{g}_t^{\ast} \;\;=\;\; -\;\,
    \Delta^{^{_{_{\Omega}}}}_{g_t}  \dot{g}^{\ast}_t \;\,-\;\, 2\, \dot{g}^{\ast}_t \;\,-\;\, 2\,
    ( \dot{g}_t^{\ast})^2 \;.
  \end{equation}
  Using this we can compute the evolution of $| \dot{g}_t |_{g_t}^2 = |
  \dot{g}^{\ast}_t |_{g_t}^2 = \tmop{Tr}_{_{\mathbbm{R}}} (
  \dot{g}_t^{\ast})^2$. Indeed we define the heat operator
  \begin{eqnarray*}
    \Box_{g_t}^{^{_{_{\Omega}}}} \;\; \assign \;\; \Delta^{^{_{_{\Omega}}}}_{g_t} \;\,+\;\,
    2 \frac{d}{d t}\;,
  \end{eqnarray*}
  and we observe the elementary identity
  \begin{eqnarray*}
    \Delta^{^{_{_{\Omega}}}}_{g_t} | \dot{g}_t |_{g_t}^2 \;\; = \;\; 2 \left\langle
    \Delta^{^{_{_{\Omega}}}}_{g_t}  \dot{g}^{\ast}_t, \dot{g}^{\ast}_t
    \right\rangle_g \;\,-\;\, 2 \,| \nabla_{g_t}  \dot{g}^{\ast}_t |_{g_t}^2 \;.
  \end{eqnarray*}
  We infer the evolution formula
  \begin{eqnarray*}
    \Box_{g_t}^{^{_{_{\Omega}}}} | \dot{g}_t |_{g_t}^2 & = & -\;\, 2\, |
    \nabla_{g_t}  \dot{g}^{\ast}_t |_{g_t}^2 \;\,-\;\, 4\, | \dot{g}_t |_{g_t}^2 \;\,-\;\, 4
    \tmop{Tr}_{_{\mathbbm{R}}} ( \dot{g}_t^{\ast})^3\\
    &  & \\
    & \leqslant & - \;\,\delta | \dot{g}_t |_{g_t}^2\;,
  \end{eqnarray*}
  thanks to the $\Omega$-SRF equation and thanks to the assumption
  $\tmop{Ric}_{g_t} (\Omega) \geqslant \delta g_t$. Applying the scalar
  maximum principle we infer the exponential estimate
  \begin{equation}
    \label{uni-exp-est} | \dot{g}_t |_{g_t} \;\;\leqslant\;\; \sup_X | \dot{g}_0
    |_{g_0} \,e^{- \delta t / 2}\,,
  \end{equation}
  for all $t \geqslant 0$. In its turn this implies the convergence of the
  integral
  \begin{eqnarray*}
    \int^{+ \infty}_0 | \dot{g}_t |_{g_t} d t \;\; \leqslant \;\; C\;,
  \end{eqnarray*}
  and thus the uniform estimate
  \begin{equation}
    \label{uni-metric}  e^{- C} g_0 \;\;\leqslant\;\; g_t \leqslant e^C g_0\;,
  \end{equation}
  for all $t \geqslant 0$, (see \cite{Ch-Kn}). Thus the convergence of the
  integral
  \begin{eqnarray*}
    \int^{+ \infty}_0 | \dot{g}_t |_{g_0} d t \;\; < \;\; +\;\, \infty\;,
  \end{eqnarray*}
  implies the existence of the metric
  \begin{eqnarray*}
    g_{\infty} \;\; \assign \;\; g_0 \;\,+\;\, \int^{+ \infty}_0 \dot{g}_t \,d t\,,
  \end{eqnarray*}
  thanks to Bochner's theorem. (The positivity of the metric $g_{\infty}$ follows from the estimate $e^{-
  C} g_0 \leqslant g_t$.) Moreover the estimate
  \begin{eqnarray*}
    |g_{\infty} \;-\; g_t |_{g_0} \;\; \leqslant \;\; \int^{+ \infty}_t | \dot{g}_s
    |_{g_0} d s \;\;\leqslant\;\; C' e^{- t}\;,
  \end{eqnarray*}
  implies the exponential convergence of the $\Omega$-SRF to $g_{\infty}$ in
  the uniform topology. We show now the $C^1(X)$-convergence. Indeed the fact
  that $(g_t)_{t \geqslant 0} \subset \Sigma_K (g_0)$ implies $\dot{g}_t \in
  \mathbbm{F}^K_{g_t}$ for all times $t \geqslant 0$ thanks to the identity
  (\ref{2flat-Scat-Space}). 
\\
Thus hold the identities $\left[ K,
  \dot{g}^{\ast}_t \right] \equiv 0$ and $\left[ K, \nabla_{g_{\tau}} 
  \dot{g}^{\ast}_{\tau} \right] \equiv 0$, which imply in their turn $\left[
  \nabla_{g_{\tau}}  \dot{g}^{\ast}_{\tau}, \dot{g}^{\ast}_t \right] \equiv 0$. 
By the variation formula (\ref{flat-vr-LC}) we deduce the identity
  \begin{equation}
    \label{inv-der} \nabla_{g_{\tau}}  \dot{g}^{\ast}_t \;\;\equiv\;\; \nabla_{g_0} 
    \dot{g}^{\ast}_t\;,
  \end{equation}
  for all $\tau, t \geqslant 0$. This combined with (\ref{evol-vr-G}) provides
  the equalities
  \begin{eqnarray*}
    2\, \frac{d}{d t}  \left( \nabla_{g_t}  \dot{g}_t^{\ast} \right) & = & 2\,
    \nabla_{g_t}  \frac{d}{d t} \, \dot{g}_t^{\ast}\\
    &  & \\
    & = & - \;\,\nabla_{g_t} \Delta^{^{_{_{\Omega}}}}_{g_t}  \dot{g}^{\ast}_t \;\,-\;\, 2\,
    \nabla_{g_t}  \dot{g}^{\ast}_t \;\,-\;\, 2\, \nabla_{g_t} ( \dot{g}_t^{\ast})^2 \\
    &  & \\
    & = & - \Delta^{^{_{_{\Omega}}}}_{g_t} \nabla_{g_t}  \dot{g}^{\ast}_t \;\,-\;\,
    \tmop{Ric}^{\ast}_g (\Omega) \bullet \nabla_{g_t}  \dot{g}^{\ast}_t 
\\
\\
&-& 2\,
    \nabla_{g_t}  \dot{g}^{\ast}_t \;\,-\;\, 4\, \dot{g}_t^{\ast} \nabla_{g_t} 
    \dot{g}^{\ast}_t \;.
  \end{eqnarray*}
  We justify the last equality. We pick geodesic coordinates centered at an
  arbitrary space time point $(x_0, t_0)$, let $(e_k)_k$ be the coordinate local tangent frame
and let $\xi, \eta$ be local vector fields with constant coefficients defined in a neighborhood of $x_0$. We expand at the space time point
  $(x_0, t_0)$ the therm
  \begin{eqnarray*}
    \nabla_{g_t, \xi} \Delta^{^{_{_{\Omega}}}}_{g_t}  \dot{g}^{\ast}_t & = & -\;\,
    \nabla_{g_t, \xi} \nabla_{g_t, e_k} \nabla_{g_t, e_k}  \dot{g}^{\ast}_t \;\,+\;\,
    \nabla_{g_t, \xi} \nabla_{g_t, e_k} e_k \;\neg\; \nabla_{g_t} 
    \dot{g}^{\ast}_t 
\\
\\
&+&
\nabla_{g_t, \xi} \nabla_{g_t, \nabla_{g_t} f_t}\, 
    \dot{g}^{\ast}_t \\
    &  & \\
    & = & - \;\,\nabla_{g_t, e_k} \nabla_{g_t, e_k} \nabla_{g_t, \xi} \,
    \dot{g}^{\ast}_t \;\,+\;\, \nabla_{g_t, \xi} \nabla_{g_t, e_k} e_k \;\neg\;
    \nabla_{g_t}  \dot{g}^{\ast}_t\\
    &  & \\
    & + & \nabla_{g_t, \nabla_{g_t} f_t} \nabla_{g_t, \xi} \, \dot{g}^{\ast}_t
    \;\,+\;\, \nabla_{g_t, \xi} \nabla_{g_t} f_t \;\neg\; \nabla_{g_t}  \dot{g}^{\ast}_t\;,
  \end{eqnarray*}
  thanks to the identity (\ref{com-cov}) and $\left[ \mathcal{R}_{g_t},
  \nabla_{g_t, e_k}  \dot{g}^{\ast}_t \right] \equiv 0, \left[
  \mathcal{R}_{g_t}, \dot{g}^{\ast}_t \right] \equiv 0, [e_k, \xi] \equiv 0$.
  Moreover at the space time point $(x_0, t_0)$ hold the identity
  \begin{eqnarray*}
    \nabla_{g_t, e_k} \nabla^2_{g_t}  \dot{g}^{\ast}_t (e_k, \xi, \eta) & = &
    \nabla_{g_t, e_k}  \big[ \nabla_{g_t, e_k} \nabla_{g_t}  \dot{g}^{\ast}_t
    (\xi, \eta) \big]\\
    &  & \\
    & = &
\nabla_{g_t, e_k} \nabla_{g_t, e_k} (\nabla_{g_t, \xi}\, 
    \dot{g}^{\ast}_t \, \eta)
\\
\\
&-&
 \nabla_{g_t, e_k}  \big[ \nabla_{g_t}  \dot{g}^{\ast}_t
    (\nabla_{g_t, e_k} \xi, \eta) \;\,+\;\, \nabla_{g_t}  \dot{g}^{\ast}_t (\xi,
    \nabla_{g_t, e_k} \eta) \big]\\
    &  & \\
    & = & \nabla_{g_t, e_k}  \big[ \nabla_{g_t, e_k} \nabla_{g_t, \xi} \,
    \dot{g}^{\ast}_t \, \eta \;\,-\;\, \nabla_{g_t}  \dot{g}^{\ast}_t (\nabla_{g_t,
    \xi} \,e_k, \eta) \big]\\
    &  & \\
    & = & \nabla_{g_t, e_k} \nabla_{g_t, e_k} \nabla_{g_t, \xi} \,
    \dot{g}^{\ast}_t \, \eta \;\,-\;\, \nabla_{g_t}  \dot{g}^{\ast}_t (\nabla_{g_t,
    e_k} \nabla_{g_t, \xi} \,e_k, \eta)\;,
  \end{eqnarray*}
  which combined with the previous expression implies the formula
  \begin{eqnarray*}
    \nabla_{g_t} \Delta^{^{_{_{\Omega}}}}_{g_t}  \dot{g}^{\ast}_t \;\; = \;\;
    \Delta^{^{_{_{\Omega}}}}_{g_t} \nabla_{g_t}  \dot{g}^{\ast}_t \;\,+\;\,
    \tmop{Ric}^{\ast}_g (\Omega) \bullet \nabla_{g_t}  \dot{g}^{\ast}_t \;.
  \end{eqnarray*}
  Thus
  \begin{eqnarray*}
    2\, \frac{d}{d t}  \left( \nabla_{g_t}  \dot{g}_t^{\ast} \right) & = & -\;\,
    \Delta^{^{_{_{\Omega}}}}_{g_t} \nabla_{g_t}  \dot{g}^{\ast}_t \;\,-\;\,
    \dot{g}^{\ast}_t \bullet \nabla_{g_t}  \dot{g}^{\ast}_t \;\,-\;\, 3\, \nabla_{g_t} 
    \dot{g}^{\ast}_t \;\,-\;\, 4\, \dot{g}_t^{\ast} \nabla_{g_t}  \dot{g}^{\ast}_t\;,
  \end{eqnarray*}
  by the $\Omega$-SRF equation. We use this to compute the evolution of the
  norm squared
  \begin{eqnarray*}
    | \nabla_{g_t}  \dot{g}^{\ast}_t |^2_{g_t} \;\; = \;\;
    \tmop{Tr}_{_{\mathbbm{R}}} \left( \nabla_{g_t, e_k}  \dot{g}^{\ast}_t\,
    \nabla_{g_t, g^{- 1}_t e^{\ast}_k}  \dot{g}^{\ast}_t \right) \;.
  \end{eqnarray*}
  Indeed at time $t_0$ hold the identities
  \begin{eqnarray*}
    \frac{d}{d t}\, | \nabla_{g_t}  \dot{g}^{\ast}_t |^2_{g_t} & = &
    \tmop{Tr}_{_{\mathbbm{R}}} \left[ 2\, e_k \;\neg\; \frac{d}{d t}  \left(
    \nabla_{g_t}  \dot{g}_t^{\ast} \right) \nabla_{g_t, e_k}  \dot{g}^{\ast}_t
    \;\,-\;\, \nabla_{g_t, e_k}  \dot{g}^{\ast}_t \,\nabla_{g_t, \dot{g}^{\ast}_t e_k} 
    \dot{g}^{\ast}_t \right]\\
    &  & \\
    & = & - \;\,\tmop{Tr}_{_{\mathbbm{R}}} \Big[ 2\, e_k\; \neg\;
    \Delta^{^{_{_{\Omega}}}}_{g_t} \nabla_{g_t}  \dot{g}^{\ast}_t\, \nabla_{g_t,
    e_k}  \dot{g}^{\ast}_t \;\,+\;\, 2\, \nabla_{g_t, \dot{g}^{\ast}_t e_k} 
    \dot{g}^{\ast}_t \,\nabla_{g_t, e_k}  \dot{g}^{\ast}_t \Big]\\
    &  & \\
    & - & \tmop{Tr}_{_{\mathbbm{R}}} \Big[ 3\, \nabla_{g_t, e_k} 
    \dot{g}_t^{\ast} \,\nabla_{g_t, e_k}  \dot{g}^{\ast}_t \;\,+\;\, 4\, \dot{g}^{\ast}_t
    \nabla_{g_t, e_k}  \dot{g}^{\ast}_t \,\nabla_{g_t, e_k}  \dot{g}^{\ast}_t
    \Big]\\
    &  & \\
    & = & - \;\,\Big\langle \Delta^{^{_{_{\Omega}}}}_{g_t} \nabla_{g_t} 
    \dot{g}^{\ast}_t, \nabla_{g_t}  \dot{g}^{\ast}_t \Big\rangle_{g_t} \;\,-\;\, 3\, |
    \nabla_{g_t}  \dot{g}^{\ast}_t |^2_{g_t} \\
    &  & \\
    & - & 2 \,\Big\langle \dot{g}^{\ast}_t \bullet \nabla_{g_t} 
    \dot{g}^{\ast}_t \;\,+\;\, 2\, \dot{g}_t^{\ast} \,\nabla_{g_t}  \dot{g}^{\ast}_t,
    \nabla_{g_t}  \dot{g}^{\ast}_t \Big\rangle_{g_t}\; .
  \end{eqnarray*}
  This combined with the elementary identity
  \begin{eqnarray*}
    \Delta^{^{_{_{\Omega}}}}_{g_t} | \nabla_{g_t}  \dot{g}^{\ast}_t |^2_{g_t}
    & = & 2 \,\Big\langle \Delta^{^{_{_{\Omega}}}}_{g_t} \nabla_{g_t} 
    \dot{g}^{\ast}_t, \nabla_{g_t}  \dot{g}^{\ast}_t \Big\rangle_{g_t} \;\,-\;\, 2\, |
    \nabla^2_{g_t}  \dot{g}^{\ast}_t |^2_{g_t}\;,
  \end{eqnarray*}
  implies the evolution formula
  \begin{eqnarray*}
    \Box_{g_t}^{^{_{_{\Omega}}}} | \nabla_{g_t}  \dot{g}^{\ast}_t |^2_{g_t} &
    = & -\;\, 2\, | \nabla^2_{g_t}  \dot{g}^{\ast}_t |^2_{g_t} \;\,-\;\, 6\, | \nabla_{g_t} 
    \dot{g}^{\ast}_t |^2_{g_t}
\\
\\
&-& 4\, \Big\langle \dot{g}^{\ast}_t \bullet
    \nabla_{g_t}  \dot{g}^{\ast}_t \;\,+\;\, 2\, \dot{g}_t^{\ast} \nabla_{g_t} 
    \dot{g}^{\ast}_t, \nabla_{g_t}  \dot{g}^{\ast}_t \Big\rangle_{g_t}\\
    &  & \\
    & \leqslant & \big[ (4 \;+\; 8\, \sqrt{n}\,) | \dot{g}_t |_{g_t} \;-\; 6 \big]\, |
    \nabla_{g_t}  \dot{g}^{\ast}_t |^2_{g_t}\\
    &  & \\
    & \leqslant & \left( C e^{- t} \;-\; 6 \right) | \nabla_{g_t} 
    \dot{g}^{\ast}_t |^2_{g_t}\;,
  \end{eqnarray*}
  thanks to the uniform exponential estimate (\ref{uni-exp-est}). An
  application of the scalar maximum principle implies the estimate
  \begin{equation}
    \label{exp-dc-der} | \nabla_{g_t}  \dot{g}^{\ast}_t |_{g_t} \;\;\leqslant\;\; C_1
    e^{- t}\;,
  \end{equation}
  for all $t \geqslant 0$, where $C_1 > 0$ is a constant uniform in time. By
  abuse of notation we will allays denote by $C$ or $C_1$ such type of
  constants. Moreover the identity (\ref{inv-der}) for $\tau = t$ combined
  with (\ref{uni-metric}) provides the exponential estimate
  \begin{eqnarray*}
    | \nabla_{g_0}  \dot{g}^{\ast}_t |_{g_0} \;\; \leqslant \;\; C_1 e^{- t}\; .
  \end{eqnarray*}
  We observe now the trivial decomposition
  \begin{eqnarray*}
    \nabla_{g_0} (g^{- 1}_0  \dot{g}_t) \;\; = \;\; \nabla_{g_0} (g^{- 1}_0 g_t)\,
    \dot{g}^{\ast}_t \;\,+\;\, g^{- 1}_0 g_t \nabla_{g_0} \, \dot{g}^{\ast}_t \;.
  \end{eqnarray*}
  Thus if we set $N_t \assign | \nabla_{g_0} (g^{- 1}_0 g_t) |_{g_0}$ we infer
  the first order differential inequality
  \begin{eqnarray*}
    \dot{N}_t & \leqslant & \sqrt{n} \,N_t\, | \dot{g}^{\ast}_t |_{g_0} \;\,+\;\, \sqrt{n}\,
    |g^{- 1}_0 g_t |_{g_0} | \,\nabla_{g_0}  \dot{g}^{\ast}_t |_{g_0}\\
    &  & \\
    & \leqslant & C\, N_t\, e^{- t} \;\,+\;\, C\, e^{- 2 t}\;,
  \end{eqnarray*}
  and thus
  \begin{eqnarray*}
    N_t \;\; \leqslant\;\; e^{C \int^t_0 e^{- s} d s}  \left[ N_0 \;+\; C \int^t_0 e^{-
    2 s} d s \right] \;\;\leqslant\;\; e^C\, (N_0 \;+\; C)\;,
  \end{eqnarray*}
  by Gronwall's inequality. We deduce in conclusion the exponential estimate
  $\dot{N}_t \leqslant C_1 e^{- t}$, i.e,
  \begin{eqnarray*}
    | \nabla_{g_0}  \dot{g}_t |_{g_0} \;\;\leqslant \;\; C_1 e^{- t}\;,
  \end{eqnarray*}
  for all $t \geqslant 0$. We infer the convergence of the integral
  \begin{eqnarray*}
    \int^{+ \infty}_0 | \nabla_{g_0}  \dot{g}_t |_{g_0}\, d t \;\; < \;\; +\;\, \infty\;,
  \end{eqnarray*}
  and thus the existence of the tensor
  \begin{eqnarray*}
    A_1 \;\; \assign \;\; \int^{+ \infty}_0 \nabla_{g_0}  \dot{g}_t\, d t\;,
  \end{eqnarray*}
  thanks to Bochner's theorem. Moreover hold the exponential estimate
  \begin{eqnarray*}
    |A_1 \;-\; \nabla_{g_0} g_t |_{g_0} \;\;\leqslant \;\; \int^{+ \infty}_t |
    \nabla_{g_0}  \dot{g}_s |_{g_0} \,d s \;\;\leqslant\;\; C'\, e^{- t}\; .
  \end{eqnarray*}
  A basic calculus fact implies that $A_1 = \nabla_{g_0} g_{\infty}$. In order
  to obtain the convergence of the higher order derivatives we need to combine
  the uniform $C^1$-estimate obtained so far with an interpolation method
  based in Hamilton's work (see \cite{Ham}). The details will be explained in the
  next more technical sections. In conclusion taking the limit as $t
  \rightarrow + \infty$ in the $\Omega$-SRF equation we deduce that
  $g_{\infty} = g_{\tmop{RS}}$ is a $\Omega$-ShRS.
\end{proof}

\section{The commutator $\left[ \nabla^p, \Delta^{\Omega} \right]$ along the
$\Omega$-Soliton-Ricci flow}

We introduce first a few product notations. Let $g$ be a metric over a vector
space $V$. For any $A \in (V^{\ast})^{\otimes p} \otimes V$, $B \in
(V^{\ast})^{\otimes q} \otimes V$ and for all integers $k, l$ such that $1
\leqslant l \leqslant k \leqslant q - 2$ we define the product $A \odot^g_{k,
l} B$ as
\begin{eqnarray*}
&&  (A \odot^g_{k, l} B) (u, v)
\\
\\
 & \assign & \tmop{Tr}_g  \Big[ B (v_1, \ldots,
  v_{l - 1}, \cdot, v_l, \ldots, v_{k - 1}, A (u, \cdot, v_k), v_{k + 1},
  \ldots, v_{q - 1}) \Big]\;,
\end{eqnarray*}
for all $u \equiv (u_1, \ldots, u_{p - 2})$ and $v \equiv (v_1, \ldots, v_{q -
1})$. Moreover for any $\sigma \in S_{p + l - 3}$ we define the product $A
\odot^{g, \sigma}_{k, l} B$ as
\begin{eqnarray*}
  (A \odot^{g, \sigma}_{k, l} B) (u, v) & \assign & (A \odot^g_{k, l} B)
  (\xi_{\sigma}, v_l, \ldots, v_{q - 1})\;,
\end{eqnarray*}
where $\xi \equiv (\xi_1, \ldots, \xi_{p + l - 3}) \assign (u_1, \ldots, u_{p
- 2}, v_1, \ldots, v_{l - 1})$. We notice also that $A \odot^{g, \sigma}_{k,
l} B \equiv A \odot^g_{k, l} B$ if $p + l - 3 \leqslant 1$. Finally we define
\begin{eqnarray*}
  A \hat{\neg}_g B & \assign & \sum^{q - 2}_{k = 1} A \odot^g_{k, 1} B\; .
\end{eqnarray*}
Let now $(X, g)$ be a Riemannian manifold and let $A \in C^{\infty} (X,
(T^{\ast}_X)^{\otimes p} \otimes T_X)$. We remind the classic formula
\begin{equation}
  \label{gen-com-der}  \big[ \nabla_{g, \xi}, \nabla_{g, \eta} \big] A \;\;=\;\;
  \big[ \mathcal{R}_g (\xi, \eta), A \big] \;\,-\;\,\mathcal{R}_g (\xi, \eta)
  \;\hat{\neg}\; A \;\,+\;\, \nabla_{g, [\xi, \eta]} \,A\;,
\end{equation}
for all $\xi, \eta \in C^{\infty} (X, T_X)$. We show now the following lemma.

\begin{lemma}
  Let $(X, g)$ be an oriented Riemannian manifold, let $\Omega > 0$ be a
  smooth volume form over $X$ and let $A \in C^{\infty} (X,
  (T^{\ast}_X)^{\otimes p} \otimes T_X)$ such that $\left[ \mathcal{R}_g, \xi
  \neg \nabla^r_g A \right] = 0$ for all $r = 0, 1$ and $\xi \in T^{\otimes p
  + r - 1}_X$. Then hold the formula
  \begin{eqnarray*}
    \left[ \nabla_g, \Delta^{^{_{_{\Omega}}}}_g  \right] A \;\; = \;\;
    \tmop{Ric}^{\ast}_g (\Omega) \bullet \nabla_g \,A \;\,+\;\, 2\,\mathcal{R}_g 
    \;\hat{\neg}_g \nabla_g\, A \;\,+\;\, \nabla^{\ast_{\Omega}}_g \mathcal{R}_g \;
    \hat{\neg}\; A\; .
  \end{eqnarray*}
\end{lemma}

\begin{proof}
  We pick geodesic coordinates centered at an arbitrary point $x_0$. Let
  $(e_k)_k$ be the coordinate local tangent frame and let $\xi, \eta \equiv (\eta_1, \ldots, \eta_p)$ be 
local vector fields with constant coefficients defined in a neighborhood of the point $x_0$.
We expand at the point $x_0$ the therm
  \begin{eqnarray*}
    \nabla_{g, \xi}\, \Delta^{^{_{_{\Omega}}}}_g A & = & - \;\,\nabla_{g, \xi}
    \nabla_{g, e_k} \nabla_{g, e_k} A 
\\
\\
&+& \nabla_{g, \xi} \nabla_{g, e_k} e_k
    \;\neg\; \nabla_g\, A\;\, +\;\, \nabla_{g, \xi} \nabla_{g, \nabla_g f} \,A\\
    &  & \\
    & = & - \;\,\nabla_{g, e_k} \nabla_{g, \xi} \nabla_{g, e_k} A \;\,+\;\,\mathcal{R}_g
    (\xi, e_k) \;\hat{\neg}\; \nabla_{g, e_k} A\\
    &  & \\
    & + & \nabla_{g, \xi} \nabla_{g, e_k} e_k \;\neg\; \nabla_g \,A \;\,+\;\, \nabla_{g,
    \nabla_g f} \nabla_{g, \xi} \,A\\
    &  & \\
    & - & \mathcal{R}_g (\xi, \nabla_g f) \;\hat{\neg}\; A \;\,+\;\, \nabla_{g, \xi}
    \nabla_g f \;\neg\; \nabla_g\, A\\
    &  & \\
    & = & - \;\,\nabla_{g, e_k} \nabla_{g, e_k} \nabla_{g, \xi} \,A
\\
\\
&+&
    2\,\mathcal{R}_g (\xi, e_k) \;\hat{\neg}\; \nabla_{g, e_k}\, A\;\, +\;\, \nabla_{g, e_k}
    \mathcal{R}_g (\xi, e_k) \;\hat{\neg}\; A\\
    &  & \\
    & + & \nabla_{g, \xi} \nabla_{g, e_k} e_k \;\neg\; \nabla_g\, A \;\,+\;\, \left(
    \nabla_g f \;\neg\; \nabla^2_g\, A \right) \left( \xi, \cdot) \right.\\
    &  & \\
    & + & \mathcal{R}_g (\nabla_g f, \xi) \;\hat{\neg}\; A\;\, +\;\, \nabla_{g, \xi}
    \nabla_g f \;\neg\; \nabla_g\, A\;,
  \end{eqnarray*}
  thanks to the identity (\ref{gen-com-der}), to the assumptions on $A$ and to
  the identity $[e_k, \xi] \equiv 0 .$ Moreover at the point $x_0$ hold the
  identity
  \begin{eqnarray*}
    \nabla_{g, e_k} \nabla^2_g \,A \,(e_k, \xi, \eta) & = & \nabla_{g, e_k} 
    \big[ \nabla_{g, e_k} \nabla_g \,A\, (\xi, \eta) \big]\\
    &  & \\
    & = & \nabla_{g, e_k}  \Big\{ \nabla_{g, e_k}  \big[ \nabla_{g, \xi} \,A\,
    (\eta) \big] \;\,-\;\, \nabla_g \,A \,(\nabla_{g_t, e_k} \xi, \eta) \Big\} \\
    &  & \\
    & - & \sum^p_{j = 1} \nabla_{g, e_k}  \big[ \nabla_g \,A\, (\xi, \eta_1,
    \ldots, \nabla_{g, e_k} \eta_j, \ldots, \eta_p) \big]\\
    &  & \\
    & = & \nabla_{g, e_k}  \big[ \nabla_{g, e_k} \nabla_{g, \xi} \,A\, (\eta)\;\, -\;\,
    \nabla_g \,A\, (\nabla_{g, \xi} e_k, \eta) \big]\\
    &  & \\
    & = & \nabla_{g, e_k} \nabla_{g, e_k} \nabla_{g, \xi} \,A\, (\eta)\;\, -\;\, \nabla_g\,
    A \,(\nabla_{g, e_k} \nabla_{g, \xi} e_k, \eta)\;,
  \end{eqnarray*}
  which combined with the previous expression implies the required formula
\end{proof}

Applying this lemma first to $A = \dot{g}^{\ast}_t$ and then to $A =
\nabla_{g_t}  \dot{g}^{\ast}_t$ along the $\Omega$-SRF we infer the formula
\begin{eqnarray*}
  \left[ \nabla^2_{g_t}, \Delta^{^{_{_{\Omega}}}}_{g_t} \right] 
  \dot{g}^{\ast}_t & = & 2\, \nabla_{g_t}^2  \dot{g}^{\ast}_t \;\,+\;\, \dot{g}^{\ast}_t
  \;\hat{\neg}\; \nabla_{g_t}^2  \dot{g}^{\ast}_t \;\,+\;\, \nabla_{g_t} 
  \dot{g}^{\ast}_t \bullet \nabla_{g_t}  \dot{g}^{\ast}_t \\
  &  & \\
  & + & 2\,\mathcal{R}_{g_t}  \;\hat{\neg}_{g_t} \nabla_{g_t}^2  \dot{g}^{\ast}_t
  \;\,+\;\, \nabla^{\ast_{\Omega}}_{g_t} \mathcal{R}_{g_t}  \;\hat{\neg}\; \nabla_{g_t} 
  \dot{g}^{\ast}_t \;.
\end{eqnarray*}
(We observe that the presence of the curvature factor turns off the power of
the maximum principle in the exponential convergence of higher order space
derivatives along the $\Omega$-SRF.) A simple induction shows the general
formula
\begin{eqnarray*}
  \left[ \nabla^p_{g_t}, \Delta^{^{_{_{\Omega}}}}_{g_t} \right] 
  \dot{g}^{\ast}_t & = & p \,\nabla_{g_t}^p  \dot{g}^{\ast}_t \;\,+\;\, \sum_{r = 1}^p \,
  \sum_{k = 1}^r\, \sum_{\sigma \in S_{p - r + k - 1}} C_{k, \sigma}^{p, r}\,
  \nabla_{g_t}^{p - r}  \dot{g}^{\ast}_t \bullet^{\sigma}_k \nabla_{g_t}^r 
  \dot{g}^{\ast}_t\\
  &  & \\
  & + & \sum_{r = 1}^{p - 1}\,  \sum_{k = 1}^r \,\sum_{\sigma \in S_{p - r + k -
  1}} K_{k, \sigma}^{p, r} \,\nabla_{g_t}^{p - r - 1}
  \nabla^{\ast_{\Omega}}_{g_t} \mathcal{R}_{g_t} \bullet^{\sigma}_k
  \nabla_{g_t}^r  \dot{g}^{\ast}_t\\
  &  & \\
  & + & 2 \sum_{r = 2}^p \, \sum_{k = 2}^r \, \sum_{l = 1}^{k - 1} \sum_{\sigma
  \in S_{p - r + l}} Q_{k, l, \sigma}^{p, r} \nabla_{g_t}^{p - r}
  \mathcal{R}_{g_t} \odot^{g_t, \sigma}_{k, l} \nabla_{g_t}^r 
  \dot{g}^{\ast}_t\;,
\end{eqnarray*}
where $C_{k, \sigma}^{p, r}, K_{k, \sigma}^{p, r}, Q_{k, l, \sigma}^{p, r} \in
\left\{ 0, 1 \right\}$.

\section{Exponentially fast convergence of higher order space derivatives along
the $\Omega$-Soliton-Ricci flow }

We use here an interpolation method introduced by Hamilton in his proof of the
exponential convergence of the Ricci flow in \cite{Ham}. The difference with the technique in \cite{Ham} is
a more involved interpolation process due to the presence of some extra curvature therms
which seem to be alien to Hamilton's argument. We are able to perform our interpolation 
process by using some intrinsic properties of the $\Omega$-SRF.

\subsection{Estimate of the heat of the derivatives norm}

The fact that $(g_t)_{t \geqslant 0} \subset \Sigma_K (g_0)$ implies
$\dot{g}_t \in \mathbbm{F}^K_{g_t}$ for all times $t \geqslant 0$ thanks to
the identity (\ref{2flat-Scat-Space}). Thus hold the identities $\left[ K,
\nabla_{g_t}  \dot{g}^{\ast}_t \right] \equiv 0$ and $\left[ K, \nabla^p_{g_t}
\dot{g}^{\ast}_t \right] \equiv 0$, which in their turn imply 
$$
\Big[
\nabla_{g_t, \xi} \, \dot{g}^{\ast}_t, \nabla^p_{g_t}  \dot{g}^{\ast}_t \Big]
\;\;\equiv\;\; 0\;,
$$ for all $p \in \mathbbm{Z}_{\geqslant 0}$, $\xi \in
T_X$. We deduce by using the identity
(\ref{var-Pder})
\begin{equation}
  \label{inv-Pder} 2\, \dot{\nabla}^p_{g_t}  \dot{g}^{\ast}_t \;\;=\;\; -\;\, \sum_{r =
  1}^{p - 1} \, \sum_{k = 1}^r \sum_{\sigma \in S_{p - r + k - 1}} C_{k,
  \sigma}^{p, r} \,\nabla_{g_t}^{p - r}  \dot{g}^{\ast}_t \bullet^{\sigma}_k
  \nabla_{g_t}^r  \dot{g}^{\ast}_t\;,
\end{equation}
for all $t \geqslant 0$ and $p \in \mathbbm{N}_{> 0}$. This combined with
(\ref{evol-vr-G}) provides the identities
\begin{eqnarray*}
  2\, \frac{d}{d t}  \left( \nabla^p_{g_t}  \dot{g}_t^{\ast} \right) & = & -\;\,
  \sum_{r = 1}^{p - 1} \, \sum_{k = 1}^r \sum_{\sigma \in S_{p - r + k - 1}}
  C_{k, \sigma}^{p, r} \,\nabla_{g_t}^{p - r}  \dot{g}^{\ast}_t
  \bullet^{\sigma}_k \nabla_{g_t}^r  \dot{g}^{\ast}_t \;\,+\;\, 2\, \nabla^p_{g_t} 
  \frac{d}{d t}  \,\dot{g}_t^{\ast}\\
  &  & \\
  & = & - \;\,\sum_{r = 1}^{p - 1} \, \sum_{k = 1}^r \sum_{\sigma \in S_{p - r + k
  - 1}} C_{k, \sigma}^{p, r} \,\nabla_{g_t}^{p - r}  \dot{g}^{\ast}_t
  \bullet^{\sigma}_k \nabla_{g_t}^r  \dot{g}^{\ast}_t\\
  &  & \\
  & - & \nabla^p_{g_t} \Delta^{^{_{_{\Omega}}}}_{g_t}  \dot{g}^{\ast}_t \;\,-\;\, 2\,
  \nabla^p_{g_t}  \dot{g}^{\ast}_t \;\,-\;\, 2\, \nabla^p_{g_t} ( \dot{g}_t^{\ast})^2 \;.
\end{eqnarray*}
We consider now a $g_{t_0}$-orthonormal basis $(e_k)_k \subset T_{X, x_0}$ and
we define the multi-vectors $e_K \assign (e_{k_1}, \ldots, e_{k_p})$, 
$g^{- 1}_t
e^{\ast}_K \assign (g^{- 1}_t e^{\ast}_{k_1}, \ldots, g^{- 1}_t
e^{\ast}_{k_p})$. 
Then at the space time point $(x_0, t_0)$ hold the
identities
\begin{eqnarray*}
&&  \frac{d}{d t}\, | \nabla^p_{g_t}  \dot{g}^{\ast}_t |^2_{g_t} 
\\
\\
& = & \frac{d}{d
  t} \tmop{Tr}_{_{\mathbbm{R}}} \left[ \nabla^p_{g_t, e_K}  \dot{g}^{\ast}_t\,
  \nabla^p_{g_t, g^{- 1}_t e_K}  \dot{g}^{\ast}_t \right]\\
  &  & \\
  & = & \tmop{Tr}_{_{\mathbbm{R}}} \left[ e_K \;\neg\; \left( 2\, \frac{d}{d t} 
  \left( \nabla^p_{g_t}  \dot{g}_t^{\ast} \right) \nabla^p_{g_t, e_K} 
  \dot{g}^{\ast}_t \;\,-\;\, \sum_{j = 1}^p \dot{g}^{\ast}_t \bullet_j \nabla^p_{g_t} 
  \dot{g}^{\ast}_t \right) \nabla^p_{g_t, e_K}  \dot{g}^{\ast}_t \right]\\
  &  & \\
  & = & \left\langle 2\, \frac{d}{d t}  \left( \nabla^p_{g_t}  \dot{g}_t^{\ast}
  \right) \;\,-\;\, \sum_{j = 1}^p \dot{g}^{\ast}_t \bullet_j \nabla^p_{g_t} 
  \dot{g}^{\ast}_t, \nabla^p_{g_t}  \dot{g}_t^{\ast} \right\rangle_{g_t}\\
  &  & \\
  & \leqslant & - \;\,\sum_{r = 1}^{p - 1}  \,\sum_{k = 1}^r \sum_{\sigma \in S_{p
  - r + k - 1}} C_{k, \sigma}^{p, r}  \Big\langle \nabla_{g_t}^{p - r} 
  \dot{g}^{\ast}_t \bullet^{\sigma}_k \nabla_{g_t}^r  \dot{g}^{\ast}_t,
  \nabla^p_{g_t}  \dot{g}_t^{\ast} \Big\rangle_{g_t}\\
  &  & \\
  & - & \Big\langle \nabla^p_{g_t} \Delta^{^{_{_{\Omega}}}}_{g_t} 
  \dot{g}^{\ast}_t \;\,+\;\, 2 \nabla^p_{g_t}  \dot{g}^{\ast}_t \;\,+\;\, 2\, \nabla^p_{g_t} (
  \dot{g}_t^{\ast})^2, \nabla^p_{g_t}  \dot{g}^{\ast}_t \Big\rangle_{g_t} \;\,+\;\,
  p\, | \dot{g}_t |_{g_t} | \nabla^p_{g_t}  \dot{g}^{\ast}_t |^2_{g_t}\\
  &  & \\
  & \leqslant & - \;\,\Big\langle \Delta^{^{_{_{\Omega}}}}_{g_t} \nabla^p_{g_t} 
  \dot{g}^{\ast}_t, \nabla^p_{g_t}  \dot{g}^{\ast}_t \Big\rangle_{g_t} \;\,+\;\, C\, |
  \nabla^p_{g_t}  \dot{g}^{\ast}_t |^2_{g_t}\\
  &  & \\
  & + & C\, \sum_{r = 1}^{p - 1} \,| \nabla_{g_t}^{p - r}  \dot{g}^{\ast}_t
  |_{g_t} | \nabla_{g_t}^r  \dot{g}^{\ast}_t |_{g_t} | \nabla^p_{g_t} 
  \dot{g}^{\ast}_t |_{g_t}\\
  &  & \\
  & + & C\, \sum_{r = 1}^{p - 1} \,| \nabla_{g_t}^{p - r - 1}
  \nabla^{\ast_{\Omega}}_{g_t} \mathcal{R}_{g_t} |_{g_t} | \nabla_{g_t}^r 
  \dot{g}^{\ast}_t |_{g_t} | \nabla^p_{g_t}  \dot{g}^{\ast}_t |_{g_t} \\
  &  & \\
  & + & C \,\sum_{r = 2}^p \,| \nabla_{g_t}^{p - r} \mathcal{R}_{g_t} |_{g_t} |
  \nabla_{g_t}^r  \dot{g}^{\ast}_t |_{g_t} | \nabla^p_{g_t}  \dot{g}^{\ast}_t
  |_{g_t}\;,
\end{eqnarray*}
where $C > 0$ will always denote a time independent constant. We observe now
that $\mathcal{R}_{g_t} \equiv \mathcal{R}_{g_0}$ since $(g_t)_{t \geqslant 0}
\subset \Sigma_K (g_0)$. Using the standard identity
\begin{eqnarray*}
  \Delta^{^{_{_{\Omega}}}}_{g_t} | \nabla^p_{g_t}  \dot{g}^{\ast}_t |^2_{g_t}
  \;\; = \;\; 2\, \Big\langle \Delta^{^{_{_{\Omega}}}}_{g_t} \nabla^p_{g_t} 
  \dot{g}^{\ast}_t, \nabla^p_{g_t}  \dot{g}^{\ast}_t \Big\rangle_{g_t} \;\,-\;\, 2\, |
  \nabla^{p + 1}_{g_t}  \dot{g}^{\ast}_t |^2_{g_t}\;,
\end{eqnarray*}
we infer the estimate of the heat of the norm squared
\begin{eqnarray*}
&&  \Box_{g_t}^{^{_{_{\Omega}}}} | \nabla^p_{g_t}  \dot{g}^{\ast}_t |^2_{g_t} 
\\
\\
&
  \leqslant & - \;\,2\, | \nabla^{p + 1}_{g_t}  \dot{g}^{\ast}_t |^2_{g_t} \;\,+\;\, C\, |
  \nabla^p_{g_t}  \dot{g}^{\ast}_t |^2_{g_t} \\
  &  & \\
  & + & C \,\sum_{r = 1}^{p - 1} \,| \nabla_{g_t}^{p - r}  \dot{g}^{\ast}_t
  |_{g_t} | \nabla_{g_t}^r  \dot{g}^{\ast}_t |_{g_t} | \nabla^p_{g_t} 
  \dot{g}^{\ast}_t |_{g_t} \\
  &  & \\
  & + & C \,\sum_{r = 2}^{p - 1} \Big[ | \nabla_{g_t}^{p - r - 1}\,
  \nabla^{\ast_{\Omega}}_{g_t} \mathcal{R}_{g_t} |_{g_t} \;+\; | \nabla_{g_t}^{p -
  r} \mathcal{R}_{g_t} |_{g_t} \Big] | \nabla_{g_t}^r  \dot{g}^{\ast}_t
  |_{g_t} | \nabla^p_{g_t}  \dot{g}^{\ast}_t |_{g_t} \\
  &  & \\
  & + & C\,e^{- t} | \nabla_{g_t}^{p - 2} \,\nabla^{\ast_{\Omega}}_{g_t}
  \mathcal{R}_{g_t} |_{g_t} | \nabla^p_{g_t}  \dot{g}^{\ast}_t |_{g_t}\;,
\end{eqnarray*}
thanks to the exponential estimate (\ref{exp-dc-der}). Integrating with
respect to the volume form $\Omega$ we obtain the inequality
\begin{eqnarray*}
&&  2\, \frac{d}{d t} \int_X | \nabla^p_{g_t}  \dot{g}^{\ast}_t |^2_{g_t} \Omega 
\\
\\
&
  \leqslant & - 2 \int_X | \nabla^{p + 1}_{g_t}  \dot{g}^{\ast}_t |^2_{g_t}\,
  \Omega \;\,+\;\, C\, \int_X | \nabla^p_{g_t}  \dot{g}^{\ast}_t |^2_{g_t} \,\Omega\\
  &  & \\
  & + & C \,\sum_{r = 1}^{p - 1} \int_X | \nabla_{g_t}^{p - r} 
  \dot{g}^{\ast}_t |_{g_t} | \nabla_{g_t}^r  \dot{g}^{\ast}_t |_{g_t} |
  \nabla^p_{g_t}  \dot{g}^{\ast}_t |_{g_t} \Omega\\
  &  & \\
  & + & C\, \sum_{r = 2}^{p - 1} \,\int_X \Big[ | \nabla_{g_t}^{p - r - 1}\,
  \nabla^{\ast_{\Omega}}_{g_t} \mathcal{R}_{g_t} |_{g_t} \;+\; | \nabla_{g_t}^{p -
  r} \mathcal{R}_{g_t} |_{g_t} \Big] | \nabla_{g_t}^r  \dot{g}^{\ast}_t
  |_{g_t} | \nabla^p_{g_t}  \dot{g}^{\ast}_t |_{g_t} \,\Omega\\
  &  & \\
  & + & C \,e^{- t}  \left[ \int_X | \nabla_{g_t}^{p - 2}\,
  \nabla^{\ast_{\Omega}}_{g_t} \mathcal{R}_{g_t} |^2_{g_t} \,\Omega \right]^{1 /
  2}  \left[ \int_X | \nabla^p_{g_t}  \dot{g}^{\ast}_t |^2_{g_t} \,\Omega
  \right]^{1 / 2},
\end{eqnarray*}
thanks to H\"older's inequality. In order to estimate the sums we need a
Hamilton's result in \cite{Ham} which restates in our setting as follows.

\subsection{Hamilton's interpolation inequalities}

\begin{lemma}
  Let $p, q, r \in \mathbbm{R}$ with $r \geqslant 1$ such that $1 / p + 1 / q
  = 1 / r$. There exists a time independent constant $C > 0$ such that along
  the $\Omega$-SRF hold the inequality
  \begin{eqnarray*}
    \left[ \int_X | \nabla_{g_t} A|^{2 r}_{g_t} \,\Omega \right]^{1 / r} \;\;
    \leqslant \;\;
 C \left[ \int_X | \nabla^2_{g_t} A|^p_{g_t} \,\Omega \right]^{1
    / p}  \left[ \int_X | A|^q_{g_t} \,\Omega \right]^{1 / q},
  \end{eqnarray*}
  for all tensors $A$.
\end{lemma}

\begin{proof}
  The argument here is the same as in Hamilton \cite{Ham}. We include it for readers
  convenience. We set $u \assign | \nabla_{g_t} A|^{2 (r - 1)}_{g_t}$ and we
  observe the identities
  \begin{eqnarray*}
    \int_X | \nabla_{g_t} A|^{2 r}_{g_t} d V_{g_t} & = & \int_X \left\langle
    A, \nabla^{\ast}_{g_t} (u \nabla_{g_t} A) \right\rangle_{g_t} d V_{g_t}\\
    &  & \\
    & = & \int_X \left\langle A, \Delta_{g_t} A \right\rangle_{g_t} |
    \nabla_{g_t} A|^{2 (r - 1)}_{g_t} d V_{g_t}\\
    &  & \\
    & - & \int_X \left\langle A, \nabla_{g_t} u \;\neg\; \nabla_{g_t} A
    \right\rangle_{g_t} d V_{g_t},
  \end{eqnarray*}
  and the inequalities
  \begin{eqnarray*}
    | \left\langle A, \Delta_{g_t} A \right\rangle_{g_t} | & \leqslant &
    \sqrt{n}\, | A|_{g_t} | \nabla^2_{g_t} A|_{g_t},\\
    &  & \\
    | \left\langle A, \nabla_{g_t} u \;\neg\; \nabla_{g_t} A \right\rangle_{g_t} |
    & \leqslant & \sqrt{n}\, | A|_{g_t} | \nabla_{g_t} A|_{g_t} | \nabla_{g_t}
    u|_{g_t} \;.
  \end{eqnarray*}
  Expanding the vector $\nabla_{g_t} u$ we infer the identities
  \begin{eqnarray*}
    \nabla_{g_t} u & = & (r \;-\; 1)\, | \nabla_{g_t} A|^{2 (r - 2)}_{g_t}\,
    \nabla_{g_t} | \nabla_{g_t} A|^2_{g_t}\\
    &  & \\
    & = & 2 \,(r\; -\; 1)\, | \nabla_{g_t} A|^{2 (r - 2)}_{g_t} \, \Big\langle
    \nabla_{g_t, e_k} \nabla_{g_t} A, \nabla_{g_t} A \Big\rangle_{g_t} e_k\;,
  \end{eqnarray*}
  where $(e_k)_k$ is a $g_t$-orthonormal basis. Thus hold the inequality
  \begin{eqnarray*}
    | \nabla_{g_t} u|_{g_t} \;\; \leqslant \;\; 2\, (r\; -\; 1) \sqrt{n}\, | \nabla_{g_t}
    A|^{2 r - 3}_{g_t}\, | \nabla^2_{g_t} A|_{g_t} \;.
  \end{eqnarray*}
  Combining the previous inequalities we infer
  \begin{eqnarray*}
    \int_X | \nabla_{g_t} A|^{2 r}_{g_t} \,d V_{g_t} \;\;\leqslant \;\; C_{r,n} \int_X | A|_{g_t} | \nabla^2_{g_t} A|_{g_t} |
    \nabla_{g_t} A|^{2 r - 2}_{g_t} \,d V_{g_t}\;,
  \end{eqnarray*}
with $C_{r,n}:=2\, (r\; -\;
    1) \,n \;+\; \sqrt{n}$. This combined with (\ref{uni-metric}) implies the inequality
  \begin{eqnarray*}
    \int_X | \nabla_{g_t} A|^{2 r}_{g_t} \,\Omega \;\; \leqslant \;\; C \int_X |
    A|_{g_t} | \nabla^2_{g_t} A|_{g_t} | \nabla_{g_t} A|^{2 r - 2}_{g_t}\,
    \Omega \;.
  \end{eqnarray*}
  We can estimate the last integral using H\"older's inequality with
  \begin{eqnarray*}
    \frac{1}{p} \;\,+\;\, \frac{1}{q} \;\,+\;\, \frac{r - 1}{r} \;\; = \;\; 1\;,
  \end{eqnarray*}
in order to obtain
  \begin{eqnarray*}
    \int_X | \nabla_{g_t} A|^{2 r}_{g_t} \,\Omega  \;\leqslant \; C \left[ \int_X
    | \nabla^2_{g_t} A|^p_{g_t} \,\Omega \right]^{1 / p}  \left[ \int_X |
    A|^q_{g_t} \,\Omega \right]^{1 / q}  \left[ \int_X | \nabla_{g_t} A|^{2
    r}_{g_t} \,\Omega \right]^{1 - 1 / r},
  \end{eqnarray*}
  and hence the required conclusion.
\end{proof}
As a corollary of this inequality Hamilton obtains in \cite{Ham} the following two
estimates. For all $r = 1, \ldots, p - 1$, hold
\begin{eqnarray}
  \label{interpI}  \int_X | \nabla^r_{g_t} A|^{2 p / r}_{g_t} \,\Omega 
&\leqslant&
  C \left[ \max_X | A|_{g_t} \right]^{2 (p / r - 1)} \int_X | \nabla^p_{g_t}
  A|^2_{g_t} \,\Omega\;,
\\\nonumber
\\
  \label{interpII}  \int_X | \nabla^r_{g_t} A|^2_{g_t} \,\Omega 
&\leqslant &
C
  \left[ \int_X | \nabla^p_{g_t} A|^2_{g_t} \,\Omega \right]^{r / p}  \left[
  \int_X | A|^2_{g_t} \,\Omega \right]^{1 - r / p} .
\end{eqnarray}

\subsection{Interpolation of the $H^p$-norms }

We estimate now the integral
\begin{eqnarray*}
  I_1 \;\; \assign \;\; \sum_{r = 1}^{p - 1} \int_X | \nabla_{g_t}^{p - r} 
  \dot{g}^{\ast}_t |_{g_t} | \nabla_{g_t}^r  \dot{g}^{\ast}_t |_{g_t} |
  \nabla^p_{g_t}  \dot{g}^{\ast}_t |_{g_t} \,\Omega\; .
\end{eqnarray*}
Indeed using H\"older's inequality we obtain
\begin{eqnarray*}
  &  & \int_X | \nabla_{g_t}^{p - r}  \dot{g}^{\ast}_t |_{g_t} |
  \nabla_{g_t}^r  \dot{g}^{\ast}_t |_{g_t} | \nabla^p_{g_t}  \dot{g}^{\ast}_t
  |_{g_t} \,\Omega\\
  &  & \\
  & \leqslant &  \left[ \int_X | \nabla^{p - r}_{g_t}  \dot{g}^{\ast}_t |^{2
  p / (p - r)}_{g_t} \,\Omega \right]^{(p - r) / 2 p}  \times
\\
\\
&\times&\left[ \int_X |
  \nabla^r_{g_t}  \dot{g}^{\ast}_t |^{2 p / r}_{g_t} \,\Omega \right]^{r / 2 p} 
  \left[ \int_X | \nabla^p_{g_t}  \dot{g}^{\ast}_t |^2_{g_t} \,\Omega \right]^{1
  / 2} .
\end{eqnarray*}
This combined with Hamilton's inequality (\ref{interpI}) implies the estimate
\begin{eqnarray*}
  I_1 \;\;\leqslant\;\; C \int_X | \nabla^p_{g_t}  \dot{g}^{\ast}_t |^2_{g_t}\,
  \Omega \;.
\end{eqnarray*}
In a similar way we can estimate the integral
\begin{eqnarray*}
  I_2 \;\; \assign \;\; \sum_{r = 2}^{p - 1} \,\int_X | \nabla_{g_t}^{p - r - 1}
  \nabla^{\ast_{\Omega}}_{g_t} \mathcal{R}_{g_t} |_{g_t} | \nabla_{g_t}^r 
  \dot{g}^{\ast}_t |_{g_t} | \nabla^p_{g_t}  \dot{g}^{\ast}_t |_{g_t} \,\Omega \;.
\end{eqnarray*}
Indeed as before we obtain the inequality
\begin{eqnarray*}
  &  & \int_X | \nabla_{g_t}^{p - r - 1} \nabla^{\ast_{\Omega}}_{g_t}
  \mathcal{R}_{g_t} |_{g_t} | \nabla_{g_t}^r  \dot{g}^{\ast}_t |_{g_t} |
  \nabla^p_{g_t}  \dot{g}^{\ast}_t |_{g_t} \,\Omega\\
  &  & \\
  & \leqslant &  \left[ \int_X | \nabla^{p - r - 1}_{g_t}
  \nabla^{\ast_{\Omega}}_{g_t} \mathcal{R}_{g_t} |^{2 p / (p - r)}_{g_t}\,
  \Omega \right]^{(p - r) / 2 p}  \times
\\
\\
&\times&\left[ \int_X | \nabla^r_{g_t} 
  \dot{g}^{\ast}_t |^{2 p / r}_{g_t} \,\Omega \right]^{r / 2 p}  \left[ \int_X |
  \nabla^p_{g_t}  \dot{g}^{\ast}_t |^2_{g_t} \,\Omega \right]^{1 / 2} .
\end{eqnarray*}
Moreover the trivial inequality
\begin{eqnarray*}
  \frac{p}{p - r} & \leqslant & \frac{p - 3}{p - 3 - r}\;,
\end{eqnarray*}
implies the estimates
\begin{eqnarray*}
&&\int_X | \nabla^{p - r - 1}_{g_t} \,\nabla^{\ast_{\Omega}}_{g_t}
  \mathcal{R}_{g_t} |^{2 p / (p - r)}_{g_t} \,\Omega 
\\
\\
& \leqslant & \int_X \Omega
  \;\,+\;\, \int_X | \nabla^{p - r - 1}_{g_t} \,\nabla^{\ast_{\Omega}}_{g_t}
  \mathcal{R}_{g_t} |^{2 (p - 3) / (p - 3 - r)}_{g_t} \,\Omega\\
  &  & \\
  & \leqslant & \int_X \Omega \;\,+\;\, \int_X | \nabla^{p - 3}_{g_t}\,
  \nabla^{\ast_{\Omega}}_{g_t} \mathcal{R}_{g_t} |^2_{g_t} \,\Omega\;,
\end{eqnarray*}
by (\ref{interpI}) since $| \nabla^{\ast_{\Omega}}_{g_t} \mathcal{R}_{g_t}
|_{g_t} \leqslant C$ thanks to the uniform $C^1$-bound on $g_t$.
Using again (\ref{interpI}) we infer the estimate
\begin{eqnarray*}
  I_2 \;\; \leqslant \;\; C\, \sum_{r = 2}^{p - 1}  \,\left[ 1 \;+\; \int_X | \nabla^{p -
  3}_{g_t} \nabla^{\ast_{\Omega}}_{g_t} \mathcal{R}_{g_t} |^2_{g_t} \,\Omega
  \right]^{(p - r) / 2 p}  \left[ \int_X | \nabla^p_{g_t}  \dot{g}^{\ast}_t
  |^2_{g_t} \,\Omega \right]^{1 / 2 + r / 2 p} .
\end{eqnarray*}
We estimate finally the integral
\begin{eqnarray*}
  I_3 & \assign & \sum_{r = 2}^{p - 1} \int_X | \nabla_{g_t}^{p - r}
  \mathcal{R}_{g_t} |_{g_t} | \nabla_{g_t}^r  \dot{g}^{\ast}_t |_{g_t} |
  \nabla^p_{g_t}  \dot{g}^{\ast}_t |_{g_t} \,\Omega\; .
\end{eqnarray*}
As before using the trivial inequality
\begin{eqnarray*}
  \frac{p}{p - r} & \leqslant & \frac{p - 2}{p - 2 - r}\;,
\end{eqnarray*}
we obtain the estimates
\begin{eqnarray*}
  \int_X | \nabla^{p - r}_{g_t} \mathcal{R}_{g_t} |^{2 p / (p - r)}_{g_t}
  \,\Omega & \leqslant & \int_X \Omega \;\,+\;\, \int_X | \nabla^{p - r}_{g_t}
  \mathcal{R}_{g_t} |^{2 (p - 2) / (p - 2 - r)}_{g_t} \,\Omega\\
  &  & \\
  & \leqslant & \int_X \Omega \;\,+\;\, \int_X | \nabla^{p - 2}_{g_t}
  \mathcal{R}_{g_t} |^2_{g_t} \,\Omega\;,
\end{eqnarray*}
by (\ref{interpI}). Thus
\begin{eqnarray*}
  I_3 \;\; \leqslant \;\; C \,\sum_{r = 2}^{p - 1}  \left[ 1 \;+\; \int_X | \nabla^{p -
  2}_{g_t} \mathcal{R}_{g_t} |^2_{g_t} \,\Omega \right]^{(p - r) / 2 p}  \left[
  \int_X | \nabla^p_{g_t}  \dot{g}^{\ast}_t |^2_{g_t} \,\Omega \right]^{1 / 2 +
  r / 2 p} .
\end{eqnarray*}
In conclusion for all integers $p > 1$ hold the estimate
\begin{eqnarray*}
 && 2\, \frac{d}{d t} \int_X | \nabla^p_{g_t}  \dot{g}^{\ast}_t |^2_{g_t} \,\Omega 
\\
\\
&\leqslant & - \;\,2 \int_X | \nabla^{p + 1}_{g_t}  \dot{g}^{\ast}_t |^2_{g_t}
  \,\Omega \;\,+\;\, C \int_X | \nabla^p_{g_t}  \dot{g}^{\ast}_t |^2_{g_t} \,\Omega\\
  &  & \\
  & + & C \sum_{r = 2}^{p - 1} \left[ 1 \;+\; \int_X | \nabla^{p - 3}_{g_t}\,
  \nabla^{\ast_{\Omega}}_{g_t} \mathcal{R}_{g_t} |^2_{g_t} \,\Omega \right]^{(p
  - r) / 2 p}  \left[ \int_X | \nabla^p_{g_t}  \dot{g}^{\ast}_t |^2_{g_t}
  \,\Omega \right]^{1 / 2 + r / 2 p}\\
  &  & \\
  & + & C \sum_{r = 2}^{p - 1}  \left[ 1 \;+\; \int_X | \nabla^{p - 2}_{g_t}
  \mathcal{R}_{g_t} |^2_{g_t} \,\Omega \right]^{(p - r) / 2 p}  \left[ \int_X |
  \nabla^p_{g_t}  \dot{g}^{\ast}_t |^2_{g_t} \,\Omega \right]^{1 / 2 + r / 2
  p}\\
  &  & \\
  & + & C \,e^{- t}  \left[ \int_X | \nabla_{g_t}^{p - 2}\,
  \nabla^{\ast_{\Omega}}_{g_t} \mathcal{R}_{g_t} |^2_{g_t} \,\Omega \right]^{1 /
  2}  \left[ \int_X | \nabla^p_{g_t}  \dot{g}^{\ast}_t |^2_{g_t} \Omega
  \right]^{1 / 2} \;.
\end{eqnarray*}
We set $\Gamma_t \assign \nabla_{g_t} - \nabla_{g_0}$, and we observe the
inequality
\begin{eqnarray*}
  | \nabla^p_{g_t} \mathcal{R}_{g_t} |_{g_t} \;\; \leqslant \;\; C \;\,+\;\, C \,\sum_{h =
  1}^{p - 1}  \,\sum_{q = 1}^h \, \sum_{r_1 + \ldots + r_{h - q + 1} = q} \prod^{h
  - q + 1}_{j = 1} | \nabla_{g_0}^{r_j} \Gamma_t |_{g_0}\;,
\end{eqnarray*}
with $r_j = 0, \ldots, q$. Thus by using Jensen's inequality we obtain
\begin{eqnarray*}
  | \nabla^p_{g_t} \mathcal{R}_{g_t} |^2_{g_t} \;\; \leqslant \;\; C\;\, +\;\, C\, \sum_{h =
  1}^{p - 1} \, \sum_{q = 1}^h \, \sum_{r_1 + \ldots + r_{h - q + 1} = q} \prod^{h
  - q + 1}_{j = 1} | \nabla_{g_0}^{r_j} \Gamma_t |^2_{g_0}\;,
\end{eqnarray*}
with $r_j = 0, \ldots, q$. Moreover using H\"older's inequality, the
$C^1$-uniform estimate on $g_t$ and the inequality (\ref{interpI}) we infer the estimates
\begin{eqnarray*}
  \int_X \prod_{j = 1}^{h - q + 1} | \nabla_{g_0}^{r_j} \Gamma_t |^2_{g_0}
  \,\Omega & \leqslant & \prod_{j = 1}^{h - q + 1} \left[ \int_X |
  \nabla^{r_j}_{g_0} \Gamma_t |^{2 q / r_j}_{g_0} \,\Omega \right]^{r_j / q}\\
  &  & \\
  & \leqslant & C \int_X | \nabla^q_{g_0} \Gamma_t |^2_{g_0} \,\Omega\; .
\end{eqnarray*}
In its turn for all $q > 1$ hold the inequalities
\begin{eqnarray*}
&&  \int_X | \nabla^{q - 1}_{g_0} \Gamma_t |^2_{g_0} \,\Omega 
\\
& \leqslant &
  \sum^q_{r, h = 1}  \int_X | \nabla^{q - r}_{g_0} g_t |_{g_0} |
  \nabla^r_{g_0} g_t |_{g_0} | \nabla^{q - h}_{g_0} g_t |_{g_0} |
  \nabla^h_{g_0} g_t |_{g_0} \,\Omega\\
  &  & \\
  & \leqslant & \sum^q_{r, h = 1} \left[ \int_X | \nabla^{q - r}_{g_0} g_t
  |^{2 q / (q - r)}_{g_0} \,\Omega \right]^{(q - r) / 2 p}  \left[ \int_X |
  \nabla^r_{g_0} g_t |^{2 q / r}_{g_0} \,\Omega \right]^{r / 2 q} \times\\
  &  & \\
  & \times & \left[ \int_X | \nabla^{q - h}_{g_0} g_t |^{2 q / (q - h)}_{g_0}
  \,\Omega \right]^{(q - h) / 2 q}  \left[ \int_X | \nabla^h_{g_0} g_t |^{2 q /
  h}_{g_0} \,\Omega \right]^{h / 2 q}\\
  &  & \\
  & \leqslant & C \int_X | \nabla^q_{g_0} g_t |^2_{g_0} \,\Omega\; .
\end{eqnarray*}
We deduce the estimate
\begin{eqnarray*}
  \int_X | \nabla^p_{g_t} \mathcal{R}_{g_t} |^2_{g_t} \,\Omega \; \;\leqslant \;\; C\;\, +\;\,
  C\, \sum_{r = 1}^p \int_X | \nabla^r_{g_0} g_t |^2_{g_0} \,\Omega\; .
\end{eqnarray*}
In a similar way we obtain the estimate
\begin{eqnarray*}
  \int_X | \nabla_{g_t}^{p - 1}\, \nabla^{\ast_{\Omega}}_{g_t} \mathcal{R}_{g_t}
  |^2_{g_t} \,\Omega \;\;\leqslant \;\; C \;\,+\;\, C\, \sum_{r = 1}^p \int_X | \nabla^r_{g_0}
  g_t |^2_{g_0} \,\Omega \;.
\end{eqnarray*}

\subsection{Exponential decay of the $H^p$-norms}

We assume by induction the uniform exponential estimates
\begin{eqnarray*}
  \int_X | \nabla^r_{g_0}  \dot{g}_t |^2_{g_0} \,\Omega & \leqslant & C_r \,e^{-
  \theta_p t}\,,
\end{eqnarray*}
for all $r = 0, \ldots p - 1$, with $C_r, \theta_r > 0$. We deduce from the
previous subsection that for all $q = p, p + 1$ hold the estimate
\begin{eqnarray*}
  2\, \frac{d}{d t} \int_X | \nabla^q_{g_t}  \dot{g}^{\ast}_t |^2_{g_t} \,\Omega &
  \leqslant & - \;\,2 \int_X | \nabla^{q + 1}_{g_t}  \dot{g}^{\ast}_t |^2_{g_t}
  \,\Omega\;\, +\;\, C \int_X | \nabla^q_{g_t}  \dot{g}^{\ast}_t |^2_{g_t} \,\Omega\\
  &  & \\
  & + & C\, \sum_{r = 2}^{q - 1} \,\left[ \int_X | \nabla^q_{g_t} 
  \dot{g}^{\ast}_t |^2_{g_t} \,\Omega \right]^{1 / 2 + r / 2 q}\\
  &  & \\
  & + & C\, e^{- t}  \left[ 1 \;+\; \int_X | \nabla_{g_0}^{q - 1} g_t |^2_{g_0}\,
  \Omega \right]^{1 / 2}  \left[ \int_X | \nabla^q_{g_t}  \dot{g}^{\ast}_t
  |^2_{g_t} \,\Omega \right]^{1 / 2} .
\end{eqnarray*}
Thus for $q = p$ hold
\begin{eqnarray*}
  2\, \frac{d}{d t} \int_X | \nabla^p_{g_t}  \dot{g}^{\ast}_t |^2_{g_t} \,\Omega &
  \leqslant & - 2 \int_X | \nabla^{p + 1}_{g_t}  \dot{g}^{\ast}_t |^2_{g_t}
  \,\Omega \;\,+\;\, C \int_X | \nabla^p_{g_t}  \dot{g}^{\ast}_t |^2_{g_t}\, \Omega\\
  &  & \\
  & + & C\, \sum_{r = 0}^{p - 1} \,\left[ \int_X | \nabla^p_{g_t} 
  \dot{g}^{\ast}_t |^2_{g_t} \,\Omega \right]^{1 / 2 + r / 2 p} .
\end{eqnarray*}
Moreover Hamilton's inequality (\ref{interpII}) implies
\begin{equation}
  \label{Hamil-I}  \int_X | \nabla^p_{g_t}  \dot{g}^{\ast}_t |^2_{g_t} \,\Omega
  \;\;\leqslant\;\; C\, \left[ \int_X | \nabla^{p + 1}_{g_t}  \dot{g}^{\ast}_t |^2_{g_t}
  \,\Omega \right]^{p / (p + 1)} \left[ \int_X | \dot{g}^{\ast}_t |^2_{g_t}
  \,\Omega \right]^{1 / (p + 1)} .
\end{equation}
Now for any $\varepsilon > 0$ and all $x, y > 0$ hold the inequality
\begin{eqnarray*}
  x^p \,y \;\; \leqslant \;\; C\, \varepsilon \,x^{p + 1} \;\,+\;\, C \,\varepsilon^{- p}\, y^{p + 1}\;,
\end{eqnarray*}
and applying this above gives
\begin{eqnarray*}
  \int_X | \nabla^p_{g_t}  \dot{g}^{\ast}_t |^2_{g_t} \,\Omega \;\; \leqslant \;\; C\,
  \varepsilon \int_X | \nabla^{p + 1}_{g_t}  \dot{g}^{\ast}_t |^2_{g_t} \,\Omega
  \;\,+\;\, C\, \varepsilon^{- p} \int_X | \dot{g}^{\ast}_t |^2_{g_t} \,\Omega\;,
\end{eqnarray*}
and also
\begin{eqnarray}
  \label{HamiltI}  \left[ \int_X | \nabla^p_{g_t}  \dot{g}^{\ast}_t |^2_{g_t}
  \,\Omega \right]^{\alpha} \;\leqslant\; C\, \varepsilon \left[ \int_X | \nabla^{p +
  1}_{g_t}  \dot{g}^{\ast}_t |^2_{g_t} \,\Omega \right]^{\alpha} \;\,+\;\, C\,
  \varepsilon^{- p} \left[ \int_X | \dot{g}^{\ast}_t |^2_{g_t} \,\Omega
  \right]^{\alpha},\quad
\end{eqnarray}
with $\alpha \assign 1 / 2 + r / 2 p$. If we choose $\varepsilon$ sufficiently
small we deduce
\begin{eqnarray*}
  2 \,\frac{d}{d t} \int_X | \nabla^p_{g_t}  \dot{g}^{\ast}_t |^2_{g_t} \Omega &
  \leqslant & - \;\int_X | \nabla^{p + 1}_{g_t}  \dot{g}^{\ast}_t |^2_{g_t}
  \,\Omega \;\,+\;\, C \,e^{- \,\delta t}\\
  &  & \\
  & + & C \,\sum_{r = 0}^{p - 1} \,\left[ \int_X | \nabla^p_{g_t} 
  \dot{g}^{\ast}_t |^2_{g_t} \,\Omega \right]^{1 / 2 \,+\, r / 2 p} .
\end{eqnarray*}
thanks to (\ref{uni-exp-est}). In order to estimate the last integral therm we
distinguish two cases. If for some $t > 0$
\begin{eqnarray*}
  \int_X | \nabla^{p + 1}_{g_t}  \dot{g}^{\ast}_t |^2_{g_t} \,\Omega \;\;\leqslant
  \;\; 1\;,
\end{eqnarray*}
then
\begin{eqnarray*}
  C\, \sum_{r = 0}^{p - 1} \left[ \int_X | \nabla^p_{g_t}  \dot{g}^{\ast}_t
  |^2_{g_t} \,\Omega \right]^{1 / 2 + r / 2 p} & \leqslant & C' \,e^{-\, \delta t /
  (p + 1)} \;.
\end{eqnarray*}
thanks to (\ref{Hamil-I}) and (\ref{uni-exp-est}). Thus at this time hold the
estimate
\begin{equation}
  \label{L2estimate} 2 \,\frac{d}{d t} \int_X | \nabla^p_{g_t}  \dot{g}^{\ast}_t
  |^2_{g_t} \,\Omega \;\;\leqslant\;\; C\, e^{- \,\delta t / (p + 1)} \;.
\end{equation}
In the case $t > 0$ satisfies
\begin{eqnarray*}
  \int_X | \nabla^{p + 1}_{g_t}  \dot{g}^{\ast}_t |^2_{g_t} \,\Omega \;\; > \;\; 1\;,
\end{eqnarray*}
then
\begin{eqnarray*}
  C \,\sum_{r = 0}^{p - 1} \left[ \int_X | \nabla^p_{g_t}  \dot{g}^{\ast}_t
  |^2_{g_t} \,\Omega \right]^{1 / 2 + r / 2 p} \;\;\leqslant \;\; C' \,\varepsilon
  \int_X | \nabla^{p + 1}_{g_t}  \dot{g}^{\ast}_t |^2_{g_t} \,\Omega \;\,+\;\, C'\,
  \varepsilon^{- p} e^{- \,\delta t / 2}\;,
\end{eqnarray*}
thanks to (\ref{HamiltI}) and (\ref{uni-exp-est}). Thus choosing $\varepsilon
> 0$ sufficiently small we infer at this time the estimate
\begin{eqnarray*}
  2\, \frac{d}{d t} \int_X | \nabla^p_{g_t}  \dot{g}^{\ast}_t |^2_{g_t} \,\Omega \;\;
  \leqslant \;\; C\, e^{-\, \delta t / 2} \;.
\end{eqnarray*}
We deduce that the estimate (\ref{L2estimate}) hold for all times $t > 0$.
This implies the uniform estimate
\begin{equation}
  \label{unif-int-est}  \int_X | \nabla^p_{g_t}  \dot{g}_t |^2_{g_t} \,\Omega \;\;=\;\;
  \int_X | \nabla^p_{g_t}  \dot{g}^{\ast}_t |^2_{g_t} \,\Omega \;\;\leqslant\;\; C_p \;.
\end{equation}
We observe now the inequality
\begin{eqnarray*}
  | \nabla^p_{g_0}  \dot{g}_t |_{g_0}  & \leqslant & | \nabla^p_{g_t} 
  \dot{g}_t |_{g_0}
\\
\\
& +& C \,\sum_{h = 0}^{p - 1} \, \sum_{q = 0}^h  \,\sum_{r_1 +
  \ldots + r_{h - q + 1} = q} \prod^{h - q + 1}_{j = 1} | \nabla_{g_0}^{r_j}
  \Gamma_t |_{g_0} | \nabla^{p - 1 - h}_{g_0}  \dot{g}_t |_{g_0}\;,
\end{eqnarray*}
with $r_j = 0, \ldots, q$. Thus using Jensen's inequality we obtain
\begin{eqnarray*}
  | \nabla^p_{g_0}  \dot{g}_t |^2_{g_0}  & \leqslant & | \nabla^p_{g_t} 
  \dot{g}_t |^2_{g_0} 
\\
\\
&+&
 C\, \sum_{h = 0}^{p - 1} \, \sum_{q = 0}^h \, \sum_{r_1 +
  \ldots + r_{h - q + 1} = q} \prod^{h - q + 1}_{j = 1} | \nabla_{g_0}^{r_j}
  \Gamma_t |^2_{g_0} | \nabla^{p - 1 - h}_{g_0}  \dot{g}_t |^2_{g_0} \;.
\end{eqnarray*}
Let $k \assign p - 1 - h$ and $m \assign k + q \leqslant p - 1$. Then
H\"older's inequality combined with the $L^2$ estimate of $| \nabla^{p -
1}_{g_0} \Gamma_t |_{g_0}$ given in the previous subsection and combined with
Hamilton's interpolation inequality (\ref{interpI}) provides the estimate
\begin{eqnarray*}
  &  & \int_X \prod^{h - q + 1}_{j = 1} | \nabla_{g_0}^{r_j} \Gamma_t
  |^2_{g_0} | \nabla^{p - 1 - h}_{g_t}  \dot{g}_t |^2_{g_0} \,\Omega\\
  &  & \\
  & \leqslant & \prod^{h - q + 1}_{j = 1}  \left[ \int_X | \nabla_{g_0}^{r_j}
  \Gamma_t |^{2 m / r_j}_{g_0} \,\Omega \right]^{r_j / m}  \left[ \int_X |
  \nabla^k_{g_0}  \dot{g}_t |^{2 m / k}_{g_0} \,\Omega \right]^{k / m}\\
  &  & \\
  & \leqslant & C \left[ \int_X | \nabla_{g_0}^m \Gamma_t |^2_{g_0} \,\Omega
  \right]^{q / m}  \left[ \int_X | \nabla^m_{g_0}  \dot{g}_t |^2_{g_0} \,\Omega
  \right]^{k / m}\\
  &  & \\
  & \leqslant & C \left[ \int_X | \nabla_{g_0}^{m + 1} g_t |^2_{g_0} \,\Omega
  \right]^{q / m} e^{- \theta_{k, m} t}\;,
\end{eqnarray*}
$\theta_{k, m} > 0$, thanks to the inductive assumption. (When $r_j = 0$ or $k
= 0$ in the previous estimate it means just that we are taking the
corresponding $L^{\infty}$-norms, which are bounded by the uniform
$C^1$-estimate.) Thus for some $\tau_p, \rho_p > 0$, hold the estimate
\begin{eqnarray*}
  \int_X | \nabla^p_{g_0}  \dot{g}_t |^2_{g_0} \,\Omega \;\; \leqslant \;\; C\, e^{-
  \tau_p t} \int_X | \nabla^p_{g_0} g_t |^2_{g_0} \,\Omega \;\,+\;\, C \left( 1 \;+\; e^{-
  \rho_p t} \right)\;,
\end{eqnarray*}
thanks to the uniform estimate (\ref{unif-int-est}). We set
$$
N_{p, t} \;\;\assign\;\; \int_X | \nabla^p_{g_0} g_t |^2_{g_0} \,\Omega\; . 
$$
We infer the first order differential inequality
\begin{eqnarray*}
  \dot{N}_{p, t} \;\; \leqslant \;\; C \,N_{p, t}\, e^{- \tau_p t} \;\,+\;\, C \left( 1 \;+\; e^{-
  \rho_p t} \right)\;,
\end{eqnarray*}
We deduce by Gronwall's inequality
\begin{eqnarray*}
  N_{p, t} \;\; \leqslant \;\; e^{C \int^t_0 e^{- \tau_p s} d s}  \left[ N_{p, 0} \;\,+\;\,
  C \int^t_0 \left( 1 \;+\; e^{- \rho_p s} \right) d s \right] \;\;\leqslant\;\; C' (1 \;+\;
  t) \;.
\end{eqnarray*}
The fact that the function $t^{1 / 2} e^{- t}$ is bounded for $t \geqslant 0$
implies that for $q = p + 1$ hold the estimate
\begin{eqnarray*}
  2\, \frac{d}{d t} \int_X | \nabla^q_{g_t}  \dot{g}^{\ast}_t |^2_{g_t} \,\Omega &
  \leqslant & - \;\,2 \int_X | \nabla^{q + 1}_{g_t}  \dot{g}^{\ast}_t |^2_{g_t}
  \,\Omega \;\,+\;\, C \int_X | \nabla^q_{g_t}  \dot{g}^{\ast}_t |^2_{g_t} \,\Omega\\
  &  & \\
  & + & C \sum_{r = 0}^{q - 1}\, \left[ \int_X | \nabla^q_{g_t} 
  \dot{g}^{\ast}_t |^2_{g_t} \,\Omega \right]^{1 / 2 + r / 2 q} \;.
\end{eqnarray*}
We deduce from the same argument showing (\ref{unif-int-est}), the uniform
estimate
$$ \int_X | \nabla^{p + 1}_{g_t}  \dot{g}^{\ast}_t |^2_{g_t} \,\Omega \;\;\leqslant\;\;
   C_{p + 1}\;, 
$$
and thus
\begin{eqnarray*}
  \int_X | \nabla^p_{g_t}  \dot{g}_t |^2_{g_t} \,\Omega \;\;=\;\; \int_X |
  \nabla^p_{g_t}  \dot{g}^{\ast}_t |^2_{g_t} \,\Omega \;\; \leqslant \;\; C_p\, e^{-\,
  \delta t / (p + 1)}\;,
\end{eqnarray*}
thanks to (\ref{Hamil-I}) and (\ref{uni-exp-est}). Applying the previous
argument to this improved estimate we infer
\begin{eqnarray*}
  \dot{N}_{p, t} \;\;\leqslant \;\; C\, N_{p, t} \,e^{- \,\tau_p t} \;\,+\;\, C\, e^{-\, \rho_p t}\;,
\end{eqnarray*}
and thus $N_{p, t} \leqslant C$ thanks to Gronwall's inequality. We deduce the
conclusion of the induction.
\begin{eqnarray*}
  \dot{N}_{p, t} \;\;=\;\; \int_X | \nabla^p_{g_0}  \dot{g}_t |^2_{g_0} \,\Omega \;\;
  \leqslant \;\; C_p\, e^{- \theta_p t} \;.
\end{eqnarray*}
Thus using the Sobolev estimate we infer for all times $t > 0$ the inequality
\begin{eqnarray*}
  | \nabla^p_{g_0}  \dot{g}_t |_{g_0} \;\; \leqslant \;\; C_p\, e^{- \varepsilon_p t}\;,
\end{eqnarray*}
which implies the convergence of the integral
\begin{eqnarray*}
  \int^{+ \infty}_0 | \nabla^p_{g_0}  \dot{g}_t |_{g_0} \,d t \;\; < \;\; +\; \infty\;,
\end{eqnarray*}
and thus the existence of the tensor
\begin{eqnarray*}
  A_p \;\;\assign \;\; \int^{+ \infty}_0 \nabla^p_{g_0}  \dot{g}_t \,d t\;,
\end{eqnarray*}
thanks to Bochner's theorem. Moreover hold the exponential estimate
\begin{eqnarray*}
  |A_p \;-\; \nabla^p_{g_0} g_t |_{g_0} \;\;\leqslant \;\; \int^{+ \infty}_t |
  \nabla^p_{g_0}  \dot{g}_s |_{g_0} \,d s \;\;\leqslant\;\; C_p \,e^{- \,\varepsilon_p t} \;.
\end{eqnarray*}
A basic calculus fact combined with an induction on $p$ implies that $A_p =
\nabla^p_{g_0} g_{\infty}$. This concludes the proof of the exponential
convergence of the $\Omega$-SRF.

\section{Appendix}

\subsection{Weitzenb\"ock type formulas}

\begin{lemma}
  \label{Endo-Div}For any $g \in \mathcal{M}$ and $u \in C^{\infty} (X, S^2
  T^{\ast}_X)$ hold the formula
  \begin{eqnarray*}
    -\;\,\left( \nabla^{\ast_{_{\Omega}}}_g \mathcal{D}_g \,u 
    \right)_g^{\ast}  \;\; = \;\; \nabla^{\ast_{_{\Omega}}}_g \nabla_{_{T_X, g}}
    u^{\ast}_g \;\,+\;\, \left( \nabla^{\ast_{_{\Omega}}}_g \nabla_{_{T_X, g}}
    u^{\ast}_g \right)_g^T \;\,-\;\, \Delta^{^{_{_{\Omega}}}}_g u^{\ast}_g \;.
  \end{eqnarray*}
\end{lemma}

\begin{proof}
  We \ need to show first that for any $h \in C^{\infty} \left( X,
  (T_X^{\ast})^{\otimes 3} \right)$ hold the identity
  \begin{equation}
    \label{div-star}  \left( \nabla_g^{\ast} \,h \right)_g^{\ast} \;\;=\;\;
    \nabla_g^{\ast}  \left( \bullet \;\neg\; h \right)_g^{\ast}\;,
  \end{equation}
  where the section $\left( \bullet \;\neg\; h \right)_g^{\ast} \in C^{\infty} \left( X,
  (T_X^{\ast})^{\otimes 2} \otimes T_X \right)$ is given by the formula
$$
\left( \bullet \;\neg\; h \right)_g^{\ast} (\xi, \eta) \;\;\assign\;\; \left( \xi \;\neg\;
  h \right)_g^{\ast} \eta\;.
$$
In fact let $(e_l)$ be a smooth local tangent frame and let $\xi, \eta$ be arbitrary smooth vector fields 
defined in a neighborhood of an arbitrary point $x_0$ such that 
$\nabla_g \,e_l (x_0)=\nabla_g \,\xi (x_0)= \nabla_g \,\eta (x_0)=0$. Then at the point $x_0$ hold the equalities
\begin{eqnarray*}
    \nabla_g^{\ast} \,h (\xi, \eta) & = & - \;\,\nabla_g \,h\, (e_l, e_l, \xi, \eta)\\
    &  & \\
    & = & - \;\,e_l \,.\, h\, (e_l, \xi, \eta)\\
    &  & \\
    & = & - \;\,e_l \,.\, g\, \left( \left( e_l \;\neg\; h \right)_g^{\ast} \xi, \eta
    \right)
\end{eqnarray*}
\begin{eqnarray*}  
    & = & - \;\,g \left( \nabla_{g, e_l}  \left[ \left( e_l \;\neg\; h
    \right)_g^{\ast} \xi \right], \eta \right)\\
    &  & \\
    & = & - \;\,g \left( \nabla_{g, e_l}  \left( \bullet \;\neg\; h \right)_g^{\ast}
    (e_l, \xi), \eta \right)\\
    &  & \\
    & = & g \left( \nabla_g^{\ast}  \left( \bullet \;\neg\; h \right)_g^{\ast}
    \xi, \eta \right) \;.
\end{eqnarray*}
  Applying the identity (\ref{div-star}) to $h = \mathcal{D}_g u$ and using
  the expression
  \begin{eqnarray*}
    \nabla^{\ast_{_{\Omega}}}_g\mathcal{D}_g \,u  \;\; = \;\; 
    \nabla_g^{\ast} \, \mathcal{D}_g \,u \;\,+\;\, \nabla_g f \;\neg\; \mathcal{D}_g \,u\;,
  \end{eqnarray*}
 we deduce the
  formula
  \begin{eqnarray*}
    \left( \nabla^{\ast_{_{\Omega}}}_g \mathcal{D}_g \,u 
    \right)^{\ast}_g & = &  \nabla_g^{\ast}  \left( \bullet \;\neg\;
    \mathcal{D}_g \,u \right)_g^{\ast} \;\,+\;\, \left( \nabla_g f \;\neg\; \mathcal{D}_g \,u
    \right)_g^{\ast} \;.
  \end{eqnarray*}
  In order to explicit this expression, for any $\mu \in C^{\infty} (X, T_X),$
  we expand the therm
  \begin{eqnarray*}
    \mathcal{D}_g \,u \,(\mu, \xi, \eta) & = & \nabla_g \,u \,(\xi, \mu, \eta) \;\,+\;\,
    \nabla_g \,u\, (\eta, \mu, \xi) \;\,-\;\, \nabla_g \,u\; (\mu, \xi, \eta)\\
    &  & \\
    & = & g \left( \nabla_g \,u^{\ast}_g \,(\xi, \mu), \eta \right) \;\,+\;\, g \left(
    \nabla_g \,u^{\ast}_g \,(\eta, \mu), \xi \right) 
\\
\\
&-&
g \left( \nabla_g\,
    u^{\ast}_g \,(\mu, \xi), \eta \right)\\
    &  & \\
    & = & g \left( \nabla_{_{T_X, g}} u^{\ast}_g \,(\xi, \mu), \eta \right) \;\,+\;\, g
    \left( \nabla_{_{T_X, g}} u^{\ast}_g \,(\eta, \mu), \xi \right) 
\\
\\
&+& 
g \left(
    \nabla_g \,u^{\ast}_g \,(\mu, \eta), \xi \right)\\
    &  & \\
    & = & - \;\,g \left( \nabla_{_{T_X, g}} u^{\ast}_g \,(\mu, \xi), \eta \right) \;\,-\;\,
    g \left( \nabla_{_{T_X, g}} u^{\ast}_g \,(\mu, \eta), \xi \right) 
\\
\\
&+& 
g \left(
    \nabla_g \,u^{\ast}_g \,(\mu, \eta), \xi \right)\\
    &  & \\
    & = & - \;\,g \left( \nabla_{_{T_X, g}} u^{\ast}_g \,(\mu, \xi), \eta \right) \;\,-\;\,
    g \left( \left( \mu \;\neg\; \nabla_{_{T_X, g}} u^{\ast}_g \right)_g^T \xi,
    \eta \right)
\\
\\
& +& g \left( \nabla_g \,u^{\ast}_g \,(\mu, \xi), \eta \right) \;.
  \end{eqnarray*}
  We deduce the identity
  \begin{eqnarray*}
    \left( \bullet \;\neg\; \mathcal{D}_g \,u \right)_g^{\ast} \;\; = \;\; -\;\,
    \nabla_{_{T_X, g}} u^{\ast}_g \;\,-\;\, \left( \nabla_{_{T_X, g}} u^{\ast}_g
    \right)_g^T \;\,+\;\, \nabla_g \,u^{\ast}_g\;,
  \end{eqnarray*}
  and thus the expression
  \begin{eqnarray*}
    -\;\,\left(\nabla^{\ast_{_{\Omega}}}_g\mathcal{D}_g \,u 
    \right)^{\ast}_g & = & \nabla^{\ast_{_{\Omega}}}_g \nabla_{_{T_X, g}}
    u^{\ast}_g \;\,+\;\, \nabla_g^{\ast}  \left( \nabla_{_{T_X, g}} u^{\ast}_g
    \right)_g^T \\
    &  & \\
    & + & \left( \nabla_g f \;\neg\; \nabla_{_{T_X, g}} u^{\ast}_g \right)_g^T \;\,-\;\,
    \nabla_g^{\ast_{_{\Omega}}} \nabla_g \,u^{\ast}_g \;.
  \end{eqnarray*}
  We observe now that at the point $x_0$ hold the following equalities
  \begin{eqnarray*}
    \nabla_g^{\ast}  \left( \bullet \;\neg\; \nabla_{_{T_X, g}} u^{\ast}_g
    \right)_g^T \xi & = & - \;\,\nabla_{g, e_l} \left( \nabla_{_{T_X, g}}
    u^{\ast}_g \right)_g^T (e_l, \xi)\\
    &  & \\
    & = & - \;\,\nabla_{g, e_l} \left[ \left( e_l \;\neg\; \nabla_{_{T_X, g}}
    u^{\ast}_g \right)_g^T \xi \right]\\
    &  & \\
    & = & - \;\,\left[ \nabla_{g, e_l} \left( e_l \;\neg\; \nabla_{_{T_X, g}}
    u^{\ast}_g \right)_g^T \right] \xi\\
    &  & \\
    & = &  \left( \nabla_g^{\ast} \,\nabla_{_{T_X, g}} u^{\ast}_g \right)_g^T
    \xi\;,
  \end{eqnarray*}
  since at this point hold the identities
  \begin{eqnarray*}
    \nabla_{g, e_l} \left( e_l \;\neg\; \nabla_{_{T_X, g}} u^{\ast}_g \right) \eta
    & = & \nabla_{g, e_l} \left[ \nabla_{_{T_X, g}} u^{\ast}_g \,(e_l, \eta)
    \right]\\
    &  & \\
    & = & \nabla_{g, e_l} \nabla_{_{T_X, g}} u^{\ast}_g \,(e_l, \eta) \;.
  \end{eqnarray*}
  We infer the required formula.
\end{proof}

We define the Hodge Laplacian (resp. the $\Omega$-Hodge Laplacian) operators acting on $T_X$-valued
$q$-forms by the formulas
\begin{eqnarray*}
  \Delta_{_{T_{X, g}}} & : = & \nabla_{_{T_{X, g}}} \nabla^{\ast}_g \;\,+\;\,
  \nabla_g^{\ast} \,\nabla_{_{T_{X, g}}}\;,\\
  &  & \\
  \Delta^{^{_{_{\Omega}}}}_{_{T_{X, g}}} & \assign & \nabla_{_{T_{X, g}}}
  \nabla^{\ast_{_{\Omega}}}_g \;\,+\;\, \nabla^{\ast_{_{\Omega}}}_g \nabla_{_{T_{X,
  g}}} \;.
\end{eqnarray*}
\begin{lemma}
  \label{TX-Lap-RmLap}Let $(X, g)$ be a Riemannian manifold and let $H \in
  C^{\infty} (X, \tmop{End} (T_X))$. Then hold the identity
\begin{eqnarray*}
    \Delta_{_{T_{X, g}}} H \;\; = \;\; \Delta_g \,H \;\,-\;\,\mathcal{R}_g \ast H \;\,+\;\, H
    \tmop{Ric}^{\ast}_g\;,
\end{eqnarray*}
  Where $\left( \mathcal{R}_g \ast H \right) \xi \assign \tmop{Tr}_g \left[
  \left( \xi \;\neg\; \mathcal{R}_g \right) H \right]$ for all $\xi \in T_X$.
\end{lemma}

\begin{proof}
  Let $(e_k)_k$ and $\xi$ as in the proof of lemma \ref{Endo-Div}. Then at
  the point $x_0$ hold the equalities
  \begin{eqnarray*}
    \Delta_{_{T_{X, g}}} H \,\xi & = & \nabla_{g, \xi} \nabla^{\ast}_g \,H \;\,-\;\,
    \nabla_{g, e_k}  \left( \nabla_{_{T_{X, g}}} H \right) (e_k, \xi)\\
    &  & \\
    & = & - \;\,\nabla_{g, \xi} \nabla_{g, e_k} H \,e_k \;\,-\;\, \nabla_{g, e_k}  \left(
    \nabla_{g, e_k} H \,\xi \;\,-\;\, \nabla_{g, \xi}\, H\, e_k  \right)\\
    &  & \\
    & = & \Delta_g H \,\xi \;\,+\;\, \nabla_{g, e_k} \nabla_{g, \xi} \,H\, e_k \;\,-\;\, \nabla_{g,
    \xi} \nabla_{g, e_k} H \,e_k\\
    &  & \\
    & = & \Delta_g H \,\xi \;\,+\;\,\mathcal{R}_g (e_k, \xi) \,H \,e_k \;\,-\;\, H\,\mathcal{R}_g
    (e_k, \xi) \,e_k\;,
  \end{eqnarray*}
  which shows the required formula.
\end{proof}

\begin{lemma}
  \label{OmTX-Lap-RmLap}Let $(X, g)$ be a orientable Riemannian manifold, let
  $\Omega > 0$ be a smooth volume form and let $H \in C^{\infty} (X,
  \tmop{End} (T_X))$. Then
  \begin{eqnarray*}
    \Delta^{^{_{_{\Omega}}}}_{_{T_{X, g}}} H & = & \Delta^{^{_{_{\Omega}}}}_g
    H \;\,-\;\,\mathcal{R}_g \ast H \;\,+\;\, H \tmop{Ric}^{\ast}_g (\Omega)\;.
  \end{eqnarray*}

\end{lemma}

\begin{proof}
  We expand the Laplacian
  \begin{eqnarray*}
    \Delta^{^{_{_{\Omega}}}}_{_{T_{X, g}}} H & = & \nabla_{_{T_{X, g}}}
    \nabla^{\ast}_g \,H \;\,+\;\, \nabla_{_{T_{X, g}}} \left( H \nabla_g f \right) 
\\
\\
&+&
    \nabla_g^{\ast}\, \nabla_{_{T_{X, g}}} H \;\,+\;\, \nabla_g f \;\neg\; \nabla_{_{T_{X,
    g}}} H\\
    &  & \\
    & = & \Delta_{_{T_{X, g}}} H \;\,+\;\, \nabla_g \,H \,\nabla_g f \;\,+\;\, H \,\nabla^2_g f 
\\
\\
&+& \nabla_g f \;\neg\; \nabla_g\, H \;\,-\;\, \nabla_g \,H\, \nabla_g f\\
    &  & \\
    & = & \Delta^{^{_{_{\Omega}}}}_g H \;\,-\;\,\mathcal{R}_g \ast H \;\,+\;\, H
    \tmop{Ric}^{\ast}_g (\Omega)\;,
  \end{eqnarray*}
  thanks to lemma \ref{TX-Lap-RmLap}.
\end{proof}

\subsection{The first variation of the pre-scattering operator}

\begin{lemma}
  \label{var-pre-scat}For any smooth family of metrics $(g_t)_{t \in
  \mathbbm{R}} \subset \mathcal{M}$ hold the total variation formula
  
  \begin{eqnarray*}
    2 \,\frac{d}{d t}  \left[ \nabla_{_{T_X, g_t}} \tmop{Ric}^{\ast}_{g_t}
    (\Omega) \right] & = & \nabla_{_{T_X, g_t}} \left(
    \nabla^{\ast_{_{\Omega}}}_{g_t} \nabla_{_{T_X, g_t}}  \dot{g}^{\ast}_t
    \right)_{g_t}^T \;\,+\;\, \underline{\tmop{div}}^{^{_{_{\Omega}}}}_{g_t}  \left[
    \mathcal{R}_{g_t}, \dot{g}^{\ast}_t \right] \\
    &  & \\
    & + & \tmop{Alt} \left[ \mathcal{R}_{g_t} \circledast \left( \nabla_{g_t} -
    \nabla_{_{T_X, g_t}} \right)  \dot{g}^{\ast}_t 
    \right] \\
    &  & \\
&+&\tmop{Alt} \left[\left( \nabla_{_{T_X,
    g_t}}  \dot{g}^{\ast}_t \right)_{g_t}^T \tmop{Ric}^{\ast}_{g_t} (\Omega)
    \right]
\\
\\
& - & \tmop{Ric}^{\ast}_{g_t} (\Omega) \;\neg\; \nabla_{_{T_X, g_t}} 
    \dot{g}^{\ast}_t \;\,-\;\, 2\, \dot{g}_t^{\ast}\, \nabla_{_{T_X, g_t}}
    \tmop{Ric}^{\ast}_{g_t} (\Omega) \;.
  \end{eqnarray*}
\end{lemma}

\begin{proof}
  Lemma \ref{Endo-Div} combined with the
  variation formulas (\ref{var-Om-Ric}) and (\ref{var-extD}) implies the equality
  \begin{eqnarray*}
    2\, \frac{d}{d t}  \left[ \nabla_{_{T_X, g_t}} \tmop{Ric}^{\ast}_{g_t}
    (\Omega) \right] & = & \nabla_{_{T_X, g_t}} \left[
    \nabla^{\ast_{_{\Omega}}}_{g_t} \nabla_{_{T_X, g_t}}  \dot{g}^{\ast}_t \;\,+\;\,
    \left( \nabla^{\ast_{_{\Omega}}}_{g_t} \nabla_{_{T_X, g_t}} 
    \dot{g}^{\ast}_t \right)_{g_t}^T \right] \\
    &  & \\
    & - & \nabla_{_{T_X, g_t}}  \Big[ \Delta^{^{_{_{\Omega}}}}_{g_t} 
    \dot{g}^{\ast}_t \;\,+\;\, \dot{g}_t^{\ast} \tmop{Ric}^{\ast}_{g_t} (\Omega)
    \Big] 
\\
\\
&-& \dot{g}_t^{\ast} \,\nabla_{_{T_X, g_t}} \tmop{Ric}^{\ast}_{g_t}
    (\Omega)
\;\, - \;\, \tmop{Ric}^{\ast}_{g_t} (\Omega) \;\neg\; \nabla_{_{T_X, g_t}} 
    \dot{g}^{\ast}_t 
\\
\\
&+& \tmop{Alt} \left[ \left( \nabla_{_{T_X, g_t}} 
    \dot{g}^{\ast}_t \right)_{g_t}^T \tmop{Ric}^{\ast}_{g_t} (\Omega) \right]\;.
  \end{eqnarray*}
Moreover using lemma \ref{OmTX-Lap-RmLap} we deduce the identities
  \begin{eqnarray*}
    \nabla_{_{T_X, g_t}} \nabla^{\ast_{_{\Omega}}}_{g_t} \nabla_{_{T_X, g_t}} 
    \dot{g}^{\ast}_t & = & \nabla_{_{T_X, g_t}}
    \Delta^{^{_{_{\Omega}}}}_{_{T_X, g_t}}  \dot{g}^{\ast}_t \;\,-\;\,
    \nabla^2_{_{T_X, g_t}} \nabla^{\ast_{_{\Omega}}}_{g_t}  \dot{g}^{\ast}_t\\
    &  & \\
    & = & \nabla_{_{T_X, g_t}}  \Big[ \Delta^{^{_{_{\Omega}}}}_{g_t} 
    \dot{g}^{\ast}_t \;\,-\;\,\mathcal{R}_{g_t} \ast \dot{g}^{\ast}_t \;\,+\;\,
    \dot{g}_t^{\ast} \tmop{Ric}^{\ast}_{g_t} (\Omega) \Big]
\\
\\
&-&\mathcal{R}_{g_t} \nabla^{\ast_{_{\Omega}}}_{g_t}  \dot{g}^{\ast}_t \;.
  \end{eqnarray*}
  We pick now an arbitrary space time point $(x_0, t_0)$
  and let $(e_k)_k$ and $\xi, \eta$ as in the proof of lemma \ref{Endo-Div} with respect to $g_{t_0}$.
  We expand at the space time point $(x_0, t_0)$ the therm
  \begin{eqnarray*}
    \left[ \nabla_{_{T_X, g_t}} (\mathcal{R}_{g_t} \ast \dot{g}^{\ast}_t)
    \right] (\xi, \eta) & = & \nabla_{g_t, \xi}  \Big[ (\mathcal{R}_{g_t}
    \ast \dot{g}^{\ast}_t) \,\eta \Big] \;\,-\;\, \nabla_{g_t, \eta}  \Big[
    (\mathcal{R}_{g_t} \ast \dot{g}^{\ast}_t) \,\xi \Big]\\
    &  & \\
    & = & \nabla_{g_t, \xi}\, \mathcal{R}_{g_t} (\eta, e_k)\, \dot{g}^{\ast}_t
    e_k \;\,+\;\,\mathcal{R}_{g_t} (\eta, e_k) \nabla_{g_t, \xi} \, \dot{g}^{\ast}_t
    e_k\\
    &  & \\
    & - & \nabla_{g_t, \eta} \,\mathcal{R}_{g_t} (\xi, e_k) \,\dot{g}^{\ast}_t
    e_k \;\,-\;\,\mathcal{R}_{g_t} (\xi, e_k) \nabla_{g_t, \eta} \, \dot{g}^{\ast}_t e_k\;
    .
  \end{eqnarray*}
  Thus using the differential Bianchi identity we infer the equalities
  \begin{eqnarray*}
    \nabla_{_{T_X, g_t}} (\mathcal{R}_{g_t} \ast \dot{g}^{\ast}_t) & = & -\;\,
    \nabla_{g_t, e_k} \mathcal{R}_{g_t}  \,\dot{g}^{\ast}_t e_k \;\,+\;\, \tmop{Alt}
    \left[ \mathcal{R}_{g_t} \ast \left( \nabla_{_{T_X, g_t}} - \nabla_{g_t}
    \right)  \dot{g}^{\ast}_t \right] \\
    &  & \\
    & = & - \;\,\nabla_{g_t, e_k} \mathcal{R}_{g_t}\,  \dot{g}^{\ast}_t e_k \;\,+\;\,
    \tmop{Alt} \left[ \mathcal{R}_{g_t} \circledast \left( \nabla_{_{T_X, g_t}} -
    \nabla_{g_t} \right)  \dot{g}^{\ast}_t \right] 
\\
\\
&+& \nabla_{g_t} 
    \dot{g}^{\ast}_t \ast \mathcal{R}_{g_t}\;,
  \end{eqnarray*}
  by the identities (\ref{curv-alg-id}) and (\ref{ext-alg-prod}). We infer the
  expression
  \begin{eqnarray*}
    \nabla_{_{T_X, g_t}} \nabla^{\ast_{_{\Omega}}}_{g_t} \nabla_{_{T_X, g_t}} 
    \dot{g}^{\ast}_t & = & \nabla_{_{T_X, g_t}}  \Big[
    \Delta^{^{_{_{\Omega}}}}_{g_t}  \dot{g}^{\ast}_t \;\,+\;\, \dot{g}_t^{\ast}
    \tmop{Ric}^{\ast}_{g_t} (\Omega) \Big] \\
    &  & \\
    & + & \tmop{Alt} \left[ \mathcal{R}_{g_t} \circledast \left( \nabla_{g_t} -
    \nabla_{_{T_X, g_t}} \right)  \dot{g}^{\ast}_t \right]\\
    &  & \\
    & + & \nabla_{g_t, e_k} \mathcal{R}_{g_t} \, \dot{g}^{\ast}_t e_k \;\,-\;\,
    \nabla_{g_t}  \dot{g}^{\ast}_t \ast \mathcal{R}_{g_t} \;\,-\;\,\mathcal{R}_{g_t}
\nabla^{\ast_{_{\Omega}}}_{g_t}  \dot{g}^{\ast}_t \;.
  \end{eqnarray*}
  This combined with the identity (\ref{div-id-curv}) and with the Bianchi
  type identity (\ref{Om-cntr-Bianc}) implies the expression
  \begin{eqnarray*}
    \nabla_{_{T_X, g_t}} \nabla^{\ast_{_{\Omega}}}_{g_t} \nabla_{_{T_X, g_t}} 
    \dot{g}^{\ast}_t & = & \nabla_{_{T_X, g_t}}  \Big[
    \Delta^{^{_{_{\Omega}}}}_{g_t}  \dot{g}^{\ast}_t \;\,+\;\, \dot{g}_t^{\ast}
    \tmop{Ric}^{\ast}_{g_t} (\Omega) \Big] \\
    &  & \\
    & + & \tmop{Alt} \left[ \mathcal{R}_{g_t} \circledast \left( \nabla_{g_t} -
    \nabla_{_{T_X, g_t}} \right)  \dot{g}^{\ast}_t \right]\\
    &  & \\
    & + & \underline{\tmop{div}}^{^{_{_{\Omega}}}}_{g_t}  \left[
    \mathcal{R}_{g_t}, \dot{g}^{\ast}_t \right] \;\,-\;\, \dot{g}_t^{\ast}
    \,\nabla_{_{T_X, g_t}} \tmop{Ric}^{\ast}_{g_t} (\Omega) \;.
  \end{eqnarray*}
  This combined with the previous variation formula for $\nabla_{_{T_X, g_t}}
  \tmop{Ric}^{\ast}_{g_t} (\Omega)$ implies the required conclusion.
\end{proof}

\subsection{An direct proof of the variation formula (\ref{col-vr-OmRc})}

We observe that (\ref{col-vr-OmRc}) it is equivalent to the variation formula
(\ref{col-vrEnOmRc}) that we show now. Let $(e_k)_k$ be a local tangent frame. Using
the identity (\ref{col-vr-Rm}) we compute the variation
\begin{eqnarray*}
  2\, \frac{d}{d t} \tmop{Ric}^{\ast}_{g_t} \xi & = & 2 \,\frac{d}{d t}\,
  \mathcal{R}_{g_t} (\xi, e_k) \,g^{- 1}_t e^{\ast}_k\\
  &  & \\
  & = & \left[ \mathcal{R}_{g_t} (\xi, e_k), \dot{g}^{\ast}_t \right] e_k \;\,-\;\,
  2\,\mathcal{R}_{g_t} (\xi, e_k) \,\dot{g}^{\ast}_t e_k \;.
\end{eqnarray*}
We deduce the variation formula
\begin{equation}
  \label{vr-End-Ric} 2 \,\frac{d}{d t} \tmop{Ric}^{\ast}_{g_t} \;\;=\;\; -\;\,
  \dot{g}^{\ast}_t \tmop{Ric}^{\ast}_{g_t} \;\,-\;\,\mathcal{R}_{g_t} \ast
  \dot{g}^{\ast}_t \;.
\end{equation}
On the other hand using the identity (\ref{col-vr-Grad}) we can compute the
variation of the Hessian
\begin{eqnarray*}
  2\, \frac{d}{d t} \,\nabla^2_{g_t} f_t \,\xi & = & 2\, \dot{\nabla}_{g_t, \xi}
  \nabla_{g_t} f_t \;\,-\;\, \nabla_{g_t, \xi}  \left( \nabla_{g_t}^{\ast} 
  \dot{g}^{\ast}_t \;\,+\;\, 2\, \dot{g}^{\ast}_t  \,\mathcal{\nabla}_{g_t} f_t \right)\\
  &  & \\
  & = & \nabla_{g_t, \xi} \, \dot{g}^{\ast}_t \,\nabla_{g_t} f_t \;\,-\;\, \nabla_{g_t,
  \xi} \nabla^{\ast}_{g_t}  \dot{g}^{\ast}_t 
\\
\\
&-& 2\, \nabla_{g_t, \xi} \,
  \dot{g}^{\ast}_t \,\nabla_{g_t} f_t \;\,-\;\, 2\, \dot{g}^{\ast}_t \,
  \mathcal{\nabla}^2_{g_t} f_t \,\xi
\end{eqnarray*}
\begin{eqnarray*}
  & = & - \;\,\nabla_{g_t, \xi} \, \dot{g}^{\ast}_t \,\nabla_{g_t} f_t \;\,-\;\, \Delta_{g_t}
  \dot{g}^{\ast}_t\, \xi 
\\
\\
&+& (\mathcal{R}_{g_t} \ast \dot{g}^{\ast}_t) \,\xi \;\,-\;\,
  \dot{g}^{\ast}_t \tmop{Ric}^{\ast}_{g_t}\xi \;\,-\;\, 2\, \dot{g}^{\ast}_t \,
  \mathcal{\nabla}^2_{g_t} f_t \,\xi\;,
\end{eqnarray*}
thanks to lemma \ref{TX-Lap-RmLap}. We infer the variation identity
\begin{eqnarray*}
  2 \,\frac{d}{d t}\, \nabla^2_{g_t} f_t \;\; = \;\; -\;\, \Delta^{^{_{_{\Omega}}}}_{g_t} 
  \dot{g}^{\ast}_t \;\,+\;\,\mathcal{R}_{g_t} \ast \dot{g}^{\ast}_t \;\,-\;\, \dot{g}^{\ast}_t
  \tmop{Ric}^{\ast}_{g_t} \;\,-\;\, 2\, \dot{g}^{\ast}_t \, \mathcal{\nabla}^2_{g_t} f_t\; .
\end{eqnarray*}
This combined with the variation formula (\ref{vr-End-Ric}) implies the
formula (\ref{col-vrEnOmRc}) and thus the variation identity
(\ref{col-vr-OmRc}).

\vspace{1cm}
\noindent
Nefton Pali
\\
Universit\'{e} Paris Sud, D\'epartement de Math\'ematiques 
\\
B\^{a}timent 425 F91405 Orsay, France
\\
E-mail: \textit{nefton.pali@math.u-psud.fr}


\begin{thebibliography}{000000}
\bibitem[Bes]{Bes} \textsc{Besse, A.L.,} 
\emph{Einstein Manifolds}, Springer-Verlag, 2007.
\bibitem[Ch-Kn]{Ch-Kn} \textsc{Chow, B., Knopf, D.,} 
\emph{The Ricci flow: An introduction}, AMS, Providence, RI, 2004.
\bibitem[Fri]{Fri} \textsc{Friedon, D.H.,} 
\emph{Nonlinear models in $2+\varepsilon$ dimensions}, Ann. Physics 163 (1985), no. 2, 318-419.
\bibitem[Pal1]{Pal1} \textsc{Pali, N.,} 
\emph{The total second variation of Perelman's $\mathcal{W}$-functional}, arXiv:math.
\bibitem[Pal2]{Pal2} \textsc{Pali, N.,} 
\emph{The Soliton-K\"ahler-Ricci Flow over Fano Manifolds}, arXiv:math
\bibitem[Per]{Per} \textsc{Perelman, G.,} 
\emph{The entropy formula for the Ricci flow 
and its geometric applications}, arXiv:math/0211159.
\bibitem[Ham]{Ham} \textsc{Hamilton, R.S,} 
\emph{Three-manifolds with positive Ricci curvature}, J. differential Geom. 17 (1982), no. 2, 255-306.
\end{thebibliography}
\end{document}